\documentclass[11pt]{amsart}
\usepackage{enumerate}
\usepackage[T1]{fontenc}
\usepackage{lmodern}
\usepackage{geometry}
\usepackage{mathrsfs}
  \usepackage{graphicx}
 \usepackage{tikz}
 \usetikzlibrary{arrows.meta}
 \usepackage{pgfplots}
 \usepackage{xcolor}
\usetikzlibrary{positioning}
 \usepgfplotslibrary{fillbetween}
 \usetikzlibrary{arrows.meta,calc,positioning}
 \pgfplotsset{compat=1.18}

\usepackage{graphicx}
\usepackage{amssymb}
\usepackage{epstopdf}
\usepackage{geometry} 
\usepackage{graphicx}
\usepackage{amsmath} 
\numberwithin{equation}{section}
\usepackage{commath}
\usepackage[numbered,framed]{matlab-prettifier}
\usepackage{changepage}
\usepackage[normalem]{ulem}
\usepackage{bbm}
\usepackage{subfig}
\usepackage{float}
\usepackage{url}
\usepackage{enumitem}
\usepackage{listings}
\usepackage{mathtools}
\usepackage{amsthm}
\usepackage{aliascnt}
\usepackage[disable]{todonotes} 
\usepackage{hyperref}
\usepackage{cleveref}

\DeclareMathOperator{\interior}{int}

\newcommand{\ubar}{\bar u}
\newcommand{\vbar}{\bar v}
\newcommand{\hmean}{\bar h}
\newcommand{\Ebar}{\bar E}
\newcommand{\Qbar}{\bar Q}

\newcommand{\Qcal}{\mathcal{Q}}

\newcommand{\SoneL}{\mathbb{R}/L\mathbb{Z}}

\newcommand{\E}{\mathbb{E}}

\DeclareMathOperator{\rank}{rank}

\theoremstyle{plain}

\newtheorem{theorem}{Theorem}[section]

\newaliascnt{lemma}{theorem}
\newtheorem{lemma}[lemma]{Lemma}
\aliascntresetthe{lemma}

\newaliascnt{corollary}{theorem}
\newtheorem{corollary}[corollary]{Corollary}
\aliascntresetthe{corollary}

\newaliascnt{proposition}{theorem}
\newtheorem{proposition}[proposition]{Proposition}
\aliascntresetthe{proposition}
\newtheorem{openquestion}[theorem]{Open Question}
\newtheorem{conjecture}{Conjecture}

\newtheorem*{theorem*}{Theorem}

\theoremstyle{remark}
\newtheorem*{remark}{Remark}

\theoremstyle{definition}

\newaliascnt{definition}{theorem}
\newtheorem{definition}[definition]{Definition}
\aliascntresetthe{definition}

\crefname{theorem}{theorem}{theorems}
\Crefname{theorem}{Theorem}{Theorems}

\crefname{lemma}{lemma}{lemmas}
\Crefname{lemma}{Lemma}{Lemmas}

\crefname{corollary}{corollary}{corollaries}
\Crefname{corollary}{Corollary}{Corollaries}

\crefname{proposition}{proposition}{propositions}
\Crefname{proposition}{Proposition}{Propositions}

\crefname{definition}{definition}{definitions}
\Crefname{definition}{Definition}{Definitions}

\allowdisplaybreaks
\newcommand{\R}{\mathbb{R}}

\newcommand{\one}{\mathbf{1}}
\newcommand{\Lip}{\operatorname{Lip}}
\newcommand{\dist}{\operatorname{dist}}

\title[Convex Integration for 1-D Hyperbolic Problems]{The computational ansatz for convex integration of hyperbolic systems and a resolution of the Strong Trace Conjecture}
\author[Sam G. Krupa]{Sam G.  Krupa}
\address[Sam G. Krupa]{\newline Département de mathématiques et applications \newline École normale supérieure, Université PSL, CNRS\newline 45 rue d'Ulm - F 75230 PARIS cedex 05 \newline France}
\email{sam.krupa@ens.fr}

\thanks{The author utilized ChatGPT 5.6 Sol Pro, ChatGPT 5.6 Sol Ultra, and ChatGPT 6 Astra Pro. The author checked and assumes responsibility for the final content. The work of the  author is funded by the European Union through the project “Quantitative Stability and Regularity of Large
Data for Conservation Laws.” Views and opinions expressed are however those of the author(s) only and do not necessarily reflect those of the European Union or European Research Executive Agency (REA). Neither the European Union nor the granting authority can be held responsible for them. Part of this work was completed during the scientific trimester ``Mathematical Developments in Geophysical Fluid Dynamics'', which was held at Institut Henri Poincaré from April 13 to July 10, 2026. The author acknowledges support of the Institut Henri Poincaré (UAR 839 CNRS-Sorbonne Université), and LabEx CARMIN (ANR-10-LABX-59-01). Part of this work was also completed during the workshop ``Fluid Dynamics, Singularities, and AI-Driven Discovery'' held at the Speinshart Center for AI and SuperTech (Speinshart, Germany), a.k.a. \emph{HighTechAbbey}, from July 27 to 31, 2026.}
\date{\today}                                           
\begin{document}
\begin{abstract}
In this paper, we consider $2\times2$ hyperbolic systems of conservation laws in one spatial dimension. We use the computational ansatz introduced in the hyperbolic theory by the author and Székelyhidi to study the constitutive set corresponding to the PDE. The rank-one convex geometry of this set relates to non-uniqueness and the existence of low-regularity solutions. Through a computer-assisted search, we find a pressure law $p$ such that the $p$-system with this pressure law, and its natural strictly convex entropy, verifies all of the conditions necessary for the large data $L^2$ stability and the technique of ``$a$-contraction with shifts'' and thus we have uniqueness of certain Riemann solutions in the class of solutions verifying the Strong Trace Property. At the same time, the constitutive set contains a $T_6$ configuration and we use it to construct non-unique solutions without the Strong Trace Property. This resolves the question of sharpness of strong traces. We also present proofs which show \emph{nonexistence} of $T_\infty$ structures for all genuinely nonlinear systems and \emph{nonexistence} of $T_N$ structures for all $N$ for the $p$-system with $p''>0$, thus blocking these routes towards convex integration.
\end{abstract}
\maketitle
\tableofcontents

\section{Introduction}

We consider one-dimensional systems of two conservation laws of the form
\begin{equation}\label{eq:system}
\begin{cases}
   U_t + F(U)_x = 0, \qquad (x,t)\in \mathbb{R}\times (0,\infty), \\
   U(x,0)=U^0(x)
\end{cases}
,
\end{equation}
where $U^0=U^0(x)\in\mathcal{V}\subset\mathbb{R}^2$ is the initial data,

$$
    U=U(x,t)\in \mathcal{V}\subset \mathbb{R}^2
$$

is the unknown state and

$$
    F\colon \mathcal{V}\to \mathbb{R}^2
$$

is a sufficiently smooth flux. Throughout, we assume that the state space
\(\mathcal{V}\) is a nonempty open and convex subset of \(\mathbb{R}^2\).

The system \eqref{eq:system} is assumed to be \emph{strictly hyperbolic} on
\(\mathcal{V}\). More precisely, for every \(U\in\mathcal{V}\), the Jacobian
matrix

$$
    DF(U)
$$

has two distinct real eigenvalues

$$
    \lambda_1(U)<\lambda_2(U).
$$

We denote by \(r_1(U),r_2(U)\in\mathbb{R}^2\) corresponding right
eigenvectors, so that
\begin{equation}\label{eq:eigenvectors}
DF(U)r_k(U)=\lambda_k(U)r_k(U),
\qquad k=1,2.
\end{equation}
Thus the system possesses two distinct characteristic families throughout
the state space.

We further assume that both characteristic families are
\emph{genuinely nonlinear}. Namely,
\begin{equation}\label{eq:genuine-nonlinearity}
\nabla\lambda_k(U)\cdot r_k(U)\neq 0,
\qquad U\in\mathcal{V},
\qquad k=1,2.
\end{equation}
This condition expresses the nondegeneracy of the variation of each
characteristic speed along its corresponding characteristic direction.
In particular, it is the structural condition underlying the formation
of the classical shock and rarefaction waves associated with each
characteristic family.

Since solutions of \eqref{eq:system} naturally develop discontinuities
even from smooth initial data, it is necessary to work with weak
solutions. A locally integrable function

$$
    U\colon \mathbb{R}\times[0,\infty)\to\mathcal{V}
$$

is a weak solution of \eqref{eq:system} with initial data
\(U^0\colon\mathbb{R}\to\mathcal{V}\) if
\begin{equation}\label{eq:weak-formulation}
\int_0^\infty\int_{\mathbb{R}}
\left(
U\cdot\varphi_t
+
F(U)\cdot\varphi_x
\right)
\,dx\,dt
+
\int_{\mathbb{R}} U^0(x)\cdot\varphi(x,0)\,dx
=
0
\end{equation}
for every test function
\(\varphi\in C_c^\infty(\mathbb{R}\times[0,\infty);\mathbb{R}^2)\).
As is well known, weak solutions are generally not unique, and an
additional admissibility criterion is required in order to select the
physically relevant solution.

Thus, we also assume that the system admits a strictly convex entropy. More
precisely, let

$$
    \eta\colon\mathcal{V}\to\mathbb{R}
$$

be a sufficiently smooth function such that
\begin{equation}\label{eq:strict-convexity-entropy}
D^2\eta(U)>0,
\qquad U\in\mathcal{V},
\end{equation}
in the sense of positive definite matrices. Associated with \(\eta\) is
an entropy flux

$$
    q\colon\mathcal{V}\to\mathbb{R}
$$

satisfying the compatibility relation
\begin{equation}\label{eq:entropy-compatibility}
\nabla q(U)
=
DF(U)^{\mathsf T}\nabla\eta(U),
\qquad U\in\mathcal{V}.
\end{equation}
\textbf{Notation:} Throughout, vectors and scalar gradients are written as column vectors. For a
smooth map \(H\), the notation \(DH(U)\) denotes its differential (Jacobian)
at \(U\), and \(DH(U)[V]\) denotes the action of that differential on a
direction \(V\). When the differential is represented by its Jacobian matrix,
we also use the ordinary matrix--vector product \(DH(U)V\). Thus \(DF(U)\) is
the flux Jacobian, whereas \(\nabla\eta(U)\) and \(\nabla q(U)\) are column
gradients.
\vspace{.09in}

The pair \((\eta,q)\) is referred to as an
\emph{entropy--entropy-flux pair} for \eqref{eq:system}.

For every smooth solution \(U\) of \eqref{eq:system}, the compatibility
condition \eqref{eq:entropy-compatibility} yields the additional
conservation law
\begin{equation}\label{eq:entropy-conservation}
\eta(U)_t+q(U)_x=0.
\end{equation}
For discontinuous weak solutions, the physically relevant admissibility
condition is instead expressed through the entropy inequality
\begin{equation}\label{eq:entropy-inequality}
\eta(U)_t+q(U)_x\leq 0
\end{equation}
in the sense of distributions. Equivalently, for every nonnegative test
function
\(\phi\in C_c^\infty(\mathbb{R}\times[0,\infty))\),
\begin{equation}
\int_0^\infty\int_{\mathbb{R}}
\left(
\eta(U)\phi_t
+
q(U)\phi_x
\right)
\,dx\,dt + \int_\mathbb{R}\eta(U(x,0))\phi(x,0)\,dx
\geq 0.
\end{equation}
A weak solution satisfying \eqref{eq:entropy-inequality} will be called
an entropy solution with respect to the entropy pair \((\eta,q)\). Systems of conservation laws with more than two conserved quantities will admit at most one entropy, entropy-flux pair, including important physical systems like the non-isentropic Euler equations (for integrability conditions on the entropy, see \cite[p.~13-14]{dafermos_big_book} and \cite[p.~54-55]{dafermos_big_book}). Systems with two conserved quantities have large families of entropies, but we restrict ourselves in this work to only one entropy as a model for systems with additional conserved quantities.

The strict convexity of \(\eta\), together
with strict hyperbolicity and genuine nonlinearity, provides the basic
structural framework in which we study the geometry and dynamics of the
system \eqref{eq:system}.

\subsection{$L^1$ and $L^2$ theories}

In the $L^1$ theory, where distances between solutions are measured in $L^1$, for the most part solutions with finite but possibly large total variation can be studied (see \cite{ lewicka2004well,lewicka2001stability,MR1375345}). One work we wish to highlight is the recent $L^1$-based study of some of the Glimm-Lax solutions with infinite total variation \cite{2025arXiv250500420B}.

On the other hand, in the $L^2$-based theory, where distances between solutions are measured in $L^2$, solutions with infinite total variation can be studied \cite{MR3519973,serre_vasseur,2017arXiv170905610K,MR4487515,MR4667839,CHENG2025113599,Leger2011_original,CKV2}. This includes recent work which shows that in fact  bounded variation solutions are unique and \emph{quantitatively} stable in a large class of solutions with possibly infinite total variation \cite{ChenFaileKrupa2025Holder,2025arXiv250916432C}. As a result of this work, we now know that all Glimm-Lax solutions with initial data in $H^s$ for any $s>0$ are in fact unique \cite{ChengFaileKrupa2026}. The method of proof in the $L^2$ theory is ``$a$-contraction with shifts.'' In a nutshell, the idea is that shocks are difficult to measure in $L^2$, so we relax the notion of stability in $L^2$ by allowing discontinuities to move with artificial velocities (``shifts''). To ensure stability, and often even contraction, in $L^2$, it is provably necessary to introduce a weight on the spatial domain given by a positive function $(x,t)\mapsto a(x,t)$ (hence ``$a$-contraction''). We refer to \cite{MR3519973} for more details.

\subsection{Strong Trace Property}
The $L^2$ theory requires a regularity assumption on solutions that is usually referred to in the literature as the \emph{Strong Trace Property}, which we define below. This is the \emph{key boundary} between solutions we can show uniqueness for and solutions for which uniqueness is more uncertain. The Strong Trace Property essentially eliminates extreme oscillations such as $\sin(1/x)$. All solutions with finite total variation, or solutions from the Glimm-Lax class \cite{MR0265767} are known to verify the Strong Trace Property. However, in general it is not known if solutions to conservation laws must satisfy the Strong Trace Property. In particular, it is currently unknown if  solutions originating from the compensated compactness method verify the Strong Trace Property (however, see some related works on this question \cite{ancona2025liouville,talamini2026strong,Golding2023}). 

The Strong Trace Property is needed for the construction and analysis of the shift functions in the ``$a$-contraction with shifts'' method. Without the Strong Trace Property, it is unknown how to properly introduce artificial velocities at shocks in order to maintain some form of stability.
\subsection{Strong Trace Conjecture}
For 15 years, it has been a completely open question if the Strong Trace Property is truly an essential condition in the $L^2$ theory. Over 15 years, the condition has not been able to be relaxed.  To summarize the situation, we formulate the ``Strong Trace Conjecture'' 

\begin{conjecture}[Strong Trace Conjecture]There is a strictly hyperbolic system of conservation laws in one spatial dimension, endowed with a strictly convex entropy $\eta$ and associated entropy-flux $q$ such that for some initial data $U^0$:
\begin{itemize}
\item there is at least one solution $\bar{U}$ verifying the Strong Trace Property and this solution is unique in the class of solutions which are in $L^\infty$ and have the Strong Trace Property,
\end{itemize}
\textbf{and}
\begin{itemize}
\item this solution $\bar{U}$ is non-unique in the class of solutions which are in $L^\infty$ but may not verify the Strong Trace Property.
\end{itemize}
Thus, it is \emph{precisely} the Strong Trace Property which is responsible for the uniqueness.
\end{conjecture}

In fact, we will prove this conjecture is true. We will consider a particular case of \eqref{eq:system}, namely the $p$-system (see below, \eqref{eq:p-system}).

\subsection{$p$-system}
We will mostly focus on the $p$-system, which is the most famous example of a system of conservation laws. It is isentropic gas dynamics in the Lagrangian coordinate.

Let $I\subset (0,\infty)$ be an open interval, and let
$p\in C^3(I)$ satisfy
\begin{equation}
    p'(v)<0,
    \qquad
    p''(v)>0,
    \qquad v\in I.
    \label{eq:p-system-pressure-assumptions}
\end{equation}
We consider the one-dimensional $p$-system in Lagrangian coordinates,
\begin{equation}
    \begin{cases}
        \partial_t v-\partial_x u=0,\\[1mm]
        \partial_t u+\partial_x p(v)=0,
    \end{cases}
    \label{eq:p-system}
\end{equation}
where $v$ is the specific volume and $u$ is the velocity.  In vector
form, with $U=(v,u)^{\mathsf T}$,
\begin{equation}
    \partial_t U+\partial_x F(U)=0,
    \qquad
    F(U)=\begin{pmatrix}-u\\ p(v)\end{pmatrix}.
    \label{eq:p-system-vector-form}
\end{equation}
The characteristic speeds are
\begin{equation}
    \lambda_1(U)=-c(v),
    \qquad
    \lambda_2(U)=c(v),
    \qquad
    c(v):=\sqrt{-p'(v)}.
    \label{eq:p-system-characteristic-speeds}
\end{equation}

Fix an arbitrary reference point $v_\ast\in I$ and define
\begin{equation}
    \Pi(v):=-\int_{v_\ast}^{v}p(s)\,ds.
\end{equation}
The natural entropy for \eqref{eq:p-system} is
\begin{equation}
    \eta(U):=\frac{u^2}{2}+\Pi(v)
    =\frac{u^2}{2}-\int_{v_\ast}^{v}p(s)\,ds.
    \label{eq:p-system-natural-entropy}
\end{equation}
Its Hessian is
\begin{equation}
    D^2\eta(U)
    =\begin{pmatrix}
        -p'(v)&0\\
        0&1
    \end{pmatrix},
\end{equation}
so $\eta$ is strictly convex under the assumption $p'<0$.  The
corresponding entropy flux is
\begin{equation}
    q(U)=u\,p(v).
\end{equation}

For the convention in which $U$ is the left state, the Lax-admissible
$k$-shock curve issuing from $U$ is denoted by $S_U^k$, and it is
parameterized by the specific-volume coordinate $v$.

Both curves are extended to their base point by
\begin{equation}
    S_U^1(v_U)=S_U^2(v_U)=U.
\end{equation}
The corresponding shock speeds are denoted by $\sigma_U^k(v)$. We refer to \Cref{Hugoniot_facts} for more details.

\subsection{Relative entropy}
Given our entropy, entropy-flux pair $\eta$, $q$ and states
$(U,V)\in\mathcal{V}\times\mathcal{V}$ we define the relative quantities
\begin{equation}
\label{eq:relative-quantities}
\begin{aligned}
\eta(U\mid V)
&:=
\eta(U)-\eta(V)-\nabla\eta(V)\cdot(U-V),\\
q(U;V)
&:=
q(U)-q(V)
-\nabla\eta(V)\cdot\bigl(F(U)-F(V)\bigr).
\end{aligned}
\end{equation}
For fixed $V$ these quantities $(\eta(\cdot\mid V),q(\cdot;V))$ then constitute an entropy, entropy-flux pair
for the system \eqref{eq:system}. We recall the following Lemma.

\begin{lemma}[Lemma 1 of \cite{Leger2011}]\label{l2control}
Assume $\mathcal{V}$ is bounded and its closure does not include vacuum.  There exists a constant
$C^*>0$ such that for all $(U,V)\in\mathcal{V}\times \mathcal{V}$ we have
\begin{equation}
\frac{1}{C^*}|U-V|^2
\leq
\eta(U\mid V)
\leq
C^*|U-V|^2.
\label{eq:relative-entropy-L2-equivalence}
\end{equation}
\end{lemma}

This Lemma states that controlling integrals of our relative entropy
is equivalent to controlling the $L^2$ difference of $U$ and $V$.

\subsection{Main theorem}
We now introduce the precise definition of the Strong Trace Property:

\begin{definition}[Strong Trace Property \cite{Leger2011}]\label{strong_trace_def}
Let $U\in L^\infty(\mathbb{R}\times[0,\infty))$. We say that $U$
verifies the \emph{Strong Trace Property} if for any Lipschitz
continuous curve $h:[0,\infty)\to\mathbb{R}$, there exist two bounded
functions $U_-,U_+\in L^\infty([0,\infty))$ such that, for any $T>0$,
\[
\begin{aligned}
\lim_{n\to\infty}
\int_0^T
\sup_{y\in(0,1/n)}
\left|U(h(t)+y,t)-U_+(t)\right|\,dt
&=0,\\
\lim_{n\to\infty}
\int_0^T
\sup_{y\in(-1/n,0)}
\left|U(h(t)+y,t)-U_-(t)\right|\,dt
&=0.
\end{aligned}
\]
\end{definition}

For convenience, we will use the notation
\[
U_+(t):=U(h(t)+,t),
\qquad
U_-(t):=U(h(t)-,t).
\]

With this definition in hand, we can define the class of solutions which are typically considered in the $L^2$ theory:

\begin{equation}
\mathcal{S}_{\mathrm{weak}}
:=
\left\{
U\in L^\infty(\mathbb{R}\times[0,T);\mathcal{V})
:
\begin{array}{l}
U \text{ weak solution to \eqref{eq:system}, as well as \eqref{eq:entropy-inequality},}\\
\text{verifying \Cref{strong_trace_def}}
\end{array}
\right\},
\label{eq:S-weak}
\end{equation}
for any sufficiently large but bounded state space $\mathcal{V}$ whose closure does not contain vacuum and for some $T>0$.

This is a large class with no assumption on the total variation at the initial time $t=0$ or later times $t>0$. The total variation may in fact be infinite, and examples of such solutions in $\mathcal{S}_{\mathrm{weak}}$ are given by the Glimm-Lax theory \cite{MR0265767}. 

We can now state precisely our main result:

\begin{theorem}[Main theorem -- uniqueness with the Strong Trace Property, non-uniqueness otherwise]\label{main_theorem}

Fix $R,T>0$.
Then there exists a smooth, strictly decreasing function $p\colon (0,\infty)\to\mathbb{R}$, $U_L\in\mathcal{V}$, and $U_R\in\mathcal{V}$, such that for the $p$-system \eqref{eq:p-system} with the pressure law $p$, and Riemann problem initial data
\begin{equation}\label{main_theorem_Riemann}
   U^0_{\text{Riemann}}=\begin{cases*}
                    U_L & if  $x < 0$  \\
                     U_R & if $x\ge 0$,
                 \end{cases*}
\end{equation}  
we have the following dichotomy:

\begin{itemize}
    \item  There are infinitely (!) many distinct weak solutions with initial data $U^0_{\text{Riemann}}$, and each of these non-unique solutions takes values in $\mathcal{V}$ and verifies the entropy inequality \eqref{eq:entropy-inequality} for the natural entropy, entropy-flux pair.

\item At the same time,
the classical self-similar solution with initial data $U^0_{\text{Riemann}}$, which we denote $U_{\text{self-similar}}$, is unique and stable in the class $\mathcal{S}_{\mathrm{weak}}$ and moreover verifies the following stability estimate against all weak solutions $U$ in  $\mathcal{S}_{\mathrm{weak}}$:

\begin{align}\label{eq:main-stability-estimate}
     \boxed{
    \bigl\|U(\cdot,\tau)-U_{\text{self-similar}}(\cdot,\tau)\bigr\|_{L^2((-R,R))}
    \leq
    K
    \bigl\|U(\cdot,0)-U_{\text{self-similar}}(\cdot,0)\bigr\|_{L^2((-R-s\tau,R+s\tau))}^{1/2}
    },
\end{align}
for all $\tau\in[0,T]$ and for constants $K,s>0$ (where $s$ is the speed of information).
\end{itemize}

\end{theorem}

\subsection{Important consequences of the main theorem}

At a high level, our method of proof is convex integration, which is a way of constructing highly oscillatory and non-unique solutions to PDEs.

Some remarks on our result: 

\subsubsection{Loss of predictive power} Our result shows that hyperbolic systems in one spatial dimension, including the Euler equations, lose their predictive power below the Strong Trace Property regularity-threshold. In other words, given fixed initial data, there is no guarantee of a unique solution even if we restrict to the smaller class of entropy-admissible solutions \eqref{eq:entropy-inequality}. This gives the analogue of the result \cite{MR3352460}, where non-uniqueness of low-regularity Riemann solutions is shown for the isentropic Euler equations in two spatial dimensions. Similar to our work, the work \cite{MR3352460} considers the pressure law as an extra degree of freedom in their search for convex integration solutions.

Crucially, our result is \emph{significantly sharper} than the result \cite{MR3352460}: while the Chiodaroli result shows \emph{non-uniqueness} in a class of very low-regularity solutions, they give a \emph{uniqueness} criteria which requires both self-similarity and BV-regularity. On the other hand, our uniqueness result depends \emph{only on the Strong Trace Property}, which is a regularity hypothesis substantially weaker than BV (or $\text{BV}_{\text{loc}}$). Notice for example that the function $x\mapsto\sqrt{x}\sin(1/x)$ is not in BV but verifies the Strong Trace Property.

\subsubsection{1-D convex integration} Moreover, we wish to point out that our work is the first work to be able to convex integrate a one dimensional hyperbolic system with a strictly convex entropy. This is in comparison with the works \cite{2024arXiv240702927C} and \cite{Krupa2026NonUniqueness}, where a strictly convex entropy condition is omitted.

In particular, by using the change of variables \cite{wagnergasdynamics}, which holds also for weak solutions and also transforms the entropy inequality \eqref{eq:entropy-inequality}, we have constructed \emph{entropy-admissible one-dimensional convex integration solutions for the 2-D compressible isentropic Euler equations}.

\subsubsection{Infinite-BV perturbations of the Riemann problem}
We can construct small $L^\infty$ (not BV !) perturbations of the Riemann data considered in \Cref{main_theorem}. We can use the Glimm-Lax existence theory of solutions with small $L^\infty$ data \cite{MR0265767} to make a perturbation away from the jump in the Riemann data, and then for some small but positive time the Glimm-Lax solution does not interact with the discontinuity in the Riemann solution. Remark that our pressure law is strictly convex in a neighborhood of the Riemann states so the Glimm-Lax theory applies.  The Glimm-Lax solutions also verify the Strong Trace Property because they instantaneously enter $\text{BV}_{\text{loc}}$. This  shows the existence of a non-BV solution, which is stable against the Riemann problem (because it has the strong traces). At the same time, the convex integration solutions exist and are non-unique.

\subsection{Additional remarks}
The crucial observation made in
\cite{MR2008346} is that the system \eqref{eq:system} with equality in \eqref{eq:entropy-inequality} for a given
entropy/entropy-flux pair $(\eta,q)$ can be equivalently restated
as a first order differential inclusion of the type
\[
D\psi(x,t)\in K_{F,\eta,q}
\qquad \text{a.e. }(x,t),
\]
where $K_{F,\eta,q}$ is a given two dimensional surface
embedded in the space of matrices $\mathbb{R}^{3\times 2}$ (see below
in \Cref{solutions_by_convex_integration_section} and \eqref{manifold_K}). 

The rank-one convex geometry of the set $K_{F,\eta,q}$ is closely related to the compactness of the system \eqref{eq:system} as well as the uniqueness and regularity of solutions to the system \eqref{eq:system}. This is explained in detail in the works \cite{MR2008346,MR3995052,2019arXiv190905938L}. In particular, we can begin to understand the rank-one convex geometry of $K_{F,\eta,q}$ by trying to find or eliminate certain special 2-point
configurations (rank-one connections) as well as special $N$-point
configurations, $N\geq 4$ ($T_N$-configurations)---see below in
\Cref{T_N_section} and \Cref{def:ordered-TN}. The existence of such configurations may lead to non-unique and low regularity solutions.

In this work, we continue our computational study of the set $K_{F,\eta,q}$. In the previous works \cite{MR3995052,KrupaSzekelyhidi2024T4,doi:10.1142/S021919972250081X}, the rank-one convex geometry of $K_{F,\eta,q}$ was constrained: under certain conditions on $F$, $\eta$, and $q$, the possibility of rank-one connections or $T_4$ or $T_5$ configurations in the set $K_{F,\eta,q}$ was eliminated.

On the other hand, we now use both a computational search of solutions to the equations defining a $T_N$ configuration (see \Cref{matlab_lemma}) as well as a careful computer-assisted study of the manifold of $T_N$ configurations to find a delicate situation where for a particular pressure law $p$, the set $K_{F,\eta,q}$ corresponding to the $p$-system with its natural entropy and entropy-flux pair admits a family of $T_6$ configurations. Moreover, even though the system is not genuinely nonlinear, in other words $p''$ may change sign, we can find particular Riemann-type initial data for which the theory of ``$a$-contraction with shifts'' will apply.

Our result uses computer-assisted proof and in particular the computational ansatz introduced to the hyperbolic theory by the author \cite{Krupa2026NonUniqueness} and later used by the author and Székelyhidi to resolve the question of uniqueness for the compressible vortex sheet \cite{KrupaSzekelyhidi2025Contact}. See \Cref{T_N_existence_section}.

\subsection{Further questions on rank-one geometry}

Our work involves questions about the global rank-one convex geometry of the constitutive set $K_{F,\eta,q}$. The focus of this work is what are the global, or local properties of the conservation law \eqref{eq:system} which allow for nontrivial rank-one geometries. This has been a focus of many earlier works \cite{MR2008346,doi:10.1142/S021919972250081X,2019arXiv190905938L,KrupaSzekelyhidi2024T4}. In this work, we discover conditions which allow for $T_N$ configurations to live in $K_{F,\eta,q}$ in the case of the $p$-system. Our pressure law is not genuinely nonlinear, however we still meet all of the requirements for the $a$-contraction theory. We complement our results with two proofs, completed mostly by ChatGPT 5.6 Sol Pro and ChatGPT 5.6 Sol Ultra, which show (a) the non-existence of $T_N$ configurations for all $N$ for the $p$-system when $p''>0$ (ensuring genuine nonlinearity) (see \Cref{sec:no-TN-p-system}) and also (b) the nonexistence of so-called $T_\infty$ configurations whenever the flux is genuinely nonlinear -- this result applies to all hyperbolic systems and not only the $p$-system. See \Cref{sec:no-Tinfty}.

The theorems are:

\begin{theorem}[No $T_N$ for genuinely nonlinear $p$-system (\Cref{sec:no-TN-p-system}]
The constitutive set for the $p$-system \eqref{eq:p-system} in the strictly hyperbolic, genuinely nonlinear setting, does not contain a $T_N$ configuration for any $N$.
\end{theorem}

\begin{theorem}[No $T_\infty$ for any genuinely nonlinear system (\Cref{sec:no-Tinfty})]
The constitutive set for the general hyperbolic system \eqref{eq:system} in the strictly hyperbolic, genuinely nonlinear setting, does not contain a $T_\infty$ configuration.
\end{theorem}

Loosely speaking $T_\infty$ configurations are the continuous analogues of the discrete $T_N$ structures (see \cite{Iqbal2000}).

\vspace{.09in}

We conclude with one final remark: the work \cite{KrupaSzekelyhidi2024T4} shows that global conditions such as the \emph{structural Liu entropy condition} can rule out rank-one connections \emph{as well as} $T_4$ configurations for general classes of  hyperbolic systems. Genuine  nonlinearity is a pointwise differentiability condition, while the structural Liu entropy condition is the global analogue of genuine nonlinearity. 

On the other hand, very recent results from ChatGPT 5.6 Sol Pro and ChatGPT 5.6 Sol Ultra have given arguments which purport to show the existence of genuinely nonlinear hyperbolic systems which admit both rank-one connections as well as non-degenerate $T_4$ configurations. This is closely related to Question 1.6 in \cite{KrupaSzekelyhidi2024T4}. These systems are very special constructions and are not physical, and the arguments only can show genuine nonlinearity on a very, very narrow rectangle. It is also not clear if these very  artificial systems are amenable to $a$-contraction. We have not checked these AI-assisted arguments yet for correctness, but in the interest of the community we include them in the GitHub repository.

\subsubsection{Open questions}
We end with the following open questions:

\Cref{main_theorem} shows that with simple Riemann initial data, the Strong Trace Property is necessary to ensure uniqueness of solutions. But what about if we consider initial data with lower regularity?

\begin{openquestion} Does there exist a hyperbolic system of conservation laws in one spatial dimension for which under natural assumptions, some initial data yield unique solutions and some initial data yield non-unique solutions -- even when all solutions under consideration have the same degree of regularity at positive times?
\end{openquestion}

This closely relates to Bressan's Open Problem $\#6$ \cite{Bressan2024OneDimensional}. It also closely connects to the research program initiated in \cite{2024arXiv240702927C} to construct non-unique solutions to 1-D systems via convex integration \emph{which are also globally continuous}. These solutions automatically verify the Strong Trace Property by virtue of being continuous. However, they do not verify all of the conditions needed for $a$-contraction to apply (namely, they lack a strictly convex entropy functional).

The convex integration solutions we construct in \Cref{main_theorem} are very low regularity. In fact, they are only $L^\infty$ and nothing more. So, it raises the natural question:

\begin{openquestion}
    Is the Strong Trace Property truly the \emph{precise} boundary between uniqueness and non-uniqueness? Or are there weaker assumptions that can also ensure uniqueness?
\end{openquestion}

In the interesting recent works \cite{talamini2026strong,Golding2023}, the authors are able to show that various trace conditions hold for $2\times2$ hyperbolic systems. However, the traces found do not imply the Strong Trace Property. It is unknown if these different trace conditions  suffice for the $a$-contraction theory. We remark that these works \cite{talamini2026strong,Golding2023} use the special fact that there is an infinite family of entropies available in the $2\times2$ case. We also wish to cite the related work on a  Liouville-type theorem for genuinely nonlinear systems \cite{ancona2025liouville}.

\subsection{Outline of the paper}
In \Cref{T_N_section} we introduce rank-one convexity and $T_N$ configurations. In \Cref{solutions_by_convex_integration_section} we introduce convex integration. In \Cref{sec:TN-tangent-space} we study the tangent space to the manifold of $T_N$ configurations. In \Cref{T_N_existence_section} we find a particular $T_N$ configuration (with $N=6$) which verifies various algebraic properties. In \Cref{extension_section}, we show that we can find a particular pressure law $p$, verifying various important properties, for which the constitutive set $K_{F,\eta,q}$ with this $p$ contains the $T_6$ we found earlier.  In \Cref{perturbative_section} we show that we can always perturb our pressure law $p$ to make sure we are in a non-degenerate situation and can thus perform convex integration. In \Cref{sec:pos} we prove the stability estimate \eqref{eq:main-stability-estimate} which holds under the Strong Trace Property. On the other hand, in \Cref{sec:neg} we prove the non-uniqueness stated in \Cref{main_theorem}.  In \Cref{sec:no-TN-p-system} we show nonexistence of $T_N$ for all $N$ for the  $p$-system with $p''>0$. In \Cref{sec:no-Tinfty} we show nonexistence of $T_\infty$ configurations for all genuinely nonlinear $2\times2$ systems.

\textbf{Additional notation.} For a matrix $A\in\mathbb{R}^{p\times q}$, we denote its entry in the $i$th row and $j$th column by $A_{i,j}$.

\vspace{.07in}

\textbf{Statement on AI Use}

The author used ChatGPT 5.6 Sol Pro and ChatGPT 5.6 Sol Ultra as well as ChatGPT 6 Astra Pro. The proofs of nonexistence of $T_N$ for the $p$-system and proof of nonexistence of $T_\infty$ for general classes of systems are almost entirely from these Large Language Models, with minor editing and integration with the rest of the manuscript. The prompting essentially consisted of uploading relevant papers including some of the author's own as well as various earlier trials on related problems. 

The main push in this paper is the resolution of what we call the ``Strong Trace Conjecture'' which has been a major open problem in the theory of conservation laws. The key idea was to use the localization argument inherent in $a$-contraction theory (see \Cref{size_pi}, and \cite{MR3519973}) as well as the proof technique developed during the author's thesis work \cite{move_entire_solution_system}, and combine these ideas with the fact that  $T_N$ configurations and convex integration can allow for large displacements in state space, outside of the localized $a$-contraction region. Large Language Models helped here, in particular with the dissipation estimates. However, most of this work was started before ChatGPT 5.6 Sol was released and we relied heavily on the computational ansatz and numerical search used in \cite{KrupaSzekelyhidi2025Contact,Krupa2026NonUniqueness}.

\vspace{.07in}

\textbf{GitHub repository.} Our proof involves symbolic MATLAB code available on the GitHub repository 
\url{https://github.com/sammykrupa/StrongTraceConjecture}

\section{$T_N$ configurations}\label{T_N_section}

We are interested in a certain type of convexity for a certain set corresponding to the conservation law \eqref{eq:system}. If this set is twisted enough, i.e. has a large enough convex hull, under a certain notion of convexity, then this will give enough room for (nonunique) solutions to exist which oscillate back and forth rapidly -- creating very rough solutions to \eqref{eq:system}. This is in fact a general principle which can be used to study nonlinear partial differential equations (see \cite{MR2008346} and \cite{MR1983780}). In order to explain this in detail, in this section we recall the relevant definitions and results regarding \emph{rank-one convexity}. A function $f\colon\mathbb{R}^{m\times n}\to\mathbb{R}$ is rank-one convex if $f$ is convex along each rank-one line. The \emph{rank-one convex hull} of a set of matrices is defined by separation with rank-one convex functions, as follows. For a compact set $K\subset \mathbb{R}^{m\times n}$, we define
\begin{align}
K^{\text{rc}}\coloneqq \big\{X\in\mathbb{R}^{m\times n} : f(X)\leq \sup_{K} f \mbox{ for all } f\colon \mathbb{R}^{m\times n}\to\mathbb{R} \mbox{ rank-one convex}\big\},
\end{align}
and for general sets
\begin{align}
E^{\text{rc}}\coloneqq \bigcup_{K\subset E \text{ compact}} K^{\text{rc}}.
\end{align}

The objects dual to rank-one convex functions are a subclass of probability measures supported on $\mathbb{R}^{m\times n}$ called \emph{laminates}. That is, a probability measure $\nu$ on the space of $m\times n$ matrices is a laminate if
\begin{align}
\langle \nu,f \rangle \geq f(\bar{\nu}) \mbox{ for all rank-one convex }f\colon \mathbb{R}^{m\times n}\to\mathbb{R},
\end{align}
where $\bar{\nu}$ denotes the barycenter of the measure $\nu$. The set of barycenters of laminates with support in a fixed compact set $K$ is exactly the rank-one convex hull $K^{\text{rc}}$.

It is possible for the rank-one convex hull of a set $K$ to be nontrivial (in other words, strictly larger than $K$), even if $K$ contains no rank-one connections, i.e. $\rank(X-Y)>1$ for any two distinct $X,Y\in K$. This fact has been observed independently by a number of authors in different contexts (see for example \cite{MR2624766,MR852476,MR1283802,MR1320538,MR1106125}). The rank-one convex hull of a set $K$ will be nontrivial if it contains points $\{X_1,\ldots,X_N\}$ verifying the following cyclic structure:

\subsection{Ordered \texorpdfstring{$T_N$}{T-N}-configurations}

Set
\begin{equation}
    \mathbb{M}:=\mathbb{R}^{3\times 2},
    \qquad
    \mathcal{R}_1:=\{C\in\mathbb{M}:\rank C=1\}.
\end{equation}
Thus $\mathcal{R}_1$ denotes the nonzero rank-one manifold in
$\mathbb{M}$.

\begin{definition}[Ordered $T_N$-configuration]\label{def:ordered-TN}
Fix $N\geq 4$.  An ordered $N$-tuple
\begin{equation}
    \mathbf{X}=(X_1,\ldots,X_N)\in\mathbb{M}^N
\end{equation}
is an ordered $T_N$-configuration if there exist
\begin{equation}
    P\in\mathbb{M},
    \qquad
    C_i\in\mathcal{R}_1,
    \qquad
    \kappa_i>1,
    \qquad i=1,\ldots,N,
\end{equation}
such that
\begin{equation}
    \sum_{i=1}^N C_i=0
    \label{eq:TN-closure}
\end{equation}
and
\begin{equation}
    X_i
    =P+\sum_{j=1}^{i-1}C_j+\kappa_i C_i,
    \qquad i=1,\ldots,N.
    \label{eq:TN-parametrization}
\end{equation}
We call
\begin{equation}
    \theta=(P,C_1,\ldots,C_N,\kappa_1,\ldots,\kappa_N)
\end{equation}
an ordered $T_N$-parameterization of $\mathbf{X}$.
\end{definition}

Some definitions additionally require
\begin{equation}
    \rank(X_i-X_j)=2,
    \qquad i\neq j,
    \label{eq:no-rank-one-connections}
\end{equation}
so that no two members of the configuration are rank-one connected.

\begin{figure}[tb]
      \includegraphics[width=.7\textwidth]{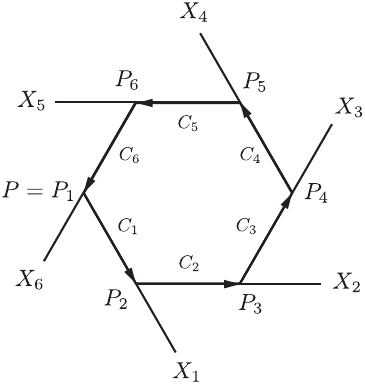}\hspace{.8in}
  \caption{A schematic of a $T_6$ configuration. }\label{example_T6}
\end{figure}

For example a $T_6$ configuration can be represented as in \Cref{example_T6}. 

Introducing the intermediate points
\begin{equation}\label{eq:TN-intermediate-points}
    P_1:=P,
    \qquad
    P_i
    :=
    P+\sum_{j=1}^{i-1}C_j,
    \qquad i=2,\ldots,N,
\end{equation}
one may equivalently write
\begin{equation}\label{eq:TN-cyclic-form}
    X_i=P_i+\kappa_i C_i,
    \qquad
    P_{i+1}=P_i+C_i,
\end{equation}
where the indices are understood cyclically, so that
\(P_{N+1}=P_1\). In particular, the points \(P_i\) form a closed polygon
whose edges are rank-one directions, whereas the matrices \(X_i\) lie
beyond the corresponding vertices along those directions.

The following well-known result shows that if a set $K$ contains such a $T_N$ configuration, then $K$ will have a nontrivial rank-one convex hull -- the striking fact is that the cyclic structure of the $T_N$ configuration adds points to the rank-one convex hull.

\begin{lemma}\label{folklore_lemma}
Let $\{X_1,\ldots,X_N\}$ be a $T_N$ configuration, and for $i=1,\ldots,N$ let $P_i=P+C_1+\cdots+C_{i-1}$ (so that $P_1=P$). Then
\begin{align}
\{P_1,\ldots,P_N\} \subset \{X_1,\ldots,X_N\}^{\text{rc}}.
\end{align}
In particular, for each $k=1,\ldots,N$ there exist numbers $\mu_{i}^{(k)}\in(0,1)$ such that the probability measures
\begin{align}
\mu^{(k)}=\sum_{i=1}^N \mu_{i}^{(k)}\delta_{X_i}
\end{align}
are laminates with barycenter $\bar{\mu}^{(k)}=P_k$.
\end{lemma}

\Cref{folklore_lemma} can be verified readily; we do not include the proof.

It is well known that $T_N$ configurations form locally a manifold in the space of ordered $N$-tuples of matrices. See for example \cite[Section 4.2]{MR1983780}, \cite[Proposition 4.26]{hab_thesis}, and \cite[Lemma 2]{MR2048569}. This is explained in more detail in \Cref{sec:TN-tangent-space}.

\section{Solutions by Convex Integration}\label{solutions_by_convex_integration_section}
 
 For a given flux $F$, entropy $\eta$, and entropy flux $q$, following \cite{MR2008346}, we consider stream functions $\psi\colon\mathbb{R}^2\to\mathbb{R}^3$ such that
 \begin{equation}
 \begin{aligned}\label{streaming}
 (u_1,-F_1(U))&=((\psi_1)_x,(\psi_1)_t)\\
 (u_2,-F_2(U))&=((\psi_2)_x,(\psi_2)_t)\\
 (\eta(U),-q(U))&=((\psi_3)_x,(\psi_3)_t),
 \end{aligned}
 \end{equation}
 where $U=(u_1,u_2)^{\mathsf T}$, $\psi=(\psi_1,\psi_2,\psi_3)^{\mathsf T}$, and $F=(F_1,F_2)^{\mathsf T}$. If $\psi$ verifies \eqref{streaming}, then the function $U\colon\mathbb{R}^2\to\mathbb{R}^2$ defined by $U\coloneqq ((\psi_1)_x,(\psi_2)_x)^{\mathsf T}$ verifies $\partial_t U+\partial_x F(U)=0$ and \eqref{eq:entropy-conservation} in the sense of distributions. This follows from the commutativity of distributional mixed derivatives. In particular, \eqref{eq:entropy-conservation} is satisfied as an exact equality.
 
 This motivates us to define, for a given flux $F$, entropy $\eta$, and entropy flux $q$, the 2-dimensional manifold
\begin{align}\label{manifold_K}
K_{F,\eta,q}
&\coloneqq
\bigl\{G(U):U=(u_1,u_2)^{\mathsf T}\in\mathcal{V}\bigr\}
\subset\mathbb{R}^{3\times2},
\end{align}
where 
\begin{equation}\label{eq:constitutive-surface}
    G(U)
    \coloneqq
    \begin{pmatrix}
        u_1 & -F_1(U)\\
        u_2 & -F_2(U)\\
        \eta(U) & -q(U)
    \end{pmatrix}
\end{equation}
and where we again write $F$ in terms of its components $F=(F_1,F_2)^{\mathsf T}$.

Then, we can rewrite \eqref{streaming} as the differential inclusion
\begin{align}
D\psi \in K_{F,\eta,q}.
\end{align}

In this paper, we will construct such stream functions $\psi$. They will be only Lipschitz, so their derivatives, which are solutions to \eqref{eq:system}, will be only $L^\infty$.

To run the convex integration scheme, it will be necessary to find a large number of $T_N$ configurations in the constitutive set $K_{F,\eta,q}$; in fact we will need to have enough $T_N$ configurations to actually make open sets. The idea will be to find one $T_N$ configuration, and then use a transversality argument to produce a submanifold of $T_N$ configurations living in $K_{F,\eta,q}$. This argument will involve a non-degeneracy condition (Condition (C), introduced below).

To explain this, we cite    \Cref{prop:parameter-manifold,thm:TN-tangent-space}, which says that, locally around an ordered $N$-tuple
\begin{align}
z^0=(Z_1^0,\ldots,Z_N^0)\in(\mathbb{R}^{3\times2})^N,
\end{align}
which is a $T_N$ configuration, there exists a smooth manifold of (ordered) $N$-tuples, $\mathcal{M}_N$, consisting of $T_N$ configurations. Moreover $\dim\mathcal{M}_N=5N$. Let
\begin{align}
\mathcal{K}_{F,\eta,q}
\coloneqq
(K_{F,\eta,q})^N
\subset
(\mathbb{R}^{3\times2})^N.
\end{align}
Thus $\mathcal{K}_{F,\eta,q}$ is a $2N$-dimensional smooth manifold. Define the maps $\pi_k,\phi_k\colon \mathcal{M}_N\to\mathbb{R}^{3\times2}$ by
\begin{align}\label{projection_maps}
\pi_k(Z_1,\ldots,Z_N)=P_k \hspace{.3in} \mbox{ and }\hspace{.3in} \phi_k(Z_1,\ldots,Z_N)=Z_k
\end{align}
for $k=1,\ldots,N$, where $P_k$ is as in \Cref{folklore_lemma}.

Now, we recall some basic facts regarding transversality (for more, see for example \cite{MR0348781}): suppose two smooth manifolds $\mathcal{M}$ and $\mathcal{K}$ embedded in $\mathbb{R}^d$ intersect at a point $z$. The intersection is \emph{transversal} if the tangent spaces at the point $z$ verify
\begin{align}
T_z\mathcal{M}+T_z\mathcal{K}=\mathbb{R}^d.
\end{align}
As a direct consequence of the implicit function theorem, if the manifolds $\mathcal{M}$ and $\mathcal{K}$ intersect transversely, then locally their intersection is a smooth manifold. Further, $\dim(\mathcal{M}\cap\mathcal{K})=\dim\mathcal{M}+\dim\mathcal{K}-d$.

Then, in our case, if $\mathcal{M}_N$ and $\mathcal{K}_{F,\eta,q}$ intersect transversely, then their intersection is a manifold of dimension $N$. Thus, for $N\geq 6$, we can expect that generically the map $\pi_k$ restricted to $\mathcal{M}_N\cap\mathcal{K}_{F,\eta,q}$ is a submersion. I.e., the image of $\pi_k$ of the intersection is an open set. Remark also that by \Cref{folklore_lemma} this image is contained in the rank-one convex hull of $K_{F,\eta,q}$.

We now give the precise condition on this genericity:
\begin{definition}[Condition (C)]\label{condition_c_def}
Suppose $F$, $\eta$, and $q$ are given, and let
\begin{equation}
z^0=(Z_1^0,\ldots,Z_N^0)\in
\mathcal{M}_N\cap\mathcal{K}_{F,\eta,q}
\end{equation}
be an ordered $T_N$ configuration, where $\mathcal{M}_N$ is the local manifold of ordered $T_N$ configurations given by \Cref{prop:parameter-manifold,thm:TN-tangent-space}. If $\mathcal{M}_N$ and $\mathcal{K}_{F,\eta,q}$ intersect transversely at $z^0$, and if, for each $k=1,\ldots,N$, the restriction
\begin{equation}
\pi_k\big|_{\mathcal{M}_N\cap\mathcal{K}_{F,\eta,q}}
\end{equation}
is a submersion at $z^0$, then the triple $(F,\eta,q)$ is said to satisfy Condition~(C) at $z^0$.
\end{definition}

Then, our main theorem (\Cref{main_theorem}) will follow from:

\begin{proposition}\label{main_prop}

Assume $F,\eta,q\in C^2$, with $F\colon \mathcal{V}\to\mathbb{R}^2$ and $\eta,q\colon \mathcal{V}\to\mathbb{R}$. Let
\begin{equation}
z^0=(Z_1^0,\ldots,Z_N^0)\in\mathcal{K}_{F,\eta,q}
\end{equation}
be an ordered $T_N$ configuration at which $(F,\eta,q)$ satisfies Condition~(C). Let $P_0\in\{Z_1^0,\ldots,Z_N^0\}^{\text{rc}}$. Consider an open set $\Omega\subset\mathbb{R}\times[0,\infty)$. For any $\delta>0$ there exists a Lipschitz map
\[
w\colon \Omega\subset \mathbb{R}\times[0,\infty)\to\mathbb{R}^3\]
with the following properties:
\begin{enumerate}
    \item[(a)]
    \begin{align}
Dw(x,t)\in K_{F,\eta,q}\cap \Big(\cup_{k=1}^N B_\delta(Z_k^0)\Big) \hspace{.3in} \text{a.e. in }\Omega;
\end{align}

    \item[(b)]
    In particular, the map
    \[
    U:=\bigl((Dw)_{1,1},(Dw)_{2,1}\bigr)^{\mathsf T}
    \]
    is a solution to the conservation law \eqref{eq:system}, \eqref{eq:entropy-conservation};

    \item[(c)]
    and
     \[
    w(x,t)=P_0\begin{bsmallmatrix},
x\\t
\end{bsmallmatrix}
    \qquad\text{on }\partial\Omega,
    \]
    \[
    |w(x,t)-P_0\begin{bsmallmatrix},
x\\t
\end{bsmallmatrix}|<\delta
    \qquad\text{in }\Omega.
    \]
\end{enumerate}
\end{proposition}

The proof of \Cref{main_prop} is nearly identical to the similar proofs in \cite{MR2048569,MR1983780}. We do not reproduce the proof here.

We then need to make sure to satisfy the hypotheses in \Cref{main_prop}. We do so for a $T_N$ configuration, with $N=6$.  In \Cref{T_N_existence_section} we give a method for finding an explicit $T_N$ configuration. In \Cref{extension_section}, we give algebraic conditions which, when verified by a $T_N$ configuration, allow for the construction of an $F,\eta$, and $q$ (in the context of \Cref{main_theorem}) such that $K_{F,\eta,q}$ contains the $T_N$ configuration. Then, in \Cref{perturbative_section}, we show that it is  possible to perturb the $F,\eta$, and $q$ from \Cref{extension_section} such that we are in the non-degenerate situation of Condition (C). 

\vspace{.05in}

\section{Tangent space to the manifold of $T_N$ configurations}\label{sec:TN-tangent-space}

As it will prove useful later, in this section we give an explicit construction of a candidate spanning family for the tangent space of the manifold of $T_N$ configurations, at a fixed ``point'' on this manifold which corresponds to a $T_N$ configuration with parameterization $(P,(C_i),(\kappa_i))$.

 Given a parameterization of a \(T_N\)
(the \(P\), the \(C_i=a_i\otimes b_i\), the \(\kappa_i\)),
we also have symbolic MATLAB code which encapsulates the  tangent-space description. It takes as input the  \(P\), \(\{a_i,b_i\}\), and
\(\{\kappa_i\}\), then returns the full \(5N\)-vector basis.

Since
\eqref{eq:no-rank-one-connections} is an open condition, it does not change
the tangent-space calculation at a configuration for which it holds.
Likewise, the inequalities $\kappa_i>1$ are open and produce no linearized
constraints.

We expect our computational analysis of the manifold of $T_N$ configurations will help also with future work on the rank-one convex geometry of hyperbolic PDE.

\subsection{The rank-one tangent spaces}

For vectors $a\in\mathbb{R}^3$ and $b\in\mathbb{R}^2$, we use the
convention
\begin{equation}
    a\otimes b:=ab^{\mathsf T}\in\mathbb{M}.
\end{equation}

\begin{lemma}[Tangent space to the rank-one manifold]
\label{lem:rank-one-tangent}
Let $C=a\otimes b\in\mathcal{R}_1$, where $a\neq 0$ and $b\neq 0$.
Then
\begin{equation}
    T_C\mathcal{R}_1
    =
    \left\{
        \alpha\otimes b+a\otimes\beta:
        \alpha\in\mathbb{R}^3,
        \ \beta\in\mathbb{R}^2
    \right\}.
    \label{eq:rank-one-tangent}
\end{equation}
In particular,
\begin{equation}
    \dim T_C\mathcal{R}_1=4.
\end{equation}
\end{lemma}

\begin{proof}
Consider the smooth map
\begin{equation}
    \mu\colon
    (\mathbb{R}^3\setminus\{0\})\times
    (\mathbb{R}^2\setminus\{0\})
    \longrightarrow \mathbb{M},
    \qquad
    \mu(a,b)=a\otimes b.
\end{equation}
Its differential at $(a,b)$ is
\begin{equation}
    D\mu(a,b)[(\alpha,\beta)]
    =\alpha\otimes b+a\otimes\beta.
\end{equation}
The kernel is the one-dimensional space
\begin{equation}
    \ker D\mu(a,b)
    =\{(ta,-tb):t\in\mathbb{R}\},
\end{equation}
which represents the infinitesimal scaling freedom
$(a,b)\mapsto(e^t a,e^{-t}b)$.  Hence $D\mu(a,b)$ has rank
$3+2-1=4$, and its image is precisely the tangent space to
$\mathcal{R}_1$ at $C$.
\end{proof}

A convenient concrete basis is obtained as follows.  Let
$e_1,e_2,e_3$ be the standard basis of $\mathbb{R}^3$ and set
\begin{equation}
    b^\perp:=(-b_2,b_1)^{\mathsf T}.
\end{equation}
Then
\begin{equation}
    e_1\otimes b,
    \quad e_2\otimes b,
    \quad e_3\otimes b,
    \quad a\otimes b^\perp
    \label{eq:rank-one-tangent-basis}
\end{equation}
form a basis of $T_C\mathcal{R}_1$.  Indeed, the first three matrices
span $\{\alpha\otimes b:\alpha\in\mathbb{R}^3\}$, whereas
$b^\perp$ is linearly independent of $b$.

\subsection{The closure constraint and the parameter manifold}

Define the closure map
\begin{equation}
    \Sigma\colon\mathcal{R}_1^N\longrightarrow\mathbb{M},
    \qquad
    \Sigma(C_1,\ldots,C_N):=\sum_{i=1}^N C_i.
\end{equation}
Its differential is
\begin{equation}
    D\Sigma(C_1,\ldots,C_N)
    [(\dot C_1,\ldots,\dot C_N)]
    =\sum_{i=1}^N \dot C_i.
    \label{eq:linearized-closure-map}
\end{equation}

\begin{definition}[Closure regularity]
\label{def:closure-regularity}
A closed family $(C_1,\ldots,C_N)\in\mathcal{R}_1^N$ satisfying
\eqref{eq:TN-closure} is called closure-regular if
\begin{equation}
    \sum_{i=1}^N T_{C_i}\mathcal{R}_1=\mathbb{M}.
    \label{eq:closure-regularity}
\end{equation}
Equivalently, the differential in \eqref{eq:linearized-closure-map} is
surjective.
\end{definition}

\begin{remark}
If the $N$ matrices in the $T_N$ configuration have no rank-one connections between them, then in particular $\rank(C_1-C_2)=2$. If we write $C_i=a_i\otimes b_i$, then $\{b_1,b_2\}$ is a basis of $\mathbb{R}^2$. Hence \eqref{eq:closure-regularity} holds.
\end{remark}

\begin{proposition}[Local parameter manifold]
\label{prop:parameter-manifold}
Suppose $(C_1,\ldots,C_N)$ is closure-regular.  Then, in a neighborhood of
that point,
\begin{equation}
    \mathcal{C}_N
    :=
    \left\{
        (D_1,\ldots,D_N)\in\mathcal{R}_1^N:
        \sum_{i=1}^N D_i=0
    \right\}
\end{equation}
is a smooth manifold of dimension $4N-6$, with tangent space
\begin{equation}
    T_{(C_1,\ldots,C_N)}\mathcal{C}_N
    =
    \left\{
        (\dot C_1,\ldots,\dot C_N):
        \dot C_i\in T_{C_i}\mathcal{R}_1,
        \ \sum_{i=1}^N\dot C_i=0
    \right\}.
    \label{eq:closure-tangent-space}
\end{equation}
Consequently, the ordered parameter space
\begin{equation}
    \mathcal{P}_N
    :=\mathbb{M}\times\mathcal{C}_N\times(1,\infty)^N
    \label{eq:parameter-manifold}
\end{equation}
is locally a smooth manifold of dimension
\begin{equation}
    \dim\mathcal{P}_N
    =6+(4N-6)+N=5N.
    \label{eq:parameter-manifold-dimension}
\end{equation}
\end{proposition}

\begin{proof}
By Lemma~\ref{lem:rank-one-tangent}, the manifold $\mathcal{R}_1^N$ has
dimension $4N$.  Closure regularity says that $0\in\mathbb{M}$ is a regular
value of $\Sigma$ at the given point.  The implicit-function theorem
therefore gives a codimension-six level set, which yields both
\eqref{eq:closure-tangent-space} and $\dim\mathcal{C}_N=4N-6$.
The remaining assertion follows by taking the product with $\mathbb{M}$ and
$(1,\infty)^N$.
\end{proof}

\subsection{The tangent space in configuration variables}

Define the ordered $T_N$ map
\begin{equation}
    \Phi\colon\mathcal{P}_N\longrightarrow\mathbb{M}^N
\end{equation}
by
\begin{equation}
    \Phi(P,C_1,\ldots,C_N,\kappa_1,\ldots,\kappa_N)
    =(X_1,\ldots,X_N),
\end{equation}
where the $X_i$ are given by \eqref{eq:TN-parametrization}.

\begin{proposition}[Differential of the ordered $T_N$ map]
\label{prop:TN-differential}
Let
\begin{equation}
    (\dot P,\dot C_1,\ldots,\dot C_N,
      \dot\kappa_1,\ldots,\dot\kappa_N)
    \in T_\theta\mathcal{P}_N.
\end{equation}
Thus
\begin{equation}
    \dot P\in\mathbb{M},
    \qquad
    \dot C_i\in T_{C_i}\mathcal{R}_1,
    \qquad
    \sum_{i=1}^N\dot C_i=0,
    \qquad
    \dot\kappa_i\in\mathbb{R}.
    \label{eq:admissible-parameter-variation}
\end{equation}
Then
\begin{equation}
    D\Phi(\theta)
    \bigl[(\dot P,(\dot C_i)_{i=1}^N,(\dot\kappa_i)_{i=1}^N)\bigr]
    =(H_1,\ldots,H_N),
    \label{eq:D-Phi}
\end{equation}
where
\begin{equation}
    H_i
    =\dot P+
      \sum_{j=1}^{i-1}\dot C_j+
      \kappa_i\dot C_i+
      \dot\kappa_i C_i,
    \qquad i=1,\ldots,N.
    \label{eq:tangent-component-formula}
\end{equation}
\end{proposition}

\begin{proof}
Differentiate \eqref{eq:TN-parametrization}.  The condition
$\sum_i\dot C_i=0$ is the differential of the closure constraint
\eqref{eq:TN-closure}.
\end{proof}

Let $\mathcal{T}_N^{\mathrm{ord}}$ denote the branch of ordered
$T_N$-configurations generated by $\Phi$ near $\theta$.  The following
formulation separates two distinct regularity conditions: closure regularity
of the parameter constraint and immersion regularity of the map from
parameters to configurations.

\begin{theorem}[Tangent space to a local ordered $T_N$ branch]
\label{thm:TN-tangent-space}
Assume that $(C_i)_{i=1}^N$ is closure-regular.  Suppose also that
$D\Phi(\theta')$ has constant rank $r$ for $\theta'$ in a neighborhood of
$\theta$.  Then the local image of
$\Phi$ is an immersed $r$-dimensional manifold and
\begin{equation}
\boxed{
    T_{\mathbf X}\mathcal{T}_N^{\mathrm{ord}}
    =
    \left\{
    (H_1,\ldots,H_N)\in\mathbb{M}^N:
    \begin{aligned}
      &H_i=\dot P+\sum_{j<i}\dot C_j
             +\kappa_i\dot C_i+\dot\kappa_i C_i,\\
      &\dot P\in\mathbb{M},\quad
       \dot\kappa_i\in\mathbb{R},\quad
       \dot C_i\in T_{C_i}\mathcal{R}_1,\quad
       \sum_{i=1}^N\dot C_i=0
    \end{aligned}
    \right\}.
}
    \label{eq:TN-tangent-space}
\end{equation}
In particular, if $D\Phi(\theta)$ is injective, then
\begin{equation}
    r=5N
    \qquad\text{and}\qquad
    \dim T_{\mathbf X}\mathcal{T}_N^{\mathrm{ord}}=5N.
    \label{eq:TN-tangent-dimension}
\end{equation}
\end{theorem}

\begin{proof}
The tangent space to the parameter manifold is given by
\eqref{eq:admissible-parameter-variation}, and
Proposition~\ref{prop:TN-differential} computes its image under $D\Phi(\theta)$.
The constant-rank theorem identifies that image with the tangent space to
the local immersed image.  If $D\Phi(\theta)$ is injective, its rank equals
$\dim\mathcal{P}_N=5N$ by
\eqref{eq:parameter-manifold-dimension}.
\end{proof}

\begin{remark}[Why the two regularity conditions should not be conflated]
The surjectivity condition \eqref{eq:closure-regularity} ensures that the
closure equation cuts out a codimension-six parameter manifold.  It does not,
by itself, imply that two infinitesimally different parameter tuples produce
different infinitesimal configurations.  The latter assertion is the
injectivity of $D\Phi(\theta)$. However, if $\Phi$ is a local parametrization of the set of ordered $T_N$-
configurations near the point under consideration---that is, a
local diffeomorphism from $\mathcal{P}_N$ onto its image---then
\[
D\Phi(\theta)\colon T_\theta \mathcal{P}_N
\longrightarrow
T_{\mathbf{X}} \mathcal{T}_N^{\mathrm{ord}}
\]
is a linear isomorphism. Consequently,
\[
\boxed{
\dim T_{\mathbf{X}} \mathcal{T}_N^{\mathrm{ord}}
=
\dim \mathcal{T}_N^{\mathrm{ord}}
=
5N.
}
\]  The basis construction below therefore first
produces $5N$ candidate directions and then, when necessary, performs one
final column reduction in configuration space. 
\end{remark}

An equivalent polygonal form is sometimes useful.  Define
\begin{equation}
    P_1:=P,
    \qquad
    P_{i+1}:=P_i+C_i,
    \qquad i=1,\ldots,N.
\end{equation}
Then $P_{N+1}=P_1$ and $X_i=P_i+\kappa_iC_i$.  A tangent vector is equivalently
described by
\begin{equation}
    \dot P_{i+1}-\dot P_i=\dot C_i,
    \qquad
    \dot P_{N+1}=\dot P_1,
    \qquad
    H_i=\dot P_i+\kappa_i\dot C_i+\dot\kappa_iC_i,
\end{equation}
with $\dot C_i\in T_{C_i}\mathcal{R}_1$.

\subsection{An explicit basis construction}
\label{subsec:explicit-TN-basis}

Fix factorizations
\begin{equation}
    C_i=a_i\otimes b_i,
    \qquad
    a_i\in\mathbb{R}^3\setminus\{0\},
    \quad
    b_i\in\mathbb{R}^2\setminus\{0\}.
\end{equation}
For each $i$, define
\begin{equation}
\begin{aligned}
    E_{i,1}&:=e_1\otimes b_i,\\
    E_{i,2}&:=e_2\otimes b_i,\\
    E_{i,3}&:=e_3\otimes b_i,\\
    E_{i,4}&:=a_i\otimes b_i^\perp,
    \qquad
    b_i^\perp:=(-b_{i,2},b_{i,1})^{\mathsf T}.
\end{aligned}
    \label{eq:Eia-basis}
\end{equation}
By \eqref{eq:rank-one-tangent-basis}, the matrices
$E_{i,1},\ldots,E_{i,4}$ form a basis of $T_{C_i}\mathcal{R}_1$.

\subsubsection{The linearized closure matrix}

Let
\begin{equation}
    \operatorname{vec}\colon\mathbb{M}\longrightarrow\mathbb{R}^6
\end{equation}
denote column-wise vectorization.  Define
\begin{equation}
    A
    :=
    \begin{bmatrix}
       \operatorname{vec}(E_{1,1}) & \cdots &
       \operatorname{vec}(E_{1,4}) &
       \operatorname{vec}(E_{2,1}) & \cdots &
       \operatorname{vec}(E_{N,4})
    \end{bmatrix}
    \in\mathbb{R}^{6\times 4N}.
    \label{eq:closure-matrix-A}
\end{equation}
If
\begin{equation}
    \dot C_i=\sum_{\alpha=1}^4 z_{i,\alpha}E_{i,\alpha},
\end{equation}
then
\begin{equation}
    \sum_{i=1}^N\dot C_i=0
    \quad\Longleftrightarrow\quad
    Az=0,
    \label{eq:Az-zero}
\end{equation}
where the coefficient vector $z\in\mathbb{R}^{4N}$ is ordered by
$(i,\alpha)$.  Closure regularity is equivalent to
\begin{equation}
    \rank A=6.
    \label{eq:A-rank-six}
\end{equation}

Choose a set $\mathcal{J}\subset\{1,\ldots,4N\}$ of six column indices for
which the square submatrix $A_{\mathcal J}$ is invertible, and let
\begin{equation}
    \mathcal{F}:=\{1,\ldots,4N\}\setminus\mathcal{J}.
\end{equation}
For each $\ell\in\mathcal{F}$, define $z^{(\ell)}\in\mathbb{R}^{4N}$ by
\begin{equation}
\begin{aligned}
    z^{(\ell)}_\ell&=1,\\
    z^{(\ell)}_m&=0,
        &&m\in\mathcal{F}\setminus\{\ell\},\\
    z^{(\ell)}_{\mathcal J}
        &=-A_{\mathcal J}^{-1}A_{(:,\ell)}.
\end{aligned}
    \label{eq:kernel-basis-coefficients}
\end{equation}
Then
\begin{equation}
    \{z^{(\ell)}:\ell\in\mathcal{F}\}
\end{equation}
is a basis of $\ker A$.  In particular, it contains $4N-6$ vectors.
For each free index $\ell$, recover the associated rank-one shape variation
by
\begin{equation}
    \dot C_i^{(\ell)}
    :=\sum_{\alpha=1}^4 z_{i,\alpha}^{(\ell)}E_{i,\alpha}.
    \label{eq:dCi-shape}
\end{equation}

\subsubsection{Translation, radial, and shape directions}

Let $E^{pq}\in\mathbb{M}$ be the matrix unit with a single $1$ in entry
$(p,q)$, where $1\leq p\leq 3$ and $1\leq q\leq 2$.

The six translation directions are
\begin{equation}
    \mathbf{T}_{pq}
    :=(E^{pq},\ldots,E^{pq})\in\mathbb{M}^N.
    \label{eq:translation-directions}
\end{equation}
For each $i$, the variation of $\kappa_i$ gives the direction
\begin{equation}
    \mathbf{K}_i
    :=(0,\ldots,0,C_i,0,\ldots,0),
    \label{eq:kappa-directions}
\end{equation}
where $C_i$ occupies the $i$th component.  Finally, for each
$\ell\in\mathcal{F}$ define the shape direction
\begin{equation}
    \mathbf{S}_\ell
    :=(S_{\ell,1},\ldots,S_{\ell,N}),
\end{equation}
where
\begin{equation}
    S_{\ell,i}
    :=\sum_{j=1}^{i-1}\dot C_j^{(\ell)}
      +\kappa_i\dot C_i^{(\ell)}.
    \label{eq:shape-directions}
\end{equation}

\begin{proposition}[Candidate basis and actual basis]
\label{prop:explicit-TN-basis}
Assume closure regularity.  The family
\begin{equation}
    \mathcal{B}_{\mathrm{cand}}
    :=
    \{\mathbf{T}_{pq}:1\leq p\leq3,\ 1\leq q\leq2\}
    \cup
    \{\mathbf{K}_i:1\leq i\leq N\}
    \cup
    \{\mathbf{S}_\ell:\ell\in\mathcal{F}\}
    \label{eq:candidate-basis}
\end{equation}
spans $T_{\mathbf X}\mathcal{T}_N^{\mathrm{ord}}$.  It has
\begin{equation}
    6+N+(4N-6)=5N
\end{equation}
members.  If $D\Phi(\theta)$ is injective, then
$\mathcal{B}_{\mathrm{cand}}$ is a basis.  Without the injectivity assumption,
an actual basis is obtained by vectorizing the members of
$\mathcal{B}_{\mathrm{cand}}$ and retaining any maximal linearly independent
subfamily.
\end{proposition}

\begin{proof}
The six $\dot P$ directions, the $N$ scalar $\dot\kappa_i$ directions, and
the $4N-6$ kernel vectors in \eqref{eq:kernel-basis-coefficients} form a
basis of $T_\theta\mathcal{P}_N$.  Equations
\eqref{eq:translation-directions}, \eqref{eq:kappa-directions}, and
\eqref{eq:shape-directions} are their images under $D\Phi(\theta)$.
Consequently, their images span $\operatorname{im} D\Phi(\theta)$, which is
the tangent space by Theorem~\ref{thm:TN-tangent-space}.  If $D\Phi(\theta)$
is injective, the image of a basis is a basis.
\end{proof}

For the final reduction, define
\begin{equation}
    \operatorname{Vec}(H_1,\ldots,H_N)
    :=
    \begin{pmatrix}
        \operatorname{vec}(H_1)\\
        \vdots\\
        \operatorname{vec}(H_N)
    \end{pmatrix}
    \in\mathbb{R}^{6N}.
\end{equation}
Place the vectors $\operatorname{Vec}(\mathbf{B})$, with
$\mathbf{B}\in\mathcal{B}_{\mathrm{cand}}$, as the columns of a matrix
$B_{\mathrm{cand}}\in\mathbb{R}^{6N\times 5N}$.  Then
\begin{equation}
    \dim T_{\mathbf X}\mathcal{T}_N^{\mathrm{ord}}
    =\rank B_{\mathrm{cand}},
    \label{eq:tangent-rank-B}
\end{equation}
and the pivot columns of $B_{\mathrm{cand}}$ form an actual tangent-space
basis.  In particular, the immersion condition is equivalent to
\begin{equation}
    \rank B_{\mathrm{cand}}=5N.
\end{equation}

\subsection{Symbolic MATLAB implementation}
\label{subsec:MATLAB-TN-tangent-basis}

We implement the preceding construction
using the MATLAB Symbolic Math Toolbox.  Its inputs are $P$, cell arrays
containing the vectors $a_i$ and $b_i$, and the vector of parameters
$\kappa_i$.  The routine performs the following operations:
\begin{enumerate}
    \item It constructs $C_i=a_i b_i^{\mathsf T}$ and the configuration
          $X_i$.
    \item It builds the four matrices in \eqref{eq:Eia-basis} for each $i$.
    \item It forms the closure matrix $A$ in \eqref{eq:closure-matrix-A},
          computes an RREF (reduced row-echelon form)  basis of $\ker A$, and pushes these shape
          variations forward through \eqref{eq:tangent-component-formula}.
    \item It appends the translation and $\kappa$ directions.

\end{enumerate}
The field \texttt{out.basis} stores each basis vector as an $N$-tuple of
$3\times2$ symbolic matrices, while \texttt{out.basisMatrix} stores the same
basis as columns in $\mathbb{R}^{6N}$.  

We refer to the file 
\path{tn_tangent_basis_symbolic.m}.

It takes \(P\), the cells \(\{a_i\}\), \(\{b_i\}\), and the vector
\(\kappa\), then:
\begin{itemize}
    \item builds
    \[
    C_i=a_ib_i^{\mathsf T};
    \]
    \item uses the local basis
    \[
    E_{i,1}=e_1\otimes b_i,
    \qquad
    E_{i,2}=e_2\otimes b_i,
    \qquad
    E_{i,3}=e_3\otimes b_i,
    \qquad
    E_{i,4}=a_i\otimes b_i^\perp;
    \]
    \item forms the symbolic closure matrix \(A\) for
    \[
    \sum_i\dot C_i=0;
    \]
    \item computes a symbolic kernel basis of \(A\) via \texttt{rref};
    \item pushes those directions forward to
    \[
    (\dot X_1,\ldots,\dot X_N);
    \]
    \item adds the \(6\) translation directions and the \(N\)
    \(\kappa_i\)-directions;
    \item then does one final symbolic column reduction so the returned
    columns are an actual basis of the linearized image.
\end{itemize}

The main output is:
\begin{itemize}
    \item \texttt{out.basis}: a cell array of basis vectors, each one
    an \(N\)-tuple of \(3\times2\) symbolic matrices;
    \item \texttt{out.basisMatrix}: the same basis as a single symbolic
    matrix, with one basis vector per column.
\end{itemize}

A typical call looks like this:

\begin{lstlisting}[language=Matlab]
syms k1 k2 k3 k4 real
P = sym(zeros(3,2));

a = { [1;0;0], [0;1;0], [-1;0;0], [0;-1;0] };
b = { [1;1], [1;-1], [1;1], [1;-1] };

out = tn_tangent_basis_symbolic(P, a, b, [k1 k2 k3 k4]);

% Basis as N-tuples:
H1 = out.basis{1};

% Basis as one symbolic matrix:
B = out.basisMatrix;

% Labels:
out.labels
\end{lstlisting}

One practical note: The routine assumes a regular \(T_N\) point.
    It checks the linearized closure rank and warns if the final
    tangent-space rank is smaller than the expected \(5N\).

\section{The existence of a $T_N$ configuration}\label{T_N_existence_section}

In this section, we show the existence of a suitable $T_N$ configuration. In particular, we prove:

\begin{lemma}[The MATLAB Lemma]\label{matlab_lemma}
For $N=6$, there exists a $T_N$ configuration $(X_1,\ldots,X_N)\in(\mathbb{R}^{3\times2})^N$ verifying the inequalities
\begin{align}\label{inequalities_for_convexity_extension_lemma}
H_j > H_i + (X_i)_{2,2}\Big((X_j)_{1,1}-(X_i)_{1,1}\Big) \mbox{ for all } i\neq j,
\end{align}
where $H_i\coloneqq (X_i)_{3,1}-\frac{1}{2}((X_i)_{2,1})^2$,
\begin{align}\label{symmetry_cond}
(X_i)_{2,1}=(X_i)_{1,2}  \mbox{ for all } i,
\end{align}
and 
\begin{align}\label{q_cond}
(X_i)_{3,2}=(X_i)_{1,2} (X_i)_{2,2}  \mbox{ for all } i.
\end{align}

Furthermore, the $N$ state vectors $((X_1)_{1,1},(X_1)_{2,1})^{\mathsf T},\ldots,((X_N)_{1,1},(X_N)_{2,1})^{\mathsf T}$ are distinct and the $T_N$ configuration contains no rank-one connections, i.e. $\rank(X_i-X_j)=2$ for all $i\neq j$. Lastly, we have the ordering
\begin{align}\label{t6_ordering}
    0<(X_2)_{1,1} < (X_4)_{1,1} <(X_5)_{1,1} < (X_1)_{1,1} < (X_3)_{1,1} < (X_6)_{1,1}.
\end{align}
\end{lemma}

\begin{remark}
 We first performed the following numerical search (see proof of \Cref{matlab_lemma}, directly below) to find a strong candidate for a $T_N$ configuration with $N=6$ which verifies \emph{all required properties}. To perturb the numerical search into an exact $T_6$ involves symbolically solving for the roots of high degree polynomials (using the theorem  of Székelyhidi \cite{Szekelyhidi2005RankOneHulls}). We then asked ChatGPT Pro to find another $T_6$ with the same exact properties (including \eqref{inequalities_for_convexity_extension_lemma}, \eqref{symmetry_cond}, \eqref{q_cond}, \eqref{t6_ordering}). With this it succeeded, see the code \path{T6_full_data.m}. See also  \Cref{T_6_with_ordering}, below. The $T_6$ from ChatGPT also verifies the same state space geometry as the $T_6$ from \Cref{matlab_lemma} in the set $\mathcal{V}$ and all other requirements used in the present paper and which are verified by the $T_6$ from \Cref{matlab_lemma}. However, the $T_6$ from ChatGPT has the benefit of being expressed in a simple form with rational numbers or roots of simple quadratic expressions.
\end{remark}

\begin{proof}
In this paper, we follow the computer-assisted proof technique for finding $T_N$ configurations, as introduced in \cite{MR2048569}. However, we expand on this technique and use a novel method of finding $T_N$ configurations for hyperbolic systems with nonlinear algebraic constraints (see \cite{Krupa2026NonUniqueness}). 

In \cite{MR2048569}, the existence of the required $T_5$ configuration was reduced to solving a linear system of inequalities -- using the \emph{simplex algorithm} in Maple V. To get a \emph{linear} system of inequalities, they write the $C_i$, as in the definition of a $T_N$ configuration, as the tensor product of two vectors $C_i=\alpha_i\otimes\beta_i$. However, in order for the constraint $\sum_i C_i=0$ to be linear in the entries of the vectors $\alpha_i$ and $\beta_i$, one of the vectors must be kept fixed. This reduces the flexibility of the $T_N$ configuration and makes it more difficult to impose additional algebraic constraints on the $T_N$ configuration (such as  \eqref{symmetry_cond} and \eqref{q_cond})  and still be able to actually find a $T_N$ configuration verifying those constraints. See \cite[p.~145]{MR2048569} for details.

However, in this paper, we give a method of finding $T_N$ configurations by solving \emph{nonlinear} systems of inequalities. 

We enter into MATLAB the constraints on any $T_N$ configuration (as in the definition of a $T_N$ configuration). In particular, for \emph{fixed} values of the $\widehat\kappa_i$, we solve for column vectors $a_i\in\mathbb{R}^{3\times1}$ and row vectors $n_i\in\mathbb{R}^{1\times2}$, for $i=1,\ldots,N$, with $C_i=a_i n_i$. We impose the closure constraint $\sum_i C_i=0$. Further, we require the matrices $X_i$ to satisfy \eqref{inequalities_for_convexity_extension_lemma}, \eqref{symmetry_cond}, and \eqref{q_cond}.

We choose random values $\widehat\kappa_i>1$ and keep them fixed during the numerical optimization step.

We then use the MATLAB R2019b solver \emph{fmincon} and the \emph{interior-point} algorithm. This gives numeric, \emph{double precision} values for the $a_i$, $n_i$, and $P$. We then convert these values to exact \emph{symbolic values} within MATLAB and perform all further computations symbolically within MATLAB, which allows for performing exact symbolic computations.  Call these symbolic values $\widehat a_i$,  $\widehat n_i$ and $\widehat P$.

Now, due to coming from the approximate, numerical solver, there are three problems with our symbolic values $\widehat a_i$ and  $\widehat n_i$ and the $T_6$ configuration they correspond to:
\begin{enumerate}
    \item We do not have precisely $\sum_i \widehat a_i \widehat n_i =0$ (see \Cref{not_quite_loop}). However, $\sum_i \widehat a_i \widehat n_i$ is very close to zero. \label{issue_no_loop}
    \item We do not have precisely \eqref{symmetry_cond}.\label{issue_sym}
    \item We do not have precisely \eqref{q_cond}.\label{issue_q}
\end{enumerate}

To fix issues \eqref{issue_no_loop} and \eqref{issue_sym} we do the following:

We first calculate values for the points in an ``approximate'' $T_6$ configuration, using the values of $\widehat a_i$, $\widehat n_i$, and $\widehat\kappa_i$, together with the definition of a $T_N$ configuration.  We then change the values of the resulting matrices so that \eqref{symmetry_cond} holds. This gives us matrices $\widehat X_1,\ldots,\widehat X_6$.

Now, we restrict ourselves to the first two rows of the matrices  $\widehat X_1,\ldots,\widehat X_6$, giving us matrices $\widetilde X_1,\ldots,\widetilde X_6\in \mathbb{R}^{2\times2}$.

We use the theory of Székelyhidi, which in the setting of $2\times2$ matrices, can determine when a set of $N$ matrices is in $T_N$ configuration. Moreover, if the matrices are indeed in $T_N$ configuration, then the theory can also provide the parameterization of the $T_N$ (i.e., the $(P,C_i,\kappa_i)$).

We can then, symbolically in MATLAB, use the theorem  of Székelyhidi \cite{Szekelyhidi2005RankOneHulls} to confirm that these matrices $\widetilde X_1,\ldots,\widetilde X_6$ correspond to a true $T_6$ configuration in $\mathbb{R}^{2\times 2}$ and the formulas from \cite{Szekelyhidi2005RankOneHulls} give us the corresponding parametrization
\begin{align}
(\widetilde P,\widetilde a_1,\ldots,\widetilde a_6,\widetilde n_1,\ldots,\widetilde n_6,\widetilde\kappa_1,\ldots,\widetilde\kappa_6),
\end{align}
where in this context the $\widetilde a_i$ are only $2\times 1$ column vectors.

Now, to rebuild the third row of the $T_6$ configuration, we use the values from the third entries of the $\widehat a_i$ as an ansatz for missing third rows of the $\widetilde a_i$: we use the values $(\widehat a_1)_{3,1},\ldots, (\widehat a_6)_{3,1}$ but with a slight linear-algebra correction to ensure that  \eqref{eq:TN-closure} holds. This process gives us $3\times 1$ column vectors which we label the $a_i$, and we then get a parameterization of a true $T_6$ in the set of $3\times 2$ matrices, which we   label
\begin{align}
    (\widetilde P,a_1,\ldots,a_6,\widetilde n_1,\ldots,\widetilde n_6,\widetilde\kappa_1,\ldots,\widetilde\kappa_6).
\end{align}

Finally, to resolve issue \eqref{issue_q} we perturb the values of the $\widetilde\kappa_i$ to new values $\kappa_i$ while keeping the rest of the $T_6$ parametrization fixed. More precisely, the new value $\kappa_i$ is a root of a quadratic polynomial which corresponds exactly to \eqref{q_cond}. This is implemented symbolically in MATLAB, and the root which causes the smallest perturbation in the resulting $T_6$ configuration is chosen as our $\kappa_i$ value.

This gives us the final $T_6$ configuration, $(X_1,\ldots,X_6)$. Remark that we have rank-one matrices $a_i\widetilde n_i$ which are symmetric and thus \eqref{symmetry_cond} is automatically satisfied.

We also see that \eqref{inequalities_for_convexity_extension_lemma} are satisfied. This is because when we find the original numerical approximations for the $T_6$ solution ($\widehat a_i$,  $\widehat n_i$ and $\widehat P$), we  can ask the MATLAB to find solutions which are very far away from satisfying \eqref{inequalities_for_convexity_extension_lemma} as exact equalities, so even with some perturbations to the $T_6$ parametrization  \eqref{inequalities_for_convexity_extension_lemma} will still be satisfied.

\begin{figure}[tb]
      \includegraphics[width=.7\textwidth]{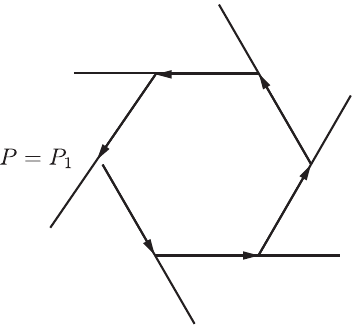}\hspace{1in}
  \caption{The MATLAB does \emph{not} return values for $\widehat a_i$ and $\widehat n_i$ which make a rank-one polygon. I.e. $\sum_i \widehat a_i \widehat n_i \neq 0$.}\label{not_quite_loop}
\end{figure}

We remark that we do not necessarily have to fix the values of the $\widehat\kappa_i$ when we run the approximate numerical solver. The difficulty is that the MATLAB solver can make the $\widehat\kappa_i$ very large. Very small perturbations of the $\widehat a_i$ or $\widehat n_i$, which are needed above to enforce the closure constraint, can then produce very large changes in the matrices $\widehat X_i$ and destroy \eqref{inequalities_for_convexity_extension_lemma}. In fact, when the solver is allowed to vary the $\widehat\kappa_i$, it appears to approach the closure constraint by making the $\widehat\kappa_i$ very large; this permits large variation in the $\widehat X_i$ while keeping $\sum_i \widehat a_i\widehat n_i$ close to zero.

We remark also that solving a nonlinear system of inequalities using MATLAB with fmincon and the interior-point algorithm,  gives different solutions depending on the initial point at which the solver starts. Our MATLAB code picks a random initial point for the solver each time the code is run. We had to try many initial points in order to find a satisfactory $T_N$ with $N=6$.

We provide the approximate $T_N$ configuration in the appendix \Cref{sec:T6-example} (the $\widehat a_i$, $\widehat n_i$, $\widehat\kappa_i$, and $\widehat P$).

We refer to the associated MATLAB scripts \path{Numerical_Search_T_N.m} and \path{Perturb_T_N_Symbolic.m}.
\end{proof}

Finally, we have

\begin{corollary}\label{T_6_with_ordering}
We can find an exact $T_6$ verifying all of the  properties \eqref{inequalities_for_convexity_extension_lemma}, \eqref{symmetry_cond}, \eqref{q_cond} of the $T_6$ from \Cref{matlab_lemma} but in a simple form with rational numbers or roots of simple quadratic expressions (the quadratic expressions come from forcing \eqref{q_cond} to hold).
\end{corollary}
See \Cref{1-shock_case_pic}.
\begin{proof}
This comes from ChatGPT after giving it the $T_6$ from \Cref{matlab_lemma}. See the code \path{T6_full_data.m} and \path{check_requirements_T_6_Sol.m}.
\end{proof}

\section{Extension to $F$, $\eta$, and $q$}\label{extension_section}

We want to find a pressure function $p$ such that the system \eqref{eq:p-system} with this pressure law admits the $T_6$ configuration from  \Cref{matlab_lemma} (or \Cref{T_6_with_ordering}).

With computer assistance we will actually be able to construct a pressure function $p$ for which we have very precise quantitative control and furthermore for this $p$ the set $K_{F,\eta,q}$ will contain the $T_6$ from \Cref{matlab_lemma} (or \Cref{T_6_with_ordering}). 

\begin{lemma}[A construction of the pressure function]\label{pressure_construction}
There exists a smooth pressure function $p\colon\mathbb{R}\to\mathbb{R}$ such that the $T_6$ configuration from \Cref{matlab_lemma} (or \Cref{sec:T6-example} or \Cref{T_6_with_ordering}) is contained in the constitutive set $K_{F,\eta,q}$ for this particular $p$. Moreover,  $p'<0$ on $\mathbb{R}$.

\end{lemma}
\begin{remark}
Our result gives some control on the convexity of $p$ we do not use in our current arguments in this paper but we give these details here for future work and to better understand the geometry of the situation we are in. In particular, $p$ can be chosen convex  on one of the intervals $((X_2)_{1,1},(X_4)_{1,1})$ or $((X_1)_{1,1},(X_3)_{1,1})$ or concave on one of 
\begin{align}
((X_4)_{1,1},(X_5)_{1,1}), \\((X_5)_{1,1},(X_1)_{1,1}), \\((X_3)_{1,1},(X_6)_{1,1}).
\end{align}

 We can use the result \Cref{thm:increasing-concave-convex-prescribed-integral} to show that we can modify $p$ such that it is convex or concave on the relevant domains.
A symbolic verification in MATLAB shows that \Cref{thm:increasing-concave-convex-prescribed-integral} applies and gives a strictly concave or convex function $f$ (playing the role of $-p$) on the various domains in $v$-space as stated in \Cref{pressure_construction}. The MATLAB scripts \path{test_for_concavity_convexity.m} and \path{checkIncreasingShape.m} check the required inequalities.
\end{remark}
\begin{proof}

By \eqref{inequalities_for_convexity_extension_lemma} and \Cref{convex_function_lemma}, we have the existence of a pressure function $p$ such that the $T_6$ from \Cref{matlab_lemma} (or \Cref{T_6_with_ordering}) is in the constitutive set $K_{F,\eta,q}$ for this  $p$. Moreover, $-\int p$ is strictly convex, which ensures that $p'<0$.
\end{proof}

\section{Stable embedding of $T_N$}\label{perturbative_section}

In this section, we show that if, for a given flux $F$ and entropy pair $(\eta,q)$, the product manifold $\mathcal{K}_{F,\eta,q}$ contains the particular $T_N$ configuration constructed above, then the pressure law can be perturbed so that Condition~(C) holds at the same $T_N$ configuration.

This section is dedicated to the following lemma.

\begin{lemma}\label{transverse_check_lemma}
Let $z^0=(X_1,\ldots,X_N)\in(\mathbb{R}^{3\times2})^N$ be the $T_N$ configuration from \Cref{matlab_lemma} (or \Cref{T_6_with_ordering}), let $\mathcal{M}_N$ be the local manifold given in \Cref{prop:parameter-manifold,thm:TN-tangent-space}, and let $\pi_k(Z_1,\ldots,Z_N)=P_k$ be as in \eqref{projection_maps}. Let $p$ be the pressure law from \Cref{pressure_construction}. Given any $\epsilon>0$ and $n\in\mathbb{N}$, there exists a pressure law $\widetilde p$ such that
\begin{equation}
    \lVert p-\widetilde p\rVert_{C^n(\mathbb{R})}\leq\epsilon,
\end{equation}
and $\widetilde p$ continues to satisfy the properties in \Cref{pressure_construction}. Let $\widetilde F$, $\widetilde\eta$, and $\widetilde q$ denote the flux, entropy, and entropy flux induced by $\widetilde p$ through \eqref{eq:p-system-vector-form} and \eqref{eq:p-system-natural-entropy}. Then for this $\widetilde F$, $\widetilde\eta$, and $\widetilde q$ we will be in the non-degenerate situation given by Condition (C) in the special case of the constitutive set \eqref{eq:constitutive-surface} with the symmetry $X_{2,1}=X_{1,2}$ for all $X\in K_{\widetilde F, \widetilde \eta, \widetilde q}$.

\end{lemma}

Remark that for general systems \eqref{eq:system}, Condition (C) is equivalent to 
\begin{equation}\label{lemma_tranverse_condition}
T_{z^0}\mathcal{K}_{F,\eta,q}
+\ker D\pi_k(z^0)
=(\mathbb{R}^{3\times2})^N
\end{equation}
for $k=1,\ldots,N$.

Let us briefly explain why \eqref{lemma_tranverse_condition} implies Condition~(C) from \Cref{condition_c_def}. We follow the argument in \cite[p.~146]{MR2048569}. At a point $z\in\mathcal{K}_{F,\eta,q}\cap\mathcal{M}_N$, Condition~(C) requires
\begin{equation}
\begin{aligned}\label{condition_c_rewrite919}
T_z\mathcal{K}_{F,\eta,q}+T_z\mathcal{M}_N
&=(\mathbb{R}^{3\times2})^N,\\
\dim D\pi_k(z)\bigl[T_z\mathcal{K}_{F,\eta,q}\cap T_z\mathcal{M}_N\bigr]
&=6,
\qquad k=1,\ldots,N.
\end{aligned}
\end{equation}
Recall that $\dim\mathcal{K}_{F,\eta,q}=2N$, $\dim\mathcal{M}_N=5N$, and $\ker D\pi_k(z)\subset T_z\mathcal{M}_N$ has dimension $5N-6$. Hence
\begin{equation}
T_z\mathcal{K}_{F,\eta,q}+\ker D\pi_k(z)
=(\mathbb{R}^{3\times2})^N
\end{equation}
implies the first line of \eqref{condition_c_rewrite919}. It also gives
\begin{equation}
\dim\bigl(T_z\mathcal{K}_{F,\eta,q}\cap\ker D\pi_k(z)\bigr)
=2N+(5N-6)-6N=N-6.
\end{equation}
Because $\ker D\pi_k(z)\subset T_z\mathcal{M}_N$, rank--nullity on the $N$-dimensional transverse intersection then gives the second line of \eqref{condition_c_rewrite919}.

We remark that due to the symmetry of the constitutive set of the $p$-system, we restrict to the space of matrices $A$ with  have $A_{2,1}=A_{1,2}$ and find
\begin{align}\label{from_symmetry}
\dim \textrm{im} D\pi_k(z) = 5.
\end{align}

We now prove \Cref{transverse_check_lemma}.
\begin{proof}[Proof of \Cref{transverse_check_lemma}]
For $N=6$, the proof gives a candidate basis for $T_{z^0}\mathcal{K}_{\widetilde F,\widetilde\eta,\widetilde q}$ and a candidate basis for $\ker D\pi_k(z^0)$. These provide $36$ elements of $(\mathbb{R}^{3\times2})^N$. We show that a generic perturbation of the pressure law makes 35 of these elements linearly independent and therefore a basis of the ambient space, including  the symmetry $X_{2,1}=X_{1,2}$ for all $X\in K_{\widetilde F, \widetilde \eta, \widetilde q}$.

We first consider $k=1$. Let $C_1^0,\ldots,C_N^0$ be the rank-one matrices corresponding to the $T_N$ configuration $(X_1,\ldots,X_N)$, and let $\kappa_i^0$ be the associated scalars.

\uline{Step 1}

Following the remark in \cite[p.~148]{MR2048569}, as well as \Cref{thm:TN-tangent-space,subsec:explicit-TN-basis}, the space $\ker D\pi_1(z^0)$ is in general for the system \eqref{eq:system} a $(5N-6)$-dimensional vector space. Its elements are $N$-tuples $(\dot X_1,\ldots,\dot X_N)$ of the form
\begin{equation}\label{form_kernel_pi}
\dot X_i
=\sum_{j=1}^{i-1}\bigl(a_j^0\dot n_j+\dot a_j n_j^0\bigr)
+\kappa_i^0 a_i^0\dot n_i
+\kappa_i^0\dot a_i n_i^0
+\dot\kappa_i a_i^0 n_i^0,
\end{equation}
where $C_i^0=a_i^0n_i^0$, with $a_i^0,\dot a_i\in\mathbb{R}^{3\times1}$ and $n_i^0,\dot n_i\in\mathbb{R}^{1\times2}$, and the variations $(\dot a_i,\dot n_i,\dot\kappa_i)$ satisfy the linearized closure constraint
\begin{equation}\label{basis_constraint}
\sum_{i=1}^N\bigl(\dot a_i n_i^0+a_i^0\dot n_i\bigr)=0.
\end{equation}
By convention, $\sum_{j=1}^{0}:=0$. The explicit tangent-space basis constructed in \Cref{subsec:explicit-TN-basis} gives a basis of $\ker D\pi_1(z^0)$ after deleting the six basis vectors corresponding to variations of the base matrix $P$. This leaves $5N-6$ basis vectors.

\uline{Step 2}

For each $i$, define the state
\begin{equation}
U_i:=\bigl((X_i)_{1,1},(X_i)_{2,1}\bigr)^{\mathsf T}\in\mathbb{R}^2.
\end{equation}
For a candidate perturbation $\widetilde p$, let $\widetilde G$ be the constitutive map associated with $\widetilde F$, $\widetilde\eta$, and $\widetilde q$. With the column-gradient convention fixed after \eqref{eq:entropy-compatibility},
\begin{equation}\label{eq:DG-constitutive-map}
D\widetilde G(U_i)[V_i]
=
\begin{pmatrix}
V_i & -D\widetilde F(U_i)V_i\\
\nabla\widetilde\eta(U_i)\cdot V_i & -\nabla\widetilde q(U_i)\cdot V_i
\end{pmatrix},
\qquad V_i\in\mathbb{R}^2.
\end{equation}
Consequently,
\begin{equation}\label{form_tangent922}
T_{z^0}\mathcal{K}_{\widetilde F,\widetilde\eta,\widetilde q}
=
\left\{
\bigl(D\widetilde G(U_1)[V_1],\ldots,D\widetilde G(U_N)[V_N]\bigr)
:
V_1,\ldots,V_N\in\mathbb{R}^2
\right\}.
\end{equation}
Taking successively each standard coordinate vector in $(\mathbb{R}^2)^N$ gives a basis of this $2N$-dimensional tangent space.

For the perturbed $p$-system, write $U=(v,u)^{\mathsf T}$. Then
\begin{equation}
\widetilde F(U)=\begin{pmatrix}-u\\\widetilde p(v)\end{pmatrix},
\qquad
D\widetilde F(U)=\begin{pmatrix}0&-1\\\widetilde p'(v)&0\end{pmatrix},
\end{equation}
and
\begin{equation}
\nabla\widetilde\eta(U)=\begin{pmatrix}-\widetilde p(v)\\u\end{pmatrix},
\qquad
\nabla\widetilde q(U)=D\widetilde F(U)^{\mathsf T}\nabla\widetilde\eta(U).
\end{equation}
Thus all entries in \eqref{eq:DG-constitutive-map} are determined by the values encoded in the $T_N$ configuration together with the numbers $\widetilde p'((X_i)_{1,1})$.

\uline{Step 3}

Identify $(\mathbb{R}^{3\times2})^N$ with $\mathbb{R}^{6N}$. The $5N-6$ potential basis vectors of $\ker D\pi_1(z^0)$ and the $2N$ basis vectors of $T_{z^0}\mathcal{K}_{\widetilde F,\widetilde\eta,\widetilde q}$ form the columns of a $6N\times(7N-6)$ matrix, which is square when $N=6$.

Due to the symmetry of the set \eqref{eq:constitutive-surface} for the $p$-system we expect 
\begin{align}\label{lemma_tranverse_condition_psystem}
\dim\big(T_{z^0}\mathcal{K}_{F,\eta,q}
+\ker D\pi_k(z^0)\big)=35,
\end{align}
which means we expect there to be a $35\times35$ submatrix whose determinant is nonzero (see \eqref{from_symmetry}).

We now show \eqref{lemma_tranverse_condition} for $k=1$.

We construct this $36\times 36$ matrix in MATLAB using exact symbolic computations and the values of $a_i^0$, $n_i^0$, and $\kappa_i^0$ determined in \Cref{matlab_lemma} (or \Cref{T_6_with_ordering}). For a selected $35\times 35$ submatrix, its determinant is a polynomial $Q_1$ in the numbers $\widetilde p'((X_i)_{1,1})$. This polynomial is not identically zero. For example, setting
\begin{equation}
D\widetilde F(U_i)=
\begin{pmatrix}
0&-1\\
-1&0
\end{pmatrix},
\qquad i=1,\ldots,N,
\end{equation}
which corresponds to $\widetilde p'((X_i)_{1,1})=-1$, gives a nonzero symbolic determinant. This polynomial argument follows \cite[p.~733]{MR1983780}.

For $k>1$, any $P_k$ can be used as the first base point in the cyclic parameterization of a $T_N$ configuration. Applying the same MATLAB computation to the corresponding cyclic relabeling produces a nonzero polynomial $Q_k$. Hence $Q:=Q_1\cdots Q_N$ is not identically zero. A generic perturbation of the finitely many derivative values $\widetilde p'((X_i)_{1,1})$, made while preserving the interpolation constraints from \Cref{pressure_construction}, therefore yields \eqref{lemma_tranverse_condition} simultaneously for every $k$.

It follows that one can choose $\widetilde p$ with
\begin{equation}
\lVert p-\widetilde p\rVert_{C^n(\mathbb{R})}\leq\epsilon
\end{equation}
while retaining all the properties in \Cref{pressure_construction}. The exact symbolic computation is implemented in \path{make_big_basis.m}.
\end{proof}

\section{Positive side of Main Theorem: $L^2$ Stability}\label{sec:pos}
In this section, we prove the stability estimate \eqref{eq:main-stability-estimate} which holds under the Strong Trace Property.

First, we develop the ``shift'' function in \Cref{sec:shift} and then we use it to prove \eqref{eq:main-stability-estimate} by way of relative entropy dissipation estimates in \Cref{subsec:two-shock-riemann-specialization}

Before we do that, we first need to define the exact Riemann problem which we are using in \Cref{main_theorem}.

\subsection{Choice of Riemann problem in \Cref{main_theorem}}

The idea is to construct the Riemann problem initial data in \Cref{main_theorem} by finding states $A$ and $B$ very, very far to the right of the $T_6$ points from \Cref{T_6_with_ordering} (or \Cref{matlab_lemma} (``large-volume states''), such that  by choosing a properly decaying pressure law for the far field states, we can derive hyperbolic estimates for shocks which connect between the far field and the ``core'' around which the $T_6$ points are clustered. Our pressure function will not be globally convex. It will be convex in the far field, but not convex in the core near the $T_6$ points (as explained in \Cref{pressure_construction}). Nonetheless, \emph{the pressure function will verify estimates between shocks connecting points in the core to points in the far field that are exactly the same as the estimates in the case the pressure $p$ is globally convex}.

\subsubsection{Choice of pressure function}
For a pressure function $p$, define
\begin{align}
    h\coloneqq -p.
\end{align}
We define $H$ to be a primitive of $h$, i.e. $H'=h$.

We then consider the pressure function $p$ from \Cref{pressure_construction}, and we modify it away from the $T_6$ points such that 
\begin{align}
    h(v)&=h_6+1-\frac{1}{1+v-v_6}, &&v\ge v_6.\label{eq:righttail}
\end{align}
Thus $h'>0$ on $(0,\infty)$ and $h''<0$ on $[v_6,\infty)$. Recall that $X_6$ is the right-most (in state space) of the $T_6$ points from \Cref{T_6_with_ordering} (or also \Cref{matlab_lemma}).

The primitive on the right tail is explicitly
\begin{equation}\label{eq:Htail}
 H(v)=H(v_6)+(h(v_6)+1)(v-v_6)-\log(1+v-v_6),\qquad v\ge v_6.
\end{equation}

\subsection{Construction of Riemann data}

Write the six states of the $T_6$ as $X_i=G(v_i,u_i)$ (see \eqref{eq:constitutive-surface}). Now define the  quantities
\begin{equation}\label{eq:nodal}
 v_i=(X_i)_{11},\quad u_i=(X_i)_{21},\quad
 E_i=(X_i)_{31},\quad h_i=(X_i)_{22},\quad H_i=E_i-u_i^2/2.
\end{equation}
The pressure construction (see \Cref{pressure_construction}) will impose $h(v_i)=h_i$, $H(v_i)=H_i$.

Then $X_i$ belongs to the actual constitutive set  \eqref{eq:constitutive-surface}.

Use coordinates $(v,u,E,h,Q)$ on $\mathbb{M}$
\begin{equation}\label{eq:symmetry}
 (v,u,E,h,Q)\longleftrightarrow
 \begin{pmatrix}v&u\\u&h\\E&Q\end{pmatrix}.
\end{equation}

\subsection{The large-volume states and the relaxed fan}\label{sec:envelope}
Take the $P$ from the parameterization of the $T_N$ from  \Cref{T_6_with_ordering} (or \Cref{matlab_lemma}), and use its coordinates
$(\vbar,\ubar,\Ebar,\hmean,\Qbar)$. For a positive $b$ with very large magnitude, to be chosen later, we define
\begin{align}
 L&=b-\vbar, & G_b&=h(b)-\hmean,\label{eq:macro-parameters}\\
 s&=\sqrt{G_b/L},& d_u&=\sqrt{LG_b}=sL,&
 A&=(b,\ubar+d_u),\quad B=(b,\ubar-d_u).\label{eq:outer-states}
\end{align}
All values of the flux, entropy, and entropy-flux
at $b$ are given by the  far-field tail formulas, so these states do not depend on what is happening near the $T_N$ points.

\subsection{Both conservation laws match exactly}\label{sec:exact_match}

We need to find left and right states for our Riemann problem \eqref{main_theorem_Riemann} which \emph{also admit a subsolution fan}, which in this case, loosely speaking, means a Riemann problem which has ``relaxed'' self-similar solutions in the set of matrices $\mathbb{M}$ which are not in the constitutive set \eqref{eq:constitutive-surface} but are in its rank-one convex hull. See \Cref{sec:neg} below for more details.

The proposed relaxed subsolution fan is $G(A)$ connecting to $P$ and then $P$ connecting to $G(B)$, with interface speeds $-s,s$. Here the constant state $P$ is not corresponding to a solution to the conservation law, but it is a relaxed state which will be used in the convex integration process below.

To construct a Lipschitz function whose derivative is this subsolution, at least at the level of the first two rows corresponding to the conservation of mass, continuity along a line
$x=\sigma t$ means precisely that the difference of the two matrix gradients can be factorized into some two-dimensional vector as well as the two-dimensional vector $(\sigma,1)^{\mathsf T}$ . The upper two rows satisfy
\begin{align*}
 (P-G(A))_{1:2,:}
 &=\begin{pmatrix}-L&-d_u\\-d_u&-G_b\end{pmatrix},\\
 (G(B)-P)_{1:2,:}
 &=\begin{pmatrix}L&-d_u\\-d_u&G_b\end{pmatrix}.
\end{align*}
Since $d_u=sL$ and $G_b=sd_u$, these matrices have a factor of
$(-s,1)^{\mathsf T}$ and $(s,1)^{\mathsf T}$, respectively. These are the
two ``Rankine--Hugoniot conditions'' at each interface. The third rows need not
satisfy this factorization; the entropy-entropy flux row may have a sign as in \eqref{eq:entropy-inequality}.

\subsection{Calculation of the two entropy productions}\label{sec:entropy_calc}
Give the relaxed middle state $P$ the entropy and entropy flux
$(\Ebar,-\Qbar)$. Define
\begin{align}
 \mathscr D&=H(b)-\left(\Ebar-\frac{\ubar^2}{2}\right)
                  -\frac L2\bigl(h(b)+\hmean\bigr),\label{eq:macro-D}\\
 \mathscr J&=\Qbar-\ubar\hmean.\label{eq:macro-J}
\end{align}
With $\mathcal P=[q]-\sigma[\eta]$, the left production is
\begin{align*}
 \mathcal P_L
 &=-\Qbar+(\ubar+d_u)h(b)
        +s\left[\Ebar-\frac{(\ubar+d_u)^2}{2}-H(b)\right]\\
 &=-\mathscr J+\ubar\bigl(h(b)-\hmean-sd_u\bigr)
       +d_u h(b)
       +s\left[\Ebar-\frac{\ubar^2}{2}-H(b)-\frac{d_u^2}{2}\right]\\
 &=-\mathscr J+s\left[Lh(b)+\Ebar-\frac{\ubar^2}{2}
                        -H(b)-\frac{L(h(b)-\hmean)}2\right]\\
 &=-s\mathscr D-\mathscr J.
\end{align*}
At the right interface,
\begin{align*}
 \mathcal P_R
 &=-(\ubar-d_u)h(b)+\Qbar
        -s\left[\frac{(\ubar-d_u)^2}{2}+H(b)-\Ebar\right]\\
 &=\mathscr J+\ubar\bigl(\hmean-h(b)+sd_u\bigr)
       +d_u h(b)
       -s\left[\frac{\ubar^2}{2}+\frac{d_u^2}{2}+H(b)-\Ebar\right]\\
 &=-s\mathscr D+\mathscr J.
\end{align*}
Consequently
\begin{equation}\label{eq:outer-productions}
 \boxed{\mathcal P_L=-s\mathscr D-\mathscr J,
 \qquad \mathcal P_R=-s\mathscr D+\mathscr J.}
\end{equation}
There is no assumption that $P$ itself is constitutive in this calculation.

Then, we can take $b$ large enough to ensure that the discontinuities in our subsolution fan dissipate entropy with the correct sign.

Let us explain this.
\subsubsection{Limits of the quantities $\mathcal P_L$ and $\mathcal P_R$ as \(b\to\infty\)}

Here we keep the $T_6$,  $P$, and the pressure function
fixed, and let the outer volume $b$ increase.

The two quantities are
\begin{equation}
  \mathscr D(b)
  =H(b)-\left(\bar E-\frac{\bar u^2}{2}\right)
  -\frac{b-\bar v}{2}\bigl(h(b)+\bar h\bigr),
\end{equation}
and
\begin{equation}
  \mathscr J=\bar Q-\bar u\bar h.
\end{equation}

More precisely, $\mathscr D(b)$ grows linearly in $b$. We derive
its leading coefficient and then explain the consequence for the two
entropy productions.

\subsubsection{The pressure tail and its primitive}

For $v\ge v_6$,
\[
  h(v)=h_6+1-\frac{1}{1+v-v_6},
\]
\[
  H(v)=H_6+(h_6+1)(v-v_6)-\log(1+v-v_6).
\]
The additive constant of $H$ is fixed by $H(v_6)=H_6$.

Define
\[
  h_\infty:=h_6+1.
\]
Then
\[
  h(b)=h_\infty-\frac{1}{1+b-v_6}
  \longrightarrow h_\infty.
\]
Also,
\[
  \frac{H(b)}{b}
  =\frac{H_6}{b}
   +h_\infty\left(1-\frac{v_6}{b}\right)
   -\frac{\log(1+b-v_6)}{b}.
\]
Since $H_6$, $v_6$, and $h_\infty$ are fixed,
\[
  \frac{H_6}{b}\longrightarrow0,
  \qquad
  \frac{v_6}{b}\longrightarrow0,
  \qquad
  \frac{\log(1+b-v_6)}{b}\longrightarrow0.
\]
Consequently,
\[
  \boxed{\frac{H(b)}{b}\longrightarrow h_\infty.}
\]
Thus $h(b)$ approaches a finite constant, but its primitive $H(b)$
grows linearly.

\subsubsection{The limit of $\mathscr D(b)$}

Divide the expression for $\mathscr D(b)$ by $b$:
\[
  \frac{\mathscr D(b)}{b}
  =\frac{H(b)}{b}
   -\frac{\bar E-\bar u^2/2}{b}
   -\frac12\left(1-\frac{\bar v}{b}\right)
     \bigl(h(b)+\bar h\bigr).
\]
Taking limits term by term, the first term tends to $h_\infty$, the
second tends to zero, and the last tends to
\[
  -\frac12(h_\infty+\bar h).
\]
Therefore,
\begin{equation}\label{eq:D-linear-limit}
  \boxed{
    \lim_{b\to\infty}\frac{\mathscr D(b)}{b}
    =h_\infty-\frac{h_\infty+\bar h}{2}
    =\frac{h_\infty-\bar h}{2}.
  }
\end{equation}

\subsubsection{Why this coefficient is positive}

The matrix $P$ is a convex combination of the six vertices with
strictly positive coefficients:
\[
  P=\sum_{i=1}^6\mu_iX_i,
  \qquad
  \mu_i>0,
  \qquad
  \sum_{i=1}^6\mu_i=1,
\]
see \Cref{folklore_lemma}.
Taking the $(2,2)$ entry gives
\[
  \bar h=\sum_{i=1}^6\mu_i h_i.
\]
Since $h$ is strictly increasing and $v_6$ is the largest
volume in the $T_6$, $h_6$ is the largest volume value in the $T_6$. The other $T_6$ volumes are
strictly smaller, so
\[
  \bar h<h_6.
\]
Indeed, the volume ordering is
\[
  v_2<v_4<v_5<v_1<v_3<v_6.
\]
Hence
\[
  h_\infty-\bar h=h_6+1-\bar h>0.
\]
In fact, for this particular tail it is greater than $1$.

Equation~\eqref{eq:D-linear-limit} therefore implies
\[
  \boxed{
    \mathscr D(b)
    \sim\frac{h_\infty-\bar h}{2}\,b
    \longrightarrow+\infty.
  }
\]

\subsubsection{A sharper expansion}

Substitution of the explicit tail formulas gives
\[
  \mathscr D(b)
  =\frac{h_\infty-\bar h}{2}\,b
   -\log(1+b-v_6)
   +C_0
   +\frac{b-\bar v}{2(1+b-v_6)},
\]
where the $b$-independent constant is
\[
  C_0
  =H_6-h_\infty v_6-\bar E+\frac{\bar u^2}{2}
   +\frac{\bar v}{2}(h_\infty+\bar h).
\]
Because
\[
  \log(1+b-v_6)=\log b+O(b^{-1}),
  \qquad
  \frac{b-\bar v}{2(1+b-v_6)}=\frac12+O(b^{-1}),
\]
we obtain
\begin{equation}\label{eq:D-sharp-expansion}
  \boxed{
    \mathscr D(b)
    =\frac{h_\infty-\bar h}{2}\,b
     -\log b+C_0+\frac12+O(b^{-1}).
  }
\end{equation}
Thus there is a negative logarithmic correction, but it is much
smaller than the positive linear term.

\subsubsection{Consequences for the outer entropy productions}

Consider now
\[
  \mathcal P_L(b)=-s(b)\mathscr D(b)-\mathscr J,
  \qquad
  \mathcal P_R(b)=-s(b)\mathscr D(b)+\mathscr J,
\]
with
\[
  s(b)=\sqrt{\frac{h(b)-\bar h}{b-\bar v}}.
\]
Put
\[
  \Delta:=h_\infty-\bar h>0.
\]
Then
\[
  s(b)\sqrt b
  =\sqrt{\frac{h(b)-\bar h}{1-\bar v/b}}
  \longrightarrow\sqrt\Delta.
\]
Together with $\mathscr D(b)/b\to\Delta/2$, this gives
\[
  \frac{s(b)\mathscr D(b)}{\sqrt b}
  =\bigl(s(b)\sqrt b\bigr)\frac{\mathscr D(b)}{b}
  \longrightarrow\frac{\Delta^{3/2}}{2}.
\]
Thus, the product $s(b)\mathscr D(b)$ tends to
$+\infty$:
\[
  s(b)\mathscr D(b)
  \sim\frac{\Delta^{3/2}}{2}\sqrt b.
\]

Since $\mathscr J$ stays fixed,
\[
  \boxed{
    \mathcal P_L(b)\longrightarrow-\infty,
    \qquad
    \mathcal P_R(b)\longrightarrow-\infty.
  }
\]

This is why choosing $b$ sufficiently large works: eventually
\[
  s(b)\mathscr D(b)>|\mathscr J|,
\]
so both interface productions are strictly negative, regardless of
the sign of the fixed $\mathscr J$.

\subsection{Riemann data}\label{sec:def_Riemann_data}
Thus, we choose for our Riemann data in \eqref{main_theorem_Riemann}
\begin{align}\label{riemann_data_def}
    &U_L\coloneqq A\\
    &U_R \coloneqq B,
\end{align}
and the \emph{classical} Riemann solution with this data will have a uniquely determined middle state which is connected to $U_L$ by a 1-shock and is connected to $U_R$ by a 2-shock. We define this middle state to be $U_M$. Remark that $U_M$
 is uniquely determined due to the monotonicity of the shock curves as well as the fact that we can put a lower bound on $p'$ for $v<(X_6)_{1,1}$ (see \Cref{sec:mono_2shock} for 2-shock curves, the proof is the same for 1-shocks). See \Cref{1-shock_case_pic}.

\subsection{Construction of the shift}\label{sec:shift}

In this section, we prove two propositions which give us the shift functions to shift the 1-shock and 2-shock coming from the solution to the Riemann problem in \Cref{main_theorem}.

\begin{proposition}[Existence of the shift function for a single shock]\label{prop:4.1_one_shock}
Fix \(T>0\). Assume \(U\) is a bounded weak solution to \eqref{eq:p-system}. Assume
\(U\) is entropic for the entropy \(\eta\), and \(U\) has strong traces
(\Cref{strong_trace_def}). We consider two cases:
\begin{itemize}
    \item \((U_L,U_M,\sigma)\) is the 1-shock from \Cref{sec:def_Riemann_data}
\end{itemize}
 \textbf{OR}
    \begin{itemize}
    \item \((U_M,U_R,\sigma)\) is the 2-shock from \Cref{sec:def_Riemann_data}.
\end{itemize}

Then, there exists a constant \(a>0\) and a Lipschitz continuous map
\(h\colon[0,T)\to\R\) with \(h(0)=0\) and such that for almost every \(t\),
\begin{equation}
\begin{aligned}
&a\bigl(q(U_+;U_M)-\dot h(t)\eta(U_+\mid U_M)\bigr)
-q(U_-; U_L)+\dot h(t)\eta(U_-\mid U_L)\\
&\hspace{13em}\leq-c\left|\sigma-\dot h(t)\right|^2,
\end{aligned}
\label{eq:4.2}
\end{equation}
for the 1-shock case,
and 
\begin{equation}
\begin{aligned}
&a\bigl(q(U_+;U_R)-\dot h(t)\eta(U_+\mid U_R)\bigr)
-q(U_-; U_M)+\dot h(t)\eta(U_-\mid U_M)\\
&\hspace{13em}\leq-c\left|\sigma-\dot h(t)\right|^2,
\end{aligned}
\label{eq:4.2_shock}
\end{equation}
for the 2-shock case, where \(U_\pm:=U(h(t)\pm,t)\). Here $c>0$ is a constant.

In the 1-shock case, $0<a<1$ and $\dot{h}<0$. In the 2-shock case, $a>1$ and $\dot{h}>0$.
\end{proposition}

The proof of \Cref{prop:4.2} is in \Cref{sec:proof4.2}, below. The proof of Proposition~\ref{prop:4.1_one_shock} uses \Cref{prop:4.2}, introduced just below. \Cref{prop:4.2} uses only the controlled shock family based near
a distinguished endpoint; in the case of the $p$-system \eqref{eq:p-system} it does not require global convexity of the
pressure. 

Write
\[
 \mathcal Q(X;Y,s)=q(X;Y)-s\eta(X\mid Y).
\]

\begin{proposition}[Localized algebraic contraction]\label{prop:4.2}
Consider the system \eqref{eq:system}.

Let $\mathcal B$ be a fixed compact allowed state set, contained in an open
convex domain on which $F,\eta,q$ are smooth, $\nabla q(U)
=
DF(U)^{\mathsf T}\nabla\eta(U)$, and
the entropy Hessian is positive definite. Fix $R_*>0$ and a  1-shock
$(U_L,U_R)$ 
with $U_R=S^1_{U_L}(r_0)$, $0<r_0<R_*$, of speed $\sigma_0$.
Suppose $S^1_Z(r),\sigma^1_Z(r)$ are a smooth family issuing from the point $Z$ near $U_L$ and
$0\le r\le R_*$, with all states in a common compact subset of that domain, and 
\begin{align}
 S^1_Z(0)=Z,\quad \sigma^1_Z(0)=\lambda_1(Z),\quad
 F(S^1_Z(r))-F(Z)=\sigma^1_Z(r)(S_Z(r)-Z),\label{need1}
 \end{align}

\begin{align}
 \frac{d}{dr}\sigma^{1}_{Z}(r)<0\quad(0\le r\le R_*),\qquad
 \partial_r\eta(Z\mid S_Z(r))>0\quad(0<r\le R_*).\label{need2}
\end{align}
For all sufficiently small $a>0$, there exist $c_a,\delta_a>0$ such that,
whenever $Z\in\mathcal B$ and
$\eta(Z\mid U_L)\le a\eta(Z\mid U_R)$,
\begin{align}
 a\mathcal Q(S_Z(r);U_R,\sigma^1_Z(r))
 -\mathcal Q(Z;U_L,\sigma^1_Z(r))
 &\le-c_a|\sigma^1_Z(r)-\sigma_0|^2,
 \quad 0\le r\le R_*,\label{eq:4.3}\\
 a\mathcal Q(Z;U_R,\lambda_1(Z))
 -\mathcal Q(Z;U_L,\lambda_1(Z))
 &\le-\delta_a.\label{eq:4.4}
\end{align}
The constants are uniform on sufficiently small compact families of reference shocks $(U_L,U_R)$. The reflected statement applies to incoming
second-family curves, with the reciprocal weight (which will be greater than 1 in this case).
\end{proposition}

\begin{remark}
The point of this proposition is to show that the $a$-contraction theory \cite{MR3519973} works even if we don't have the global properties normally assumed on the hyperbolic system in the $a$-contraction theory, such as genuine nonlinearity. We only need \eqref{need1} and \eqref{need2} locally at the base state $U_L$, which we will have in our case for the Riemann data in \Cref{main_theorem}.
\end{remark}

The proof of \Cref{prop:4.2} is in \Cref{sec:proof4.2}.

\subsubsection{Proof of \Cref{prop:4.1_one_shock}}

We will use the argument from \cite{move_entire_solution_system}.

We will focus on the case of the 1-shock. The 2-shock case is similar. 

We will use \Cref{prop:4.2}. Let us comment on why we can use \Cref{prop:4.2}

\subsubsection{Showing we have \eqref{need1} and \eqref{need2}}

We note that \eqref{need1} follows immediately from the geometry of the Hugoniot locus (see \Cref{Hugoniot_facts}) independent of the sign of $p''$ and \eqref{need2} follows from \Cref{subsec:p-system-relative-entropy} where this is shown for 2-shocks but the calculation is identical for 1-shocks. We also remark that the monotonicity of shock speed follows from \eqref{eq:p-system-speed-first-derivative} and \Cref{lem:remote-endpoint-control}.

We can now begin the proof of \Cref{prop:4.1_one_shock}.

\textbf{Proof of \Cref{prop:4.1_one_shock}}

Pick \(0<a<1\) as in  \Cref{prop:4.2}. 

Throughout this proof, \(c\) denotes a generic constant.

\medskip
\noindent\underline{Step 1}

We now show that for any \(\gamma_0>0\),
\begin{equation}
\inf \eta(U\mid U_L)-a\eta(U\mid U_M)\geq c_4\gamma_0^2
\label{eq:4.103}
\end{equation}
for a constant \(c_4>0\), where the infimum runs over all \((U,U_L,U_M)\)
such that
\[
\dist\bigl(U,\{W\mid \eta(W\mid U_L)\leq a\eta(W\mid U_M)\}\bigr)
\geq\gamma_0
\]
and $U_L,U_M\in\mathcal{V}$. The
distance \(\dist(W,\mathcal A)\) between a point \(W\) and a set \(\mathcal A\) is defined in
the usual way,
\begin{equation}
\dist(W,\mathcal A):=\inf_{Y\in\mathcal A}|W-Y|.
\label{eq:4.104}
\end{equation}

Consider any triple \((U,U_L,U_M)\) such that
\[
\dist\bigl(U,\{W\mid \eta(W\mid U_L)\leq a\eta(W\mid U_M)\}\bigr)
\geq\gamma_0
\]
and $U_L,U_M\in\mathcal{V}$.

By  \Cref{size_pi}, the set
\(\{W\mid\eta(W\mid U_L)\leq a\eta(W\mid U_M)\}\) is compact. Thus, there
exists \(W_0\) in this set such that
\begin{equation}
|U-W_0|=\dist\bigl(U,\{W\mid\eta(W\mid U_L)\leq a\eta(W\mid U_M)\}\bigr).
\label{eq:4.105}
\end{equation}

We Taylor expand the function
\begin{equation}
\Gamma(U):=\eta(U\mid U_L)-a\eta(U\mid U_M)
\label{eq:4.106}
\end{equation}
around the point \(W_0\):
\begin{equation}
\begin{aligned}
\Gamma(U)={}&\Gamma(W_0)+\nabla\Gamma(W_0)\cdot(U-W_0)\\
&+\int_0^1(1-t)(U-W_0)^{\mathsf T}D^2\Gamma(W_0+t(U-W_0))(U-W_0)\,dt.
\end{aligned}
\label{eq:4.107}
\end{equation}

By definition of \(W_0\), we must have \(\Gamma(W_0)=0\) and
\(\nabla\Gamma(W_0)\cdot(U-W_0)\geq0\). Note that
\(D^2\Gamma=(1-a)D^2\eta\). Thus, by strict convexity of \(\eta\)
and because \(0<a<1\), we have \(D^2\Gamma\geq cI_2\) for some constant
\(c>0\).

We then calculate,
\begin{align}
&\int_0^1(1-t)(U-W_0)^{\mathsf T}D^2\Gamma(W_0+t(U-W_0))(U-W_0)\,dt
\label{eq:4.108}\\
&\quad\geq\int_0^{1/2}(1-t)(U-W_0)^{\mathsf T}
D^2\Gamma(W_0+t(U-W_0))(U-W_0)\,dt,
\label{eq:4.109}
\end{align}
where we have changed the limits of integration. Continuing,
\begin{equation}
\geq\frac{3c}{8}|U-W_0|^2\geq\frac{3c}{8}\gamma_0^2,
\label{eq:4.110}
\end{equation}
where the last inequality comes from
\(\dist(U,\{W\mid\eta(W\mid U_L)\leq a\eta(W\mid U_M)\})\geq\gamma_0\).
This proves \eqref{eq:4.103}.

We choose
\begin{equation}
\gamma_0:=\frac{c_1}{2L_*},
\label{eq:4.111}
\end{equation}
where \(c_1\) is from \Cref{prop:4.2} and \(L_*\) is the Lipschitz constant of
the map
\begin{equation}
(U,U_L,U_M)\longmapsto
a\bigl(q(U;U_M)-\lambda_1(U)\eta(U\mid U_M)\bigr)
-q(U;U_L)+\lambda_1(U)\eta(U\mid U_L).
\label{eq:4.112}
\end{equation}

\medskip
\noindent\underline{Step 2}

Define
\begin{equation}
\mathcal{W}(U,t):=\lambda_1(U)-C_*\one_{\{U\mid
a\eta(U\mid U_M)<\eta(U\mid U_L)\}}(U),
\label{eq:4.113}
\end{equation}
where \(C_*>0\) is a large constant, which we can pick to be
\begin{equation}
C_*:=\frac{1}{c_4\gamma_0^2}
\left(\sup_{U,U_L,U_M\in \mathcal{V}}|a q(U;U_M)-q(U;U_L)|+1\right)
+2\sup_{U\in \mathcal{V}}|\lambda_1(U)|,
\label{eq:4.114}
\end{equation}
where \(c_4\) is from \eqref{eq:4.103}.

We solve the following ODE in the sense of Filippov flows,
\begin{equation}
\begin{cases}
\dot h(t)=\mathcal{W}(U(h(t),t),t),\\
h(0)=0,
\end{cases}
\label{eq:4.115}
\end{equation}
The existence of such an \(h\) comes from the following lemma.

\begin{lemma}[Existence of Filippov flows]\label{lem:4.2}
Let \(\mathcal{W}(U,t)\colon\R^2\times[0,\infty)\to\R\) be bounded on
\(\R^2\times[0,\infty)\), upper semi-continuous in \(U\), and measurable in
\(t\). Let \(U\) be a bounded weak solution to \eqref{eq:system}, entropic for the
entropy \(\eta\). Assume also that \(U\) verifies the Strong Trace Property
(\Cref{strong_trace_def}). Let \(x_0\in\R\). Then we can solve
\begin{equation}
\begin{cases}
\dot h(t)=\mathcal{W}(U(h(t),t),t),\\
h(0)=x_0,
\end{cases}
\label{eq:4.116}
\end{equation}
in the Filippov sense. That is, there exists a Lipschitz function
\(h\colon[0,\infty)\to\R\) such that
\begin{align}
\Lip[h]&\leq\lVert\mathcal{W}\rVert_{L^\infty},
\label{eq:4.117}\\
h(0)&=x_0,
\label{eq:4.118}
\end{align}
and
\begin{equation}
\dot h(t)\in\left[\inf\mathcal{W},
\max\{\mathcal{W}(U_+,t),\mathcal{W}(U_-,t)\}\right],
\label{eq:4.119}
\end{equation}
for almost every \(t\), where \(U_\pm:=U(h(t)\pm,t)\), and where the infimum
is taken over the domain of \(\mathcal{W}\). Similarly, for almost every \(t\) such
that for \(t\) fixed \(\mathcal{W}(\cdot,t)\) is continuous at \(U_+\) and \(U_-\),
then
\begin{equation}
\dot h(t)\in I[\mathcal{W}(U_+,t),\mathcal{W}(U_-,t)],
\label{eq:4.120}
\end{equation}
where \(I[a,b]\) denotes the interval with endpoints \(a\) and \(b\).

Moreover, for almost every \(t\),
\begin{align}
F(U_+)-F(U_-)&=\dot h(U_+-U_-),
\label{eq:4.121}\\
q(U_+)-q(U_-)&\leq\dot h(\eta(U_+)-\eta(U_-)),
\label{eq:4.122}
\end{align}
which means that for almost every \(t\), either \((U_+,U_-,\dot h)\) is an
entropic shock (for \(\eta\)) or \(U_+=U_-\).
\end{lemma}

The proof of \eqref{eq:4.117}, \eqref{eq:4.118}, \eqref{eq:4.119}, and
\eqref{eq:4.120} is very similar to the proof of Proposition 1 in \cite{Leger2011}. We do not reproduce the proof here.

It is classical that \eqref{eq:4.121} and \eqref{eq:4.122} are true for any
Lipschitz continuous function \(h\colon[0,\infty)\to\R\) when \(U\) is BV. When
instead \(U\) is only known to verify the  Strong Trace Property (\Cref{strong_trace_def}), then
\eqref{eq:4.121} and \eqref{eq:4.122} are given in Lemma 6 in \cite{Leger2011}. We do not
prove \eqref{eq:4.121} and \eqref{eq:4.122} here; a proof can be found in the 
appendix in \cite{Leger2011}.

Note that \(\mathcal{W}\) (see \eqref{eq:4.113}) is upper semi-continuous in \(U\)
because indicator functions of open sets are lower semi-continuous and the
negative of a lower semi-continuous function is upper semi-continuous.

\medskip
\noindent\underline{Step 3}

Let \(U_\pm:=U(h(t)\pm,t)\). Note that by Lemma~\ref{lem:4.2} and
\eqref{eq:4.114},
\begin{equation}
\begin{aligned}
\dot h(t)\in\Bigl[&-\sup_{U\in \mathcal{V}}|\lambda_1(U)|-C_*,\\
&\max\bigl\{\lambda_1(U_+)-C_*
\one_{\{U\mid a\eta(U\mid U_M)<\eta(U\mid U_L)\}}(U_+),\\
&\hspace{4.3em}\lambda_1(U_-)-C_*
\one_{\{U\mid a\eta(U\mid U_M)<\eta(U\mid U_L)\}}(U_-)
\bigr\}\Bigr].
\end{aligned}
\label{eq:4.123}
\end{equation}

We are now ready to show \eqref{eq:4.2}. For each fixed time \(t\), we have 4
cases to consider to prove \eqref{eq:4.2}:

\medskip
\noindent\textit{Case 1}
\begin{align}
a\eta(U_-\mid U_M)&<\eta(U_-\mid U_L),
\label{eq:4.124}\\
a\eta(U_+\mid U_M)&<\eta(U_+\mid U_L).
\label{eq:4.125}
\end{align}

\noindent\textit{Case 2}
\begin{align}
a\eta(U_-\mid U_M)&<\eta(U_-\mid U_L),
\label{eq:4.126}\\
a\eta(U_+\mid U_M)&\geq\eta(U_+\mid U_L).
\label{eq:4.127}
\end{align}

\noindent\textit{Case 3}
\begin{align}
a\eta(U_-\mid U_M)&\geq\eta(U_-\mid U_L),
\label{eq:4.128}\\
a\eta(U_+\mid U_M)&<\eta(U_+\mid U_L).
\label{eq:4.129}
\end{align}

\noindent\textit{Case 4}
\begin{align}
a\eta(U_-\mid U_M)&\geq\eta(U_-\mid U_L),
\label{eq:4.130}\\
a\eta(U_+\mid U_M)&\geq\eta(U_+\mid U_L).
\label{eq:4.131}
\end{align}

Note that we allow for \(U_+=U_-\).

We start with

\noindent\textit{Case 1}
We consider only the 1-shock case. The 2-shock case is very similar.
Then, in this (Case 1), by \eqref{eq:4.123} and \eqref{eq:4.114} we know that
\begin{equation}
\begin{aligned}
\dot h(t)
&\leq-\frac{1}{c_4\gamma_0^2}
\left(\sup_{U,U_L,U_M\in \mathcal{V}}|a q(U;U_M)-q(U;U_L)|+1\right)
-\sup_{U\in \mathcal{V}}|\lambda_1(U)|\\
&<\inf_{U\in \mathcal{V}}\lambda_1(U).
\end{aligned}
\label{eq:4.132}
\end{equation}

If \(U_+\neq U_-\), then we have \eqref{eq:4.121} and \eqref{eq:4.122}.
But then this is a contradiction because the shock speed $\dot{h}(t)$ is slower (more negative) than any 1-shock (remark also that 2-shocks always have positive shock speed). Thus, we conclude that \(U_+=U_-\).

Let \(\overline U:=U_+=U_-\). If
\[
\dist\bigl(\overline U,\{W\mid\eta(W\mid U_L)
\leq a\eta(W\mid U_M)\}\bigr)\geq\gamma_0,
\]
then
\begin{equation}
\begin{aligned}
&a\bigl(q(U_+; U_M)-\dot h(t)\eta(U_+\mid U_M)\bigr)
-q(U_-; U_L)+\dot h(t)\eta(U_-\mid U_L)\\
&=a\bigl(q(\overline U; U_M)-\dot h(t)\eta(\overline U\mid U_M)\bigr)
-q(\overline U; U_L)+\dot h(t)\eta(\overline U\mid U_L)\\
&=a q(\overline U; U_M)-q(\overline U; U_L)
-\dot h(t)\bigl(a\eta(\overline U\mid U_M)-\eta(\overline U\mid U_L)\bigr)\\
&\leq-1,
\end{aligned}
\label{eq:4.133}
\end{equation}
because of \eqref{eq:4.132} and \eqref{eq:4.103}. Because \(\sigma\) is
fixed and \(\dot h\) is bounded by \eqref{eq:4.117}, the term
\(|\sigma-\dot h(t)|^2\) on the right-hand side of \eqref{eq:4.2} is bounded.
Thus, we have proven
\eqref{eq:4.2} by choosing \(c\) sufficiently small.

If on the other hand,
\[
\dist\bigl(\overline U,\{W\mid\eta(W\mid U_L)
\leq a\eta(W\mid U_M)\}\bigr)<\gamma_0,
\]
then
\begin{align*}
&a\bigl(q(U_+; U_M)-\dot h(t)\eta(U_+\mid U_M)\bigr)
-q(U_-; U_L)+\dot h(t)\eta(U_-\mid U_L)\\
&=a\bigl(q(\overline U; U_M)-\dot h(t)\eta(\overline U\mid U_M)\bigr)
-q(\overline U; U_L)+\dot h(t)\eta(\overline U\mid U_L)\\
&=a q(\overline U; U_M)-q(\overline U; U_L)
-\dot h(t)\bigl(a\eta(\overline U\mid U_M)-\eta(\overline U\mid U_L)\bigr)\\
&\leq a\bigl(q(\overline U; U_M)-\lambda_1(\overline U)\eta(\overline U\mid U_M)\bigr)
-q(\overline U; U_L)+\lambda_1(\overline U)\eta(\overline U\mid U_L),
\end{align*}
because \(\eta(\overline U\mid U_L)-a\eta(\overline U\mid U_M)\geq0\) and
\(\dot h\leq-\sup_{U\in \mathcal{V}}|\lambda_1(U)|\). Continuing, we get
\begin{equation}
\leq-\frac12c_1,
\label{eq:4.134}
\end{equation}
from \eqref{eq:4.4}, the definition of \(\gamma_0\) (see \eqref{eq:4.111}), and  the
assumption that
\begin{equation}
\dist\bigl(\overline U,\{W\mid\eta(W\mid U_L)
\leq a\eta(W\mid U_M)\}\bigr)<\gamma_0.
\label{eq:4.135}
\end{equation}
Again because the term
\(|\sigma-\dot h(t)|^2\) on the right-hand side of \eqref{eq:4.2} is bounded
due to \eqref{eq:4.117}, we have proven
\eqref{eq:4.2} by choosing \(c\) sufficiently small.

\noindent\textit{Case 2}

In this case, we must have \(U_-\neq U_+\).

We have from
\eqref{eq:4.123} that
\begin{equation}\label{contact_question}
\dot h\in\left[-\sup_{U\in \mathcal{V}}|\lambda_1(U)|-C_*,\lambda_1(U_+)\right].
\end{equation}

This implies that $\dot{h}<0$ and thus \((U_+,U_-,\dot h)\) is a 1-shock. But, this is in contradiction  with \Cref{lem:remote-endpoint-control} so this case cannot occur.

\begin{figure}[tb]
      \includegraphics[width=\textwidth]{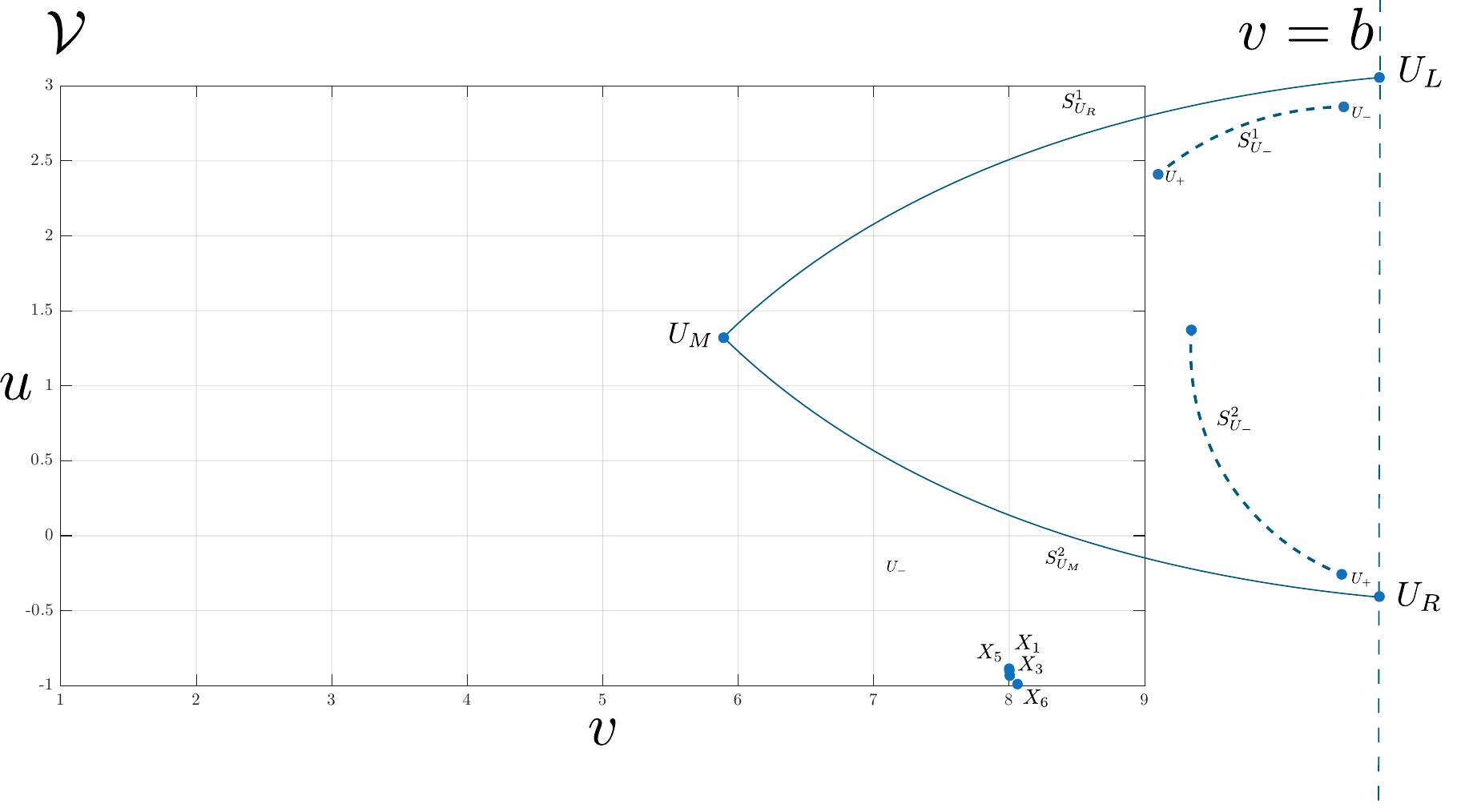}
  \caption{This diagram shows the state space $\mathcal{V}$ and the 6 $T_N$ points $X_i$ from \Cref{T_6_with_ordering} (projected to the $v,u$ plane). Note that the $v$-coordinates of the $X_i$ are ordered (from left to right): $X_2,X_4,X_5,X_1,X_3,X_6$. We also show the Riemann states $U_L, U_R$ and the intermediate state $U_M$ (see \Cref{sec:def_Riemann_data}). The line $v=b$ from \Cref{sec:envelope} is also shown.}\label{1-shock_case_pic}
\end{figure}

\noindent\textit{Case 3}

In this case (Case 3), we must have \(U_-\neq U_+\). Further, in this case, we have
from \eqref{eq:4.119} that
\begin{equation}
\dot h\in\left[-\sup_{U\in \mathcal{V}}|\lambda_1(U)|-C_*,\lambda_1(U_-)\right].
\label{eq:4.136}
\end{equation}
This implies that $\dot{h}<0$ and thus \((U_+,U_-,\dot h)\) is a 1-shock.

Also,
\(U_-\) verifies
\(a\eta(U_-\mid U_M)\geq\eta(U_-\mid U_L)\). Thus, we can
apply \Cref{prop:4.2}. We receive \eqref{eq:4.2}. We refer to \Cref{1-shock_case_pic}.

\noindent\textit{Case 4}

For this Case 4, we then have two subcases.

Consider first when \(U_+\neq U_-\) (so \((U_+,U_-)\) is an entropic shock
by \eqref{eq:4.121} and \eqref{eq:4.122}). In this case, we have from
\eqref{eq:4.119} that
\[
\dot h\in\left[-\sup_{U\in \mathcal{V}}|\lambda_1(U)|-C_*,
\max\{\lambda_1(U_+),\lambda_1(U_-)\}\right].
\]
This implies that $\dot{h}<0$. We conclude that \((U_+,U_-,\dot h)\)
is a 1-shock. Moreover, \(U_-\) verifies
\(a\eta(U_-\mid U_M)\geq\eta(U_-\mid U_L)\). We can now apply
\Cref{prop:4.2}. This gives \eqref{eq:4.2}.

On the other hand, consider when \(U_+=U_-\). Define \(\overline U:=U_+=U_-\). We
then now have two (sub-)subcases. Consider when
\(\eta(\overline U\mid U_L)=a\eta(\overline U\mid U_M)\). Then,
\begin{equation}
\begin{aligned}
&a\bigl(q(U_+; U_M)-\dot h(t)\eta(U_+\mid U_M)\bigr)
-q(U_-; U_L)+\dot h(t)\eta(U_-\mid U_L)\\
&=a\bigl(q(\overline U; U_M)-\dot h(t)\eta(\overline U\mid U_M)\bigr)
-q(\overline U; U_L)+\dot h(t)\eta(\overline U\mid U_L)\\
&=a q(\overline U; U_M)-q(\overline U; U_L)
-\dot h(t)\bigl(a\eta(\overline U\mid U_M)-\eta(\overline U\mid U_L)\bigr)\\
&=a\bigl(q(\overline U; U_M)-\lambda_1(\overline U)\eta(\overline U\mid U_M)\bigr)
-q(\overline U; U_L)+\lambda_1(\overline U)\eta(\overline U\mid U_L),
\end{aligned}
\label{eq:4.138}
\end{equation}
because \(\eta(\overline U\mid U_L)-a\eta(\overline U\mid U_M)=0\). Thus
\eqref{eq:4.2} follows directly from \Cref{prop:4.2} (in particular,
\eqref{eq:4.4}). As in Case 1 above, because the term
\(|\sigma-\dot h(t)|^2\) on the right hand side of \eqref{eq:4.2} is bounded
due to \eqref{eq:4.117}, we have proven
\eqref{eq:4.2} by choosing \(c\) sufficiently small.

In the other (sub)-subcase, consider when
\(\eta(\overline U\mid U_L)<a\eta(\overline U\mid U_M)\) (strict inequality).
Then, \(\mathcal{W}(\cdot,t)\) is continuous at \(\overline U\) and thus \eqref{eq:4.120}
applies, giving us that \(\dot h(t)=\lambda_1(\overline U)\). Again, \eqref{eq:4.2}
follows directly from \Cref{prop:4.2}.

\subsection{Proof of \Cref{prop:4.2}}\label{sec:proof4.2}

This argument follows very similarly to \cite{MR3519973} but we include the details for completeness.

\subsubsection{Why the localized algebraic contraction estimate is true}

For the reference 1-shock, abbreviate
\[
L=U_L,\qquad R=U_R,
\]
and write
\[
R=S^1_L(r_0),\qquad \sigma_0=\sigma^1_L(r_0),\qquad r_0>0.
\]
For any state $X$, comparison state $Y$, and number $s$, set
\begin{align}\label{Qcal_def}
\Qcal(X;Y,s):=q(X;Y)-s\eta(X\mid Y).
\end{align}
The desired estimate is
\begin{equation}
\boxed{
a\Qcal(S^1_Z(r);R,\sigma^1_Z(r))
-\Qcal(Z;L,\sigma^1_Z(r))
\le -c_a|\sigma^1_Z(r)-\sigma_0|^2
}
\tag{11}
\end{equation}
whenever
\begin{equation}
\eta(Z\mid L)\le a\eta(Z\mid R).
\tag{12}
\end{equation}

\subsubsection{Why condition (12) puts $Z$ near $L$}

By \Cref{size_pi}, base-state error
\[
e:=|Z-L|
\]
is $O(\sqrt a)$.

\subsubsection{Differentiate Rankine--Hugoniot}

We now show, for a fixed base $Z$,

\begin{equation}
\boxed{
\frac{d}{dr}\Qcal(S^1_Z(r);Y,\sigma^1_Z(r))
=(\sigma^1_Z)'(r)
\big[\eta(Z\mid S^1_Z(r))-\eta(Z\mid Y)\big].
}
\tag{15}
\end{equation}

This follows immediately from the Lax dissipation formula \Cref{LDF} but we provide a complete proof for the reader:
Consider Rankine-Hugoniot
\[
F(S^1_Z(r))-F(Z)=\sigma^1_Z(r)(S^1_Z(r)-Z).
\]
Differentiation gives
\begin{equation}
(DF(S^1_Z(r))-\sigma^1_Z(r)I)(S^1_Z)'(r)
=(\sigma^1_Z)'(r)(S^1_Z(r)-Z).
\tag{14}
\end{equation}
Entropy compatibility implies
\[
D_Xq(X;Y)=(D\eta(X)-D\eta(Y))DF(X).
\]
Also,
\[
D_X\eta(X\mid Y)=D\eta(X)-D\eta(Y).
\]
Therefore
\[
\begin{aligned}
\frac{d}{dr}\Qcal(S^1_Z(r);Y,\sigma^1_Z(r))
&=(D\eta(S^1_Z(r))-D\eta(Y))
  (DF(S^1_Z(r))-\sigma^1_Z(r)I)(S^1_Z)'(r)\\
&\quad-(\sigma^1_Z)'(r)\eta(S^1_Z(r)\mid Y).
\end{aligned}
\]
Use (14). The result is
\[
(\sigma^1_Z)'(r)
\Big[(D\eta(S^1_Z(r))-D\eta(Y))(S^1_Z(r)-Z)
-\eta(S^1_Z(r)\mid Y)\Big].
\]
Expanding the relative entropies (see also \cite[Lemma 4.1]{MR3519973}) shows that the bracket equals
\[
\eta(Z\mid S^1_Z(r))-\eta(Z\mid Y).
\]
Hence we get (15).

\subsubsection{First compare to the moving state $S^1_Z(r_0)$}

Take
\[
Y=S^1_Z(r_0),
\]
which is fixed while $r$ varies. At $r=r_0$,
\[
\Qcal(S^1_Z(r_0);S^1_Z(r_0),\sigma^1_Z(r_0))=0.
\]
Integrating (15),
\begin{equation}
\begin{aligned}
\Qcal(S^1_Z(r);S^1_Z(r_0),\sigma^1_Z(r))
&=\int_{r_0}^{r}(\sigma^1_Z)'(t)
\big[\eta(Z\mid S^1_Z(t))-\eta(Z\mid S^1_Z(r_0))\big]dt.
\end{aligned}
\tag{16}
\end{equation}
Why is this negative for $r\ne r_0$?

For $r>r_0$, entropy strengthening makes the bracket positive, while $(\sigma^1_Z)'<0$. The integral is negative.

For $r<r_0$, the bracket is negative on $(r,r_0)$, so the integrand is positive there. But the integral runs backwards, from $r_0$ down to $r$. It is again negative.

Near $r=r_0$, the quadratic coefficient is
\[
\frac12(\sigma^1_Z)'(r_0)\partial_r\eta(Z\mid S^1_Z(r_0))<0.
\]
Thus the quotient
\[
-\frac{\Qcal(S^1_Z(r);S^1_Z(r_0),\sigma^1_Z(r))}{(r-r_0)^2}
\]
extends continuously and positively through $r=r_0$.

Bases and strengths range over compact sets. Its minimum is therefore positive:
\begin{equation}
\boxed{
\Qcal(S^1_Z(r);S^1_Z(r_0),\sigma^1_Z(r))
\le -c_0(r-r_0)^2.
}
\tag{17}
\end{equation}

\subsubsection{$S^1_Z(r_0)$ versus $R$}

The reference state is
\[
R=S^1_L(r_0),
\]
not $S^1_Z(r_0)$. They agree only when $Z=L$.

Smooth dependence gives
\[
|S^1_Z(r_0)-R|\le Ce.
\]
Define the error caused by changing comparison state:
\[
\begin{aligned}
D(Z,r)
&:=\Qcal(S^1_Z(r);R,\sigma^1_Z(r))\\
&\quad-\Qcal(S^1_Z(r);S^1_Z(r_0),\sigma^1_Z(r)).
\end{aligned}
\]
At $r=r_0$, this is $O(e^2)$, because relative entropy and relative entropy flux vanish to second order on the diagonal.

Its $r$-derivative is $O(e)$. One can see this directly from (15): subtracting the two differentiated formulas leaves
\[
(\sigma^1_Z)'(r)\big[\eta(Z\mid S^1_Z(r_0))-\eta(Z\mid R)\big],
\]
which is $O(e)$ on the compact family.

Therefore
\[
|D(Z,r)|\le Ce^2+Ce|r-r_0|.
\]
Young's inequality absorbs the mixed term:
\[
Ce|r-r_0|
\le \frac{c_0}{2}|r-r_0|^2+C'e^2.
\]
Together with (17),
\begin{equation}
\boxed{
\Qcal(S^1_Z(r);R,\sigma^1_Z(r))
\le -\frac{c_0}{2}|r-r_0|^2+C_0e^2.
}
\tag{18}
\end{equation}

\subsubsection{Estimate the other term by Taylor expansion}

Let
\[
\xi=Z-L,\qquad H_L=D^2\eta(L).
\]
At $Z=L$, both the value and first derivative of $\Qcal(Z;L,s)$ vanish. Its quadratic term is
\begin{equation}
\boxed{
\Qcal(Z;L,s)
=\frac12\xi^T H_L(DF(L)-sI)\xi+O(e^3).
}
\tag{19}
\end{equation}
Here entropy compatibility makes $H_LDF(L)$ symmetric. In the $H_L$-inner product, the smallest eigenvalue of $DF(L)$ is $\lambda_1(L)$. Consequently,
\begin{equation}
\xi^T H_L(DF(L)-sI)\xi
\ge(\lambda_1(L)-s)\xi^T H_L\xi.
\tag{20}
\end{equation}
We now split the strength range.

\subsubsection{Strengths bounded away from zero}

Suppose
\[
r\ge r_0/2.
\]
Strict speed monotonicity gives
\[
\sigma^1_L(r)\le\sigma^1_L(r_0/2)<\lambda_1(L).
\]
By continuity, after restricting $Z$ to a sufficiently small neighborhood of $L$,
\[
\sigma^1_Z(r)\le\lambda_1(L)-\kappa
\]
for some $\kappa>0$.
Equations (19)--(20) then give
\[
\Qcal(Z;L,\sigma^1_Z(r))\ge c_1e^2,
\]
after making the neighborhood small enough to absorb the cubic remainder.

Using (18),
\[
\begin{aligned}
a\Qcal(S^1_Z(r);R,\sigma^1_Z(r))
-\Qcal(Z;L,\sigma^1_Z(r))
&\le -\frac{ac_0}{2}|r-r_0|^2
-(c_1-aC_0)e^2.
\end{aligned}
\]
Choose $a$ so that
\[
aC_0\le c_1/2.
\]
Both terms are then nonpositive, with quantitative control of both
\[
|r-r_0|^2\qquad\text{and}\qquad e^2.
\]
Finally, smoothness gives
\[
|\sigma^1_Z(r)-\sigma_0|\le C(e+|r-r_0|).
\]
Thus the negative bound controls the squared speed difference, proving (11) in this range.

\subsubsection{Small strengths}

Now suppose
\[
0\le r\le r_0/2.
\]
We cannot use the same fixed spectral gap: at $r=0$,
\[
\sigma^1_Z(0)=\lambda_1(Z).
\]
Instead, observe that at $Z=L$,
\[
\Qcal(S^1_L(r);R,\sigma^1_L(r))<0
\]
throughout this compact interval, by (16). Because the interval stays away from $r_0$, continuity gives
\[
\Qcal(S^1_Z(r);R,\sigma^1_Z(r))\le-\kappa_1<0
\]
for nearby $Z$.

For the base term, monotonicity gives
\[
\sigma^1_Z(r)\le\lambda_1(Z)=\lambda_1(L)+O(e).
\]
Therefore the possible negative part of the quadratic form in (19) is only $O(e^3)$:
\[
-\Qcal(Z;L,\sigma^1_Z(r))\le Ce^3.
\]
Using $e^2\le B_0a$ for some constant $B_0>$,
\[
\begin{aligned}
a\Qcal(S^1_Z(r);R,\sigma^1_Z(r))
-\Qcal(Z;L,\sigma^1_Z(r))
&\le -a\kappa_1+CB_0^{3/2}a^{3/2}.
\end{aligned}
\]
Thus sufficiently small $a$ gives
\[
\le -\frac12a\kappa_1.
\]
The speed difference is bounded on the compact family, so this strictly negative constant also bounds
\[
-c_a|\sigma^1_Z(r)-\sigma_0|^2
\]
after reducing $c_a>0$.

At $r=0$, we have $S^1_Z(0)=Z$, and obtain the second estimate:
\begin{equation}
\boxed{
a\Qcal(Z;R,\lambda_1(Z))
-\Qcal(Z;L,\lambda_1(Z))\le-\delta_a.
}
\tag{21}
\end{equation}
That completes the localized algebraic proposition. 

\subsection{$L^2$ stability for shocks}

From \Cref{l2control}, given an $i$-shock
$(U_L,U_R,\sigma(U_L,U_R))$, where $\sigma(U_L,U_R)$ is the
shock velocity dictated by the Rankine--Hugoniot condition, and a
solution $U\in\mathcal{S}_{\mathrm{weak}}$ we consider the weighted
relative entropy
\begin{equation}
E_t(U;a_1,a_2,h)
:=
\int_{-\infty}^{h(t)}
a_1\eta(U\mid U_L)\,dx
+
\int_{h(t)}^\infty
a_2\eta(U\mid U_R)\,dx
\label{eq:weighted-relative-entropy}
\end{equation}
with the constants $a_1,a_2>0$ and $h$ a Lipschitz curve.
Differentiating and using both \Cref{strong_trace_def} and \eqref{eq:entropy-inequality}
we find the dissipation functional $D_t$,
\begin{align}
\label{dissipation_functional}
\frac{d}{dt}E_t(U;a_1,a_2,h)
\leq {}&
a_2
\underbrace{
\left[
q(U(h(t)+,t);U_R)
-
\dot h(t)\eta(U(h(t)+,t)\mid U_R)
\right]
}_{\text{dissipation from right side of $h$}}
\nonumber\\
&\quad
-
a_1
\underbrace{
\left[
q(U(h(t)-,t);U_L)
-
\dot h(t)\eta(U(h(t)-,t)\mid U_L)
\right]
}_{\text{dissipation from left side of $h$}}
\nonumber\\
={}&
D_t(U;a_1,a_2,h).
\end{align}

Thus, showing the existence of a triple $(a_1,a_2,h)$ which gives
bounds on $D_t$ then gives us bounds on the $L^2$ distance of $U$ and
$(U_L,U_R,h)$, which is no longer a shock but rather a shifted front.

\subsection{Proof of \eqref{eq:main-stability-estimate}}
\label{subsec:two-shock-riemann-specialization}

We consider the Riemann problem with left state $U_L$ and right state $U_R$ (see \Cref{sec:def_Riemann_data}), which admits a classical self-similar solution consisting of a 1-shock from $U_L$ to an intermediate state $U_M$ and a 2-shock from $U_M$ to $U_R$. We denote this classical solution by $U_{\text{self-similar}}$. We also construct a shifted profile, using artificial shifts and an $a$-contraction weight, which remains $L^2$ stable with respect to any perturbation $U\in\mathcal{S}_{\mathrm{weak}}$ (and thus having the Strong Trace Property!). We denote this shifted profile by $U_{\mathrm{shift}}$.

Suppose that
\begin{equation}
    U_{\mathrm{shift}}(x,0)
    =
    \begin{cases}
        U_L, & x<0,\\
        U_R, & x>0,
    \end{cases}
    \label{eq:two-shock-riemann-data}
\end{equation}
and that, for \(t>0\), the corresponding profile is
\begin{equation}
    U_{\mathrm{shift}}(x,t)
    =
    \begin{cases}
        U_L, & x<h_1(t),\\[1mm]
        U_M, & h_1(t)<x<h_2(t),\\[1mm]
        U_R, & x>h_2(t),
    \end{cases}
    \label{eq:two-shock-riemann-profile}
\end{equation}
where the $h_i$ come from \Cref{prop:4.1_one_shock} (and the corresponding 2-shock statement)  and in particular $h_1$ is the shift function corresponding to the 1-shock $(U_L,U_M)$ and $h_2$ is the shift function corresponding to the 2-shock $(U_M,U_R)$.
We have
\begin{equation}
    h_1(0)=h_2(0)=0,
    \end{equation}
    and
 \begin{equation}
    h_1(t)<h_2(t)
    \quad\text{for }t>0,
    \label{eq:two-shock-initial-fronts}
\end{equation}
where this follows because $\dot{h}_1<0$ and $\dot{h}_2>0$ by \Cref{prop:4.1_one_shock}.

Since no interaction occurs after \(t=0\), the associated weight we define  is constant
in each of the three regions.  We write
\begin{equation}
    a(x,t)
    =
    \begin{cases}
        a_L, & x<h_1(t),\\[1mm]
        a_M, & h_1(t)<x<h_2(t),\\[1mm]
        a_R, & x>h_2(t),
    \end{cases}
    \qquad
    a_L,a_M,a_R>0.
    \label{eq:two-shock-weight}
\end{equation}
The ratios \(a_M/a_L\) and \(a_R/a_M\) are chosen according to the $a$-contraction conditions.  In particular, the weight decreases across the \(1\)-shock and
increases across the \(2\)-shock.

Fix \(\tau>0\) and \(R>0\).  The lateral boundaries of the information cone
are
\begin{equation}
    h_0(t):=-R+s(t-\tau),
    \qquad
    h_3(t):=R-s(t-\tau),
    \label{eq:two-shock-information-cone}
\end{equation}
where the information speed \(s>0\) is chosen so that
\begin{equation}
    \lvert q(V;W)\rvert
    \leq
    s\,\eta(V\mid W)
    \label{eq:two-shock-information-speed}
\end{equation}
for every relevant pair of states \(V,W\) and also $s$  is larger than the Lipschitz norms of $h_1,h_2$ (which are themselves uniformly bounded). Notice that
\begin{equation}
    h_0(0)=-R-s\tau,
    \qquad
    h_3(0)=R+s\tau.
\end{equation}
Recall from \eqref{Qcal_def} that
\begin{equation}
    \Qcal(V;W,\xi)
    =
    q(V;W)-\xi\,\eta(V\mid W).
    \label{eq:two-shock-Qcal-definition}
\end{equation}
We use the trace notation
\begin{equation}
    U_i^{\pm}(r):=U(h_i(r)\pm,r),
    \qquad i=1,2,
    \label{eq:two-shock-interior-traces}
\end{equation}
and, at the lateral boundaries of the cone,
\begin{equation}
    U_0^+(r):=U(h_0(r)+,r),
    \qquad
    U_3^-(r):=U(h_3(r)-,r).
    \label{eq:two-shock-exterior-traces}
\end{equation}

For $\tau>0$, in any fixed ``quadrilateral''
\begin{equation}
\square
=
\left\{
(x,r)\mid 0\leq r<t,\ h_i(r)<x<h_{i+1}(r)
\right\},
\end{equation}
the functions $U_{\mathrm{shift}}|_{\square}$ and $a|_{\square}$ are both constant. We then integrate the
entropy inequality over $\square$ with entropy $\eta(\,\cdot\mid U_{\mathrm{shift}}|_{\square})$, entropy-flux
$q(\,\cdot\,;U_{\mathrm{shift}}|_{\square})$, and use that $U$ has the Strong Trace
Property (\Cref{strong_trace_def}).

\subsubsection*{The three regional inequalities}

On the left region, where \(U_{\mathrm{shift}}=U_L\), the quadrilateral estimate gives
\begin{align}
    a_L\int_{h_0(t)}^{h_1(t)}
        \eta\bigl(U(x,t)\mid U_L\bigr)\,dx
    &\leq
    a_L\int_{-R-s\tau}^{0}
        \eta\bigl(U(x,0)\mid U_L\bigr)\,dx
    \notag\\
    &\quad
    +a_L\int_0^t
        \left[
            \Qcal\bigl(U_0^+(r);U_L,s\bigr)
            -
            \Qcal\bigl(U_1^-(r);U_L,\dot h_1(r)\bigr)
        \right]dr.
    \label{eq:two-shock-left-region}
\end{align}
where the various $\Qcal$ terms are the right/left parts of the
dissipation functional \eqref{dissipation_functional} for the fronts bounding the left/right
edges of the quadrilateral.

On the middle region, where \(U_{\mathrm{shift}}=U_M\), one obtains
\begin{align}
    a_M\int_{h_1(t)}^{h_2(t)}
        \eta\bigl(U(x,t)\mid U_M\bigr)\,dx
    &\leq
    a_M\int_0^t
        \left[
            \Qcal\bigl(U_1^+(r);U_M,\dot h_1(r)\bigr)
        \right.
    \notag\\[-1mm]
    &\hspace{37mm}\left.
            -
            \Qcal\bigl(U_2^-(r);U_M,\dot h_2(r)\bigr)
        \right]dr.
    \label{eq:two-shock-middle-region}
\end{align}
There is no initial integral in \eqref{eq:two-shock-middle-region}, since the
middle interval has zero length at \(t=0\).  Equivalently, one may first
apply the estimate on a time interval \([\delta,t]\), with \(\delta>0\),
and then let \(\delta\downarrow0\).

On the right region, where \(U_{\mathrm{shift}}=U_R\), the estimate becomes
\begin{align}
    a_R\int_{h_2(t)}^{h_3(t)}
        \eta\bigl(U(x,t)\mid U_R\bigr)\,dx
    &\leq
    a_R\int_{0}^{R+s\tau}
        \eta\bigl(U(x,0)\mid U_R\bigr)\,dx
    \notag\\
    &\quad
    +a_R\int_0^t
        \left[
            \Qcal\bigl(U_2^+(r);U_R,\dot h_2(r)\bigr)
            -
            \Qcal\bigl(U_3^-(r);U_R,-s\bigr)
        \right]dr.
    \label{eq:two-shock-right-region}
\end{align}

\subsubsection*{Step 1: Summation over the three regions}

Define the weighted relative entropy inside the information cone by
\begin{align}
    \mathcal{E}(t)
    :={}&
    a_L\int_{h_0(t)}^{h_1(t)}
        \eta\bigl(U(x,t)\mid U_L\bigr)\,dx
    \notag\\
    &+
    a_M\int_{h_1(t)}^{h_2(t)}
        \eta\bigl(U(x,t)\mid U_M\bigr)\,dx
    \notag\\
    &+
    a_R\int_{h_2(t)}^{h_3(t)}
        \eta\bigl(U(x,t)\mid U_R\bigr)\,dx,
    \label{eq:two-shock-energy}
\end{align}
and set
\begin{align}
    \mathcal{E}(0)
    :={}&
    a_L\int_{-R-s\tau}^{0}
        \eta\bigl(U(x,0)\mid U_L\bigr)\,dx
    \notag\\
    &+
    a_R\int_{0}^{R+s\tau}
        \eta\bigl(U(x,0)\mid U_R\bigr)\,dx.
    \label{eq:two-shock-initial-energy}
\end{align}
The dissipation contributions at the first and second shocks are,
respectively,
\begin{align}
    D_1(r)
    :={}&
    a_M\left[
        q\bigl(U_1^+(r);U_M\bigr)
        -\dot h_1(r)\,
            \eta\bigl(U_1^+(r)\mid U_M\bigr)
    \right]
    \notag\\
    &-
    a_L\left[
        q\bigl(U_1^-(r);U_L\bigr)
        -\dot h_1(r)\,
            \eta\bigl(U_1^-(r)\mid U_L\bigr)
    \right],
    \label{eq:two-shock-D1}
\end{align}
and
\begin{align}
    D_2(r)
    :={}&
    a_R\left[
        q\bigl(U_2^+(r);U_R\bigr)
        -\dot h_2(r)\,
            \eta\bigl(U_2^+(r)\mid U_R\bigr)
    \right]
    \notag\\
    &-
    a_M\left[
        q\bigl(U_2^-(r);U_M\bigr)
        -\dot h_2(r)\,
            \eta\bigl(U_2^-(r)\mid U_M\bigr)
    \right].
    \label{eq:two-shock-D2}
\end{align}
Adding \eqref{eq:two-shock-left-region},
\eqref{eq:two-shock-middle-region}, and
\eqref{eq:two-shock-right-region} yields
\begin{align}
    \mathcal{E}(t)
    \leq{}&
    \mathcal{E}(0)
    +\int_0^t
        \left[
            a_L\Qcal\bigl(U_0^+(r);U_L,s\bigr)
            -
            a_R\Qcal\bigl(U_3^-(r);U_R,-s\bigr)
        \right]dr
    \notag\\
    &+
    \int_0^t\bigl(D_1(r)+D_2(r)\bigr)\,dr.
    \label{eq:two-shock-summed-before-boundary}
\end{align}
By \eqref{eq:two-shock-information-speed},
\begin{equation}
    \Qcal(V;W,s)
    =q(V;W)-s\eta(V\mid W)
    \leq0,
    \qquad
    \Qcal(V;W,-s)
    =q(V;W)+s\eta(V\mid W)
    \geq0.
    \label{eq:two-shock-boundary-signs}
\end{equation}
Consequently, both lateral-boundary contributions in
\eqref{eq:two-shock-summed-before-boundary} are nonpositive.  We therefore
obtain the two-shock specialization
\begin{equation}
    \boxed{
        \mathcal{E}(t)
        \leq
        \mathcal{E}(0)
        +
        \int_0^t\bigl(D_1(r)+D_2(r)\bigr)\,dr
    }.
    \label{eq:two-shock-specialization}
\end{equation}

Remark that if 
\begin{equation}
    h_1(r)=\sigma_1 r,
    \qquad
    h_2(r)=\sigma_2 r,
    \label{eq:two-shock-self-similar-fronts}
\end{equation}
where \(\sigma_1\) and \(\sigma_2\) are the two Rankine--Hugoniot speeds, for $(U_L,U_M)$ and $(U_M,U_R)$, respectively, then \(U_{\mathrm{shift}}\) is the unshifted self-similar Riemann solution.

\begin{figure}[t]
    \centering
    \includegraphics[
        width=0.78\textwidth
    ]{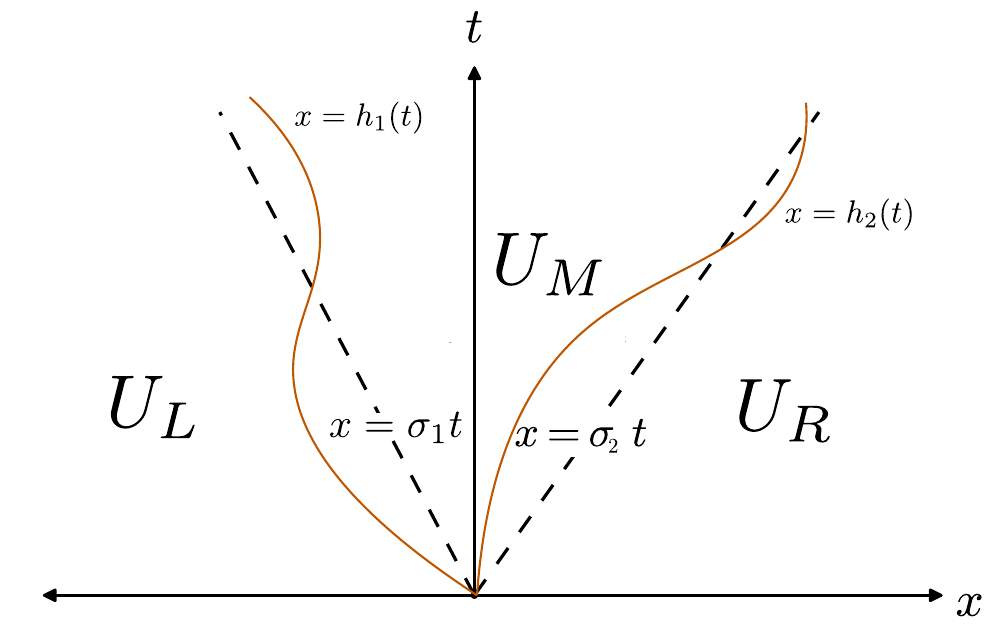}
    \caption{
        This diagram shows the exact speeds $x=\sigma_i t$ for the shocks $(U_L,U_M)$ and $(U_M,U_R)$ given by Rankine-Hugoniot (straight lines), as well as the \emph{shifted positions} $h_i(t)$ (curvy lines).
    }
    \label{fig:shifted_shocks}
\end{figure}

In the shifted
construction, the curves \(h_1\) and \(h_2\) remain Lipschitz, and the
variable velocities \(\dot h_1\) and \(\dot h_2\) are retained in
\eqref{eq:two-shock-D1}--\eqref{eq:two-shock-D2}. See \Cref{fig:shifted_shocks}.

\subsubsection*{Whole-line formulation}

If the weighted relative entropy is integrable on the whole line, the cone
boundaries may be sent to infinity.  In that case,
\begin{align}
    &a_L\int_{-\infty}^{h_1(t)}
        \eta\bigl(U(x,t)\mid U_L\bigr)\,dx
    +a_M\int_{h_1(t)}^{h_2(t)}
        \eta\bigl(U(x,t)\mid U_M\bigr)\,dx
    \notag\\
    &\qquad
    +a_R\int_{h_2(t)}^{\infty}
        \eta\bigl(U(x,t)\mid U_R\bigr)\,dx
    \notag\\
    &\quad\leq
    a_L\int_{-\infty}^{0}
        \eta\bigl(U(x,0)\mid U_L\bigr)\,dx
    +a_R\int_{0}^{\infty}
        \eta\bigl(U(x,0)\mid U_R\bigr)\,dx
    \notag\\
    &\qquad
    +\int_0^t\bigl(D_1(r)+D_2(r)\bigr)\,dr.
    \label{eq:two-shock-whole-line}
\end{align}

\subsubsection*{Quantitative form of the two-shock dissipation estimate}

Let
\begin{equation}
    \delta_1:=\lvert U_M-U_L\rvert,
    \qquad
    \delta_2:=\lvert U_R-U_M\rvert,
    \qquad
    \delta:=\delta_1+\delta_2.
    \label{eq:two-shock-strengths}
\end{equation}
Denote by \(\sigma_1\) and \(\sigma_2\) the Rankine--Hugoniot speeds of
\((U_L,U_M)\) and \((U_M,U_R)\), respectively.  Thus, the unshifted
Riemann solution has discontinuities along \(x=\sigma_1t\) and
\(x=\sigma_2t\). 

The dissipation estimate \eqref{eq:4.2} implies that there exists a
constant \(K_{\mathrm d}\geq1\), depending only on the fixed compact set of
admissible states and on the constants entering the weight construction,
such that, for almost every \(r>0\),
\begin{equation}
    D_1(r)
    \leq
    -\frac{\delta_1}{K_{\mathrm d}}
        \bigl\lvert \dot h_1(r)-\sigma_1\bigr\rvert^2,
    \qquad
    D_2(r)
    \leq
    -\frac{\delta_2}{K_{\mathrm d}}
        \bigl\lvert \dot h_2(r)-\sigma_2\bigr\rvert^2.
    \label{eq:two-shock-pointwise-dissipation}
\end{equation}
Combining \eqref{eq:two-shock-specialization} with
\eqref{eq:two-shock-pointwise-dissipation}, we obtain, for every
\(\tau>0\),
\begin{align}
    \mathcal{E}(\tau)
    &+
    \frac{1}{K_{\mathrm d}}
    \int_0^\tau
    \left[
        \delta_1\bigl\lvert \dot h_1(r)-\sigma_1\bigr\rvert^2
        +
        \delta_2\bigl\lvert \dot h_2(r)-\sigma_2\bigr\rvert^2
    \right]dr
    \notag\\
    &\leq
    \mathcal{E}(0).
    \label{eq:two-shock-quantitative-dissipation}
\end{align}

\subsubsection*{Step 2: Bound on shifting}

Let \(U_{\text{self-similar}}\) denote the unshifted self-similar Riemann solution with the same
initial data as \(U_{\mathrm{shift}}\):
\begin{equation}
    U_{\text{self-similar}}(x,t)
    =
    \begin{cases}
        U_L, & x<\sigma_1t,\\[1mm]
        U_M, & \sigma_1t<x<\sigma_2t,\\[1mm]
        U_R, & x>\sigma_2t.
    \end{cases}
    \label{eq:two-shock-unshifted-riemann-solution}
\end{equation}
In particular,
\begin{equation}
    U_{\text{self-similar}}(\cdot,0)=U_{\mathrm{shift}}(\cdot,0).
    \label{eq:two-shock-same-initial-data}
\end{equation}
In the present two-front setting, the cost of shifting can be computed
directly, without introducing the general weighted \(L^1\) functional (as in the works \cite{ChenFaileKrupa2025Holder,2025arXiv250916432C}).
Indeed, writing the two profiles as sums of Heaviside jumps and using
\begin{equation}
    \left\|
        \boldsymbol{1}_{(a,\infty)}
        -
        \boldsymbol{1}_{(b,\infty)}
    \right\|_{L^1(\mathbb{R})}
    =\lvert a-b\rvert,
\end{equation}
one obtains
\begin{align}
    \bigl\|U_{\text{self-similar}}(\cdot,\tau)-U_{\mathrm{shift}}(\cdot,\tau)\bigr\|_{L^1(\mathbb{R})}
    &\leq
    \delta_1\bigl\lvert \sigma_1\tau-h_1(\tau)\bigr\rvert
    +
    \delta_2\bigl\lvert \sigma_2\tau-h_2(\tau)\bigr\rvert.
    \label{eq:two-shock-shift-position-bound}
\end{align}
Since \(h_1(0)=h_2(0)=0\), we have
\begin{align}
    \bigl\|U_{\text{self-similar}}(\cdot,\tau)-U_{\mathrm{shift}}(\cdot,\tau)\bigr\|_{L^1(\mathbb{R})}
    &\leq
    \int_0^\tau
    \left[
        \delta_1\bigl\lvert \sigma_1-\dot h_1(r)\bigr\rvert
        +
        \delta_2\bigl\lvert \sigma_2-\dot h_2(r)\bigr\rvert
    \right]dr.
    \label{eq:two-shock-shift-integral-bound}
\end{align}
Applying the Cauchy--Schwarz inequality with the shock strengths as weights,
we find
\begin{align}
    \bigl\|U_{\text{self-similar}}(\cdot,\tau)-U_{\mathrm{shift}}(\cdot,\tau)\bigr\|_{L^1(\mathbb{R})}
    &\leq
    \left[
        \int_0^\tau(\delta_1+\delta_2)\,dr
    \right]^{1/2}
    \notag\\
    &\quad\times
    \left[
        \int_0^\tau
        \left(
            \delta_1\bigl\lvert \sigma_1-\dot h_1(r)\bigr\rvert^2
            +
            \delta_2\bigl\lvert \sigma_2-\dot h_2(r)\bigr\rvert^2
        \right)dr
    \right]^{1/2}
    \notag\\
    &=
    \sqrt{\tau\delta}
    \left[
        \int_0^\tau
        \left(
            \delta_1\bigl\lvert \sigma_1-\dot h_1(r)\bigr\rvert^2
            +
            \delta_2\bigl\lvert \sigma_2-\dot h_2(r)\bigr\rvert^2
        \right)dr
    \right]^{1/2}.
    \label{eq:two-shock-shift-cauchy-schwarz}
\end{align}
The dissipation estimate
\eqref{eq:two-shock-quantitative-dissipation} therefore yields
\begin{equation}
    \boxed{
    \bigl\|U_{\text{self-similar}}(\cdot,\tau)-U_{\mathrm{shift}}(\cdot,\tau)\bigr\|_{L^1(\mathbb{R})}
    \leq
    \sqrt{K_{\mathrm d}\tau\delta}\,
    \mathcal{E}(0)^{1/2}
    }.
    \label{eq:two-shock-shift-final}
\end{equation}

\subsubsection*{Step 3: Putting the estimates together}

Set
\begin{equation}
    I_\tau:=(-R-s\tau,R+s\tau),
    \qquad
    J_R:=(-R,R),
    \label{eq:two-shock-space-intervals}
\end{equation}
and define
\begin{equation}
    a_*:=\min\{a_L,a_M,a_R\},
    \qquad
    a^*:=\max\{a_L,a_M,a_R\}.
    \label{eq:two-shock-weight-extrema}
\end{equation}
Let \(C_\eta\geq1\) be such that, on the relevant compact set of states,
\begin{equation}
    \frac{1}{C_\eta}\lvert A-B\rvert^2
    \leq
    \eta(A\mid B)
    \leq
    C_\eta\lvert A-B\rvert^2.
    \label{eq:two-shock-relative-entropy-equivalence}
\end{equation}
It follows from \eqref{eq:two-shock-quantitative-dissipation} and the lower
bound in \eqref{eq:two-shock-relative-entropy-equivalence} that
\begin{equation}
    \bigl\|U(\cdot,\tau)-U_{\mathrm{shift}}(\cdot,\tau)\bigr\|_{L^2(J_R)}
    \leq
    \sqrt{\frac{C_\eta}{a_*}}\,\mathcal{E}(\tau)^{1/2}
    \leq
    \sqrt{\frac{C_\eta}{a_*}}\,\mathcal{E}(0)^{1/2}.
    \label{eq:two-shock-U-shift-L2}
\end{equation}
The triangle inequality, the inclusion \(L^2(J_R)\hookrightarrow L^1(J_R)\),
and \eqref{eq:two-shock-shift-final} give
\begin{align}
    &\bigl\|U(\cdot,\tau)
        -U_{\text{self-similar}}(\cdot,\tau)\bigr\|_{L^1(J_R)}
    \notag\\
    &\quad\leq
    \bigl\|U(\cdot,\tau)-U_{\mathrm{shift}}(\cdot,\tau)\bigr\|_{L^1(J_R)}
    \notag\\
    &\qquad
    +\bigl\|U_{\mathrm{shift}}(\cdot,\tau)
        -U_{\text{self-similar}}(\cdot,\tau)\bigr\|_{L^1(J_R)}
    \notag\\
    &\quad\leq
    \sqrt{2R}\,
    \bigl\|U(\cdot,\tau)-U_{\mathrm{shift}}(\cdot,\tau)\bigr\|_{L^2(J_R)}
    \notag\\
    &\qquad
    +\bigl\|U_{\mathrm{shift}}(\cdot,\tau)
        -U_{\text{self-similar}}(\cdot,\tau)\bigr\|_{L^1(\mathbb{R})}
    \notag\\
    &\quad\leq
    \left[
        \sqrt{\frac{2RC_\eta}{a_*}}
        +
        \sqrt{K_{\mathrm d}\tau\delta}
    \right]
    \mathcal{E}(0)^{1/2}.
    \label{eq:two-shock-final-L1-energy}
\end{align}
Because \(U_{\text{self-similar}}(\cdot,0)=U_{\mathrm{shift}}(\cdot,0)\), the upper bound in
\eqref{eq:two-shock-relative-entropy-equivalence} implies
\begin{align}
    \mathcal{E}(0)
    &=
    \int_{I_\tau}
        a(x,0)\eta\bigl(U(x,0)\mid U_{\text{self-similar}}(x,0)\bigr)\,dx
    \notag\\
    &\leq
    a^*C_\eta
    \bigl\|U(\cdot,0)-U_{\text{self-similar}}(\cdot,0)\bigr\|_{L^2(I_\tau)}^2.
    \label{eq:two-shock-initial-energy-equivalence}
\end{align}
Consequently,
\begin{equation}
    \begin{aligned}
    \bigl\|U(\cdot,\tau)-U_{\text{self-similar}}(\cdot,\tau)\bigr\|_{L^1(J_R)}
    &\leq
    \mathfrak{C}_{R,\tau,\delta}
    \bigl\|U(\cdot,0)-U_{\text{self-similar}}(\cdot,0)\bigr\|_{L^2(I_\tau)},
    \\
    \mathfrak{C}_{R,\tau,\delta}
    &:={}
    \sqrt{a^*C_\eta}
    \left[
        \sqrt{\frac{2RC_\eta}{a_*}}
        +
        \sqrt{K_{\mathrm d}\tau\delta}
    \right].
    \end{aligned}
    \label{eq:two-shock-final-L1}
\end{equation}
Finally, assume that \(U\), \(U_{\text{self-similar}}\), and \(U_{\mathrm{shift}}\) take values in a fixed
bounded set \(\mathcal{V}_0\), and set
\begin{equation}
    M_{\mathcal{V}_0}
    :=
    \sup\bigl\{\lvert A-B\rvert:A,B\in\mathcal{V}_0\bigr\}.
    \label{eq:two-shock-state-diameter}
\end{equation}
Using the simple interpolation estimate for any function $g$
\begin{equation}
    \lVert g\rVert_{L^2(J_R)}^2
    \leq
    \lVert g\rVert_{L^1(J_R)}
    \lVert g\rVert_{L^\infty(J_R)},
\end{equation}
we conclude from \eqref{eq:two-shock-final-L1} that
\begin{equation}
    \begin{aligned}
    \bigl\|U(\cdot,\tau)-U_{\text{self-similar}}(\cdot,\tau)\bigr\|_{L^2(J_R)}
    &\leq
    \bigl(M_{\mathcal{V}_0}\mathfrak{C}_{R,\tau,\delta}\bigr)^{1/2}
    \\
    &\quad\times
    \bigl\|U(\cdot,0)-U_{\text{self-similar}}(\cdot,0)\bigr\|_{L^2(I_\tau)}^{1/2}.
    \end{aligned}
    \label{eq:two-shock-final-holder-estimate}
\end{equation}
In particular, for \(0\leq\tau\leq T\), all
coefficients may be absorbed into a constant \(K>0\), yielding
\begin{equation}
    \boxed{
    \bigl\|U(\cdot,\tau)-U_{\text{self-similar}}(\cdot,\tau)\bigr\|_{L^2((-R,R))}
    \leq
    K
    \bigl\|U(\cdot,0)-U_{\text{self-similar}}(\cdot,0)\bigr\|_{L^2((-R-s\tau,R+s\tau))}^{1/2}
    }.
    \label{eq:two-shock-holder-concise}
\end{equation}

\section{Proof of negative side of \Cref{main_theorem} via convex integration}\label{sec:neg}

In this section, we prove the non-uniqueness stated in \Cref{main_theorem}.

\paragraph{Step 1: The constitutive inclusion with the entropy pair.}

Let
\[
\mathcal{X}:=\{X_1,\ldots,X_N\}\subset \mathbb{R}^{3\times 2}
\]
be the finite configuration constructed above in \Cref{matlab_lemma} or \Cref{T_6_with_ordering}, equipped with the chosen
\(T_N\)-ordering. By the  construction in \Cref{pressure_construction}, there exists a smooth flux
\[
F=(F_1,F_2)^{\mathsf T}\colon\mathcal{V}\to\mathbb{R}^2
\]
and a smooth entropy--entropy flux pair
\[
(\eta,q)\colon\mathcal{V}\to\mathbb{R}\times\mathbb{R},
\qquad
\nabla q(U)=DF(U)^{\mathsf T}\nabla\eta(U),
\]
such that
\[
X_i=G(U_i)
\qquad\text{for some }U_i\in\mathcal{V},\quad i=1,\ldots,N,
\]
where $G(U)$ is from \eqref{eq:constitutive-surface}
Thus,
\[
\mathcal{X}\subset K_{F,\eta,q}.
\]

\paragraph{Step 2: Construction of the fan subsolution.}

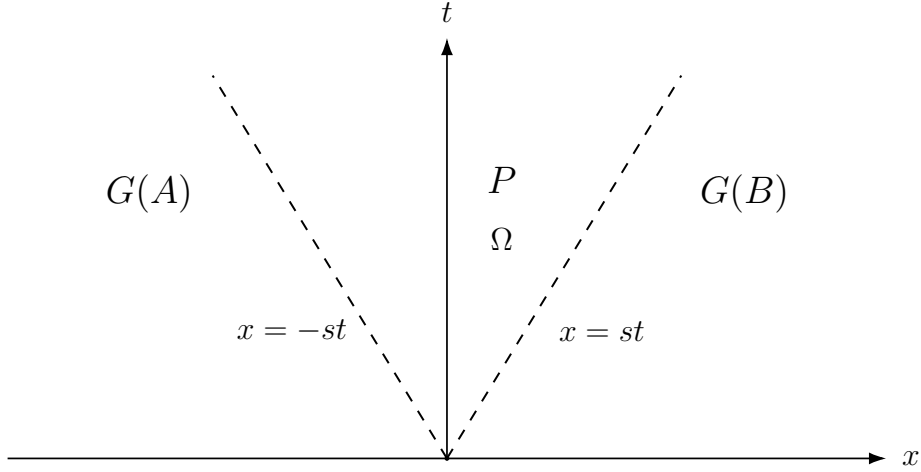
\begin{figure}[t]
    \centering

\begin{tikzpicture}[
    x=1.25cm,
    y=1.25cm,
    line cap=butt,
    line join=round,
    every node/.style={font=\large}
]
    \draw[-{Latex[length=2.3mm,width=1.6mm]},line width=0.65pt]
        (-4.65,0) -- (4.65,0)
        node[right=2pt] {$x$};
    \draw[-{Latex[length=2.3mm,width=1.6mm]},line width=0.65pt]
        (0,0) -- (0,4.45)
        node[above=2pt] {$t$};

    \draw[
        line width=0.75pt,
        dash pattern=on 4.2pt off 4.2pt
    ]
        (0,0) -- (-2.48,4.05);
    \draw[
        line width=0.75pt,
        dash pattern=on 4.2pt off 4.2pt
    ]
        (0,0) -- (2.48,4.05);

    \fill (0,0) circle[radius=0.9pt];

    \node[font=\Large] at (-3.15,2.80) {$G(A)$};
    \node[font=\Large] at (0.58,2.95) {$P$};
    \node[font=\large] at (0.58,2.32) {$\Omega$};
    \node[font=\Large] at (3.15,2.80) {$G(B)$};

    \node[fill=white,inner sep=1.5pt] at (-1.64,1.35) {$x=-st$};
    \node[fill=white,inner sep=1.5pt] at (1.64,1.35) {$x=st$};
\end{tikzpicture}

    \caption{
        Values of the matrix field $\overline{M}$.
        The  interior is
        \(\Omega\), where
        $\overline{M}=P$.
        The exterior states are \(U_R\) and \(U_L\).}
    \label{fig:subsolution-fan}
\end{figure}

By \Cref{sec:exact_match}, there exist $A,B$ in the state space $\mathcal{V}$ such that  $G(A)-P$ is rank-one and $G(B)-P$ is rank-one (when we restrict to the first two rows), where $P$ is from the parameterization of the $T_N$ from \Cref{matlab_lemma} or \Cref{T_6_with_ordering}.

We then define a matrix field, 
\[
\overline{M}
=
\begin{cases}
G(A), & x<-st,\\
P, & |x|<st,\\
G(B), & x>st,
\end{cases}
\]
where $s$ is the speed from \Cref{sec:exact_match}. See \Cref{fig:subsolution-fan}.

Suppose an interior convex-integration theorem supplies a Lipschitz
perturbation $\phi$ with $(x,t)\mapsto \phi(x,t)\in\mathbb{R}^3$, and with zero trace on the entire boundary of the wedge $\{\abs{x}<st\}$,
such that
\[
P+D\phi \in K_{F,\eta,q}
\qquad
\text{almost everywhere inside the wedge}.
\]
Extend $\phi$ by zero outside and set
\[
M=\overline{M}+D\phi.
\]

For the row-wise operator
\[
\mathcal{C}(M)
=
\partial_t M_{\cdot,1}
-
\partial_x M_{\cdot,2},
\]
commutation of distributional derivatives gives
\[
\mathcal{C}(D\phi)=0.
\]
Consequently,
\[
\mathcal{C}(M)
=
\begin{pmatrix}
0\\
0\\
\mathcal{P}_L\delta_{x=-st}
+
\mathcal{P}_R\delta_{x=st}
\end{pmatrix}.
\]

Once we have $U$ such that
\[
M=G(U)
\qquad
\text{almost everywhere},
\]
this is exactly the pair of conservation laws and the entropy inequality. The entropy terms $\mathcal P_L$ and $\mathcal P_R$ are negative as we know from  \Cref{sec:entropy_calc}

\paragraph{Step 3: Convex integration}

For general systems \eqref{eq:system}, the construction of the Lipschitz perturbation $\phi$ exists by \Cref{transverse_check_lemma} and the general convex-integration result \Cref{main_prop}. However in the precise case of the $p$-system, due to the symmetry in the constitutive set \eqref{eq:constitutive-surface} we  use a convex integration scheme which allows for linear constraints (in this case, $X_{2,1}=X_{1,2}$ for all $X$). This follows immediately from \cite[Theorem 1.3]{Kim2020ConvexIntegration} and repeating the argument from \cite[p.~8]{MR2048569}. We reproduce \cite[Theorem 1.3]{Kim2020ConvexIntegration} below in \Cref{sec:kim_theorem} as \Cref{kim_theorem}. \Cref{kim_theorem} is formulated for bounded Lipschitz domains, whereas we want the whole infinite wedge. So we apply it on bounded wedge pieces and glue zero-boundary perturbations.

Finally, the fan geometry gives the Riemann initial trace
\[
U^0(x)
=
\begin{cases}
U_L,
&
x<0,
\\[1mm]
U_R,
&
x>0.
\end{cases}
\]

Indeed, for each \(t>0\), the set on which \(U(\cdot,t)\) can differ
from the corresponding Riemann state is contained in an interval of
length \(O(t)\). Since \(U\) is uniformly bounded, it follows that
\[
U(\cdot,t)\longrightarrow U^0
\qquad\text{in }L^1_{\mathrm{loc}}(\mathbb{R})
\quad\text{as }t\to0^+.
\]
The same argument also gives
\[
\eta(U(\cdot,t))
\longrightarrow
\eta(U^0)
\qquad\text{in }L^1_{\mathrm{loc}}(\mathbb{R})
\quad\text{as }t\to0^+.
\]

Remark that we get infinitely many different solutions by varying the parameter $\delta$ in \Cref{main_prop} (or \Cref{kim_theorem}).

\section{Proof of nonexistence of $T_N$ for $p$-system}\label{sec:no-TN-p-system}

We consider the one-dimensional $p$-system in Lagrangian coordinates under the
assumptions $p'<0$ and $p''>0$, augmented by its strictly convex mechanical
entropy and the associated entropy flux. We prove a statement stronger than
the exclusion of ordered $T_N$-configurations: no tuple of length $N\geq2$ in the
constitutive set can satisfy the algebraic ordered $T_N$ parameterization,
even without imposing distinctness or pairwise-rank conditions on the labeled
states. In particular, ordered $T_6$-configurations are impossible. The proof
combines a strictly positive barycentric representation of every intermediate
point of the cyclic construction (see \Cref{folklore_lemma}, above) with a minimal-specific-volume argument, an
exact relative-entropy row reduction, and strict moment inequalities implied
by the convexity of the pressure.

\subsection{The constitutive set and the statement}

Consider
\begin{equation}
    \begin{cases}
        v_t-u_x=0,\\
        u_t+p(v)_x=0,
    \end{cases}
    \qquad v>0,
    \label{eq:p-system-no-TN}
\end{equation}
where $p\in C^2((0,\infty))$ satisfies
\begin{equation}
    p'(v)<0,
    \qquad
    p''(v)>0
    \qquad (v>0).
    \label{eq:pressure-assumptions}
\end{equation}
Let $\Pi$ be any primitive of $-p$,
\begin{equation}
    \Pi'(v)=-p(v),
\end{equation}
For $U=(v,u)^{\mathsf T}$, set
\begin{equation}
    \eta(U)=\frac{u^2}{2}+\Pi(v),
    \qquad
    q(U)=u p(v).
    \label{eq:entropy-pair}
\end{equation}
With $F(U)=(-u,p(v))^{\mathsf T}$, we can consider the map $G$ defined in \eqref{eq:constitutive-surface} and the set $K_{F,\eta,q}$ defined in \eqref{manifold_K}.

\begin{theorem}\label{thm:no-TN}
Let $N\geq2$ and let
\begin{equation}
    U_i=(v_i,u_i)^{\mathsf T}\in(0,\infty)\times\mathbb{R},
    \qquad
    X_i=G(U_i)\in K_{F,\eta,q},
    \qquad i=1,\ldots,N.
\end{equation}
Then $(X_1,\ldots,X_N)$ does not admit an algebraic ordered $T_N$
parameterization. Consequently, for every $N\geq4$, the product set $\mathcal{K}_{F,\eta,q}$ contains no
ordered $T_N$-configuration in the sense of
Definition~\ref{def:ordered-TN}.
\end{theorem}

We use the standard notation $P$ for the base matrix of the
$T_N$-configuration. It should not be confused with the scalar primitive
$\Pi$ of the pressure.

\subsection{The extremal-state contradiction}

\begin{proof}[Proof of Theorem~\ref{thm:no-TN}]
Suppose, for contradiction, that the matrices $X_i=G(U_i)$ are in $T_N$ configuration. Choose an index $i$ such that
\begin{equation}
    v_i=v_*:=\min_{1\leq j\leq N}v_j.
    \label{eq:min-v}
\end{equation}
Use the coefficients $\mu_j^{(i)}$ supplied by
Lemma~\ref{folklore_lemma}. For any labeled family
$a=(a_1,\ldots,a_N)$ of scalars, vectors, or matrices, write
\begin{equation}
    \E[a]:=\sum_{j=1}^N\mu_j^{(i)}a_j.
    \label{eq:expectation}
\end{equation}
We use the standard shorthand $\E[x]=\sum_j\mu_j^{(i)}x_j$,
$\E[yz]=\sum_j\mu_j^{(i)}y_jz_j$, and similarly for other expressions.
In particular,
\begin{equation}
    P_i=\E[X].
    \label{eq:Pi-expectation}
\end{equation}

For each $j$, set
\begin{equation}
    x_j:=v_j-v_*\geq0,
    \qquad
    y_j:=u_j-u_i,
    \label{eq:x-y}
\end{equation}
and define
\begin{equation}
    \phi(x):=p(v_*)-p(v_*+x),
    \qquad
    z_j:=\phi(x_j)=p(v_*)-p(v_j).
    \label{eq:phi-z}
\end{equation}
Also define
\begin{align}
    H(x)
    &:=p(v_*)x+\bigl(\Pi(v_*+x)-\Pi(v_*)\bigr)
    \notag\\
    &=p(v_*)x-\int_{v_*}^{v_*+x}p(s)\,ds
    \notag\\
    &=\int_0^x\bigl(p(v_*)-p(v_*+r)\bigr)\,dr
    =\int_0^x\phi(r)\,dr.
    \label{eq:H}
\end{align}

\subsubsection*{Exact normalization of the entropy row}

Write $p_*:=p(v_*)$ and $u_*:=u_i$, and introduce
\begin{equation}
    L_i:=
    \begin{pmatrix}
        1&0&0\\
        0&1&0\\
        p_*&-u_*&1
    \end{pmatrix}.
    \label{eq:Li}
\end{equation}
Since $\det L_i=1$, left multiplication by $L_i$ preserves matrix rank.
We now calculate $L_i\bigl(G(U_j)-G(U_i)\bigr)$ entry by entry.

The first two rows are
\begin{equation}
    \begin{pmatrix}
        x_j&y_j\\
        y_j&z_j
    \end{pmatrix},
    \label{eq:top-two-rows}
\end{equation}
because
\begin{equation}
    -p(v_j)-(-p_*)=p_*-p(v_j)=z_j.
\end{equation}

For the first entry of the transformed third row,
\begin{align}
    \eta(U_j)-\eta(U_i)
    &=
    \frac{(u_*+y_j)^2-u_*^2}{2}
    +\bigl(\Pi(v_*+x_j)-\Pi(v_*)\bigr)
    \notag\\
    &=u_*y_j+\frac{y_j^2}{2}
    +\bigl(\Pi(v_*+x_j)-\Pi(v_*)\bigr).
\end{align}
The row operation adds $p_*x_j-u_*y_j$, so
\begin{align}
    &p_*x_j-u_*y_j+\eta(U_j)-\eta(U_i)
    \notag\\
    &\qquad=
    \frac{y_j^2}{2}
    +p_*x_j+\bigl(\Pi(v_*+x_j)-\Pi(v_*)\bigr)
    \notag\\
    &\qquad=\frac{y_j^2}{2}+H(x_j).
    \label{eq:third-row-first-entry}
\end{align}

For the second entry, use
\begin{equation}
    u_j=u_*+y_j,
    \qquad
    p(v_j)=p_*-z_j.
\end{equation}
The original third-row difference is
\begin{equation}
    -q(U_j)+q(U_i)
    =-(u_*+y_j)(p_*-z_j)+u_*p_*.
\end{equation}
The row operation adds $p_*y_j-u_*z_j$. Therefore
\begin{align}
    &-(u_*+y_j)(p_*-z_j)+u_*p_*+p_*y_j-u_*z_j
    \notag\\
    &\qquad=
    -u_*p_*+u_*z_j-y_jp_*+y_jz_j
    +u_*p_*+p_*y_j-u_*z_j
    \notag\\
    &\qquad=y_jz_j.
    \label{eq:third-row-second-entry}
\end{align}
Combining \eqref{eq:top-two-rows},
\eqref{eq:third-row-first-entry}, and
\eqref{eq:third-row-second-entry}, we obtain the exact identity
\begin{equation}
    L_i\bigl(G(U_j)-G(U_i)\bigr)
    =
    \begin{pmatrix}
        x_j&y_j\\
        y_j&z_j\\
        \dfrac{y_j^2}{2}+H(x_j)&y_jz_j
    \end{pmatrix}.
    \label{eq:transformed-difference}
\end{equation}
Equivalently, for $U_j=(v_j,u_j)^{\mathsf T}$ and $U_i=(v_*,u_*)^{\mathsf T}$, define the relative
entropy and relative entropy flux by
\begin{align}
    \eta(U_j\mid U_i)
    &:=
    \eta(U_j)-\eta(U_i)-\nabla\eta(U_i)\cdot(U_j-U_i),
    \label{eq:relative-entropy-definition}\\
    q(U_j;U_i)
    &:=
    q(U_j)-q(U_i)
    -\nabla\eta(U_i)\cdot\bigl(F(U_j)-F(U_i)\bigr).
    \label{eq:relative-flux-definition}
\end{align}
The calculations above show precisely that
\begin{equation}
    \eta(U_j\mid U_i)=\frac{y_j^2}{2}+H(x_j),
    \qquad
    q(U_j;U_i)=-y_jz_j.
    \label{eq:relative-pair-formulas}
\end{equation}
Thus the last row on the right-hand side of
\eqref{eq:transformed-difference} is
\begin{equation}
    \bigl(\eta(U_j\mid U_i),-q(U_j;U_i)\bigr).
\end{equation}

Average \eqref{eq:transformed-difference} using
\eqref{eq:expectation}. By \eqref{eq:Pi-expectation},
\begin{align}
    M_i
    &:=\E\!\left[
    L_i\bigl(G(U_j)-G(U_i)\bigr)
    \right]
    \notag\\
    &=L_i\bigl(\E[X]-X_i\bigr)
    =L_i(P_i-X_i).
    \label{eq:Mi-definition}
\end{align}
Note $P_i-X_i=-\kappa_iC_i$ is
nonzero and has rank one. Since $L_i$ is invertible, $M_i$ is also nonzero
and has rank one. Explicitly,
\begin{equation}
    M_i=
    \begin{pmatrix}
        \E[x]&\E[y]\\
        \E[y]&\E[z]\\
        \dfrac12\E[y^2]+\E[H(x)]&\E[yz]
    \end{pmatrix}.
    \label{eq:Mi}
\end{equation}

\subsubsection*{The nonconstant-specific-volume case}

Assume first that the numbers $v_1,\ldots,v_N$ are not all equal. By the
minimality of $v_*$, at least one $x_j$ is strictly positive. Full support of
the weights gives
\begin{equation}
    A:=\E[x]>0.
    \label{eq:A-positive}
\end{equation}
Because $p'<0$, the function $p$ is strictly decreasing, so
\begin{equation}
    x_j>0\quad\Longrightarrow\quad
    z_j=p(v_*)-p(v_*+x_j)>0.
\end{equation}
Consequently,
\begin{equation}
    \E[z]>0.
    \label{eq:Ez-positive}
\end{equation}

The first column of $M_i$ is nonzero because its first entry is $A>0$.
Since $M_i$ has rank one, its second column is a scalar multiple of its first:
there exists $s\in\R$ such that
\begin{equation}
    \operatorname{col}_2(M_i)=s\operatorname{col}_1(M_i).
\end{equation}
Reading the three rows of \eqref{eq:Mi} gives
\begin{equation}
    \E[y]=sA,
    \label{eq:rank-one-row1}
\end{equation}
\begin{equation}
    \E[z]=s\E[y]=s^2A,
    \label{eq:rank-one-row2}
\end{equation}
and
\begin{equation}
    \E[yz]
    =s\left(\frac12\E[y^2]+\E[H(x)]\right).
    \label{eq:rank-one-row3}
\end{equation}
In view of \eqref{eq:A-positive}, \eqref{eq:Ez-positive}, and
\eqref{eq:rank-one-row2}, one has $s\neq0$.

We now use the elementary inequality
\begin{equation}
    \E[(z-sy)^2]\geq0.
\end{equation}
Expanding and then using \eqref{eq:rank-one-row3},
\begin{align}
    \E[(z-sy)^2]
    &=\E[z^2]-2s\E[yz]+s^2\E[y^2]
    \notag\\
    &=\E[z^2]
    -2s^2\left(\frac12\E[y^2]+\E[H(x)]\right)
    +s^2\E[y^2]
    \notag\\
    &=\E[z^2]-2s^2\E[H(x)].
\end{align}
It follows that
\begin{equation}
    2s^2\E[H(x)]\leq\E[z^2].
\end{equation}
Substitution of $s^2=\E[z]/A$ from \eqref{eq:rank-one-row2}, followed by
multiplication by $A>0$, gives
\begin{equation}
    2\E[z]\E[H(x)]
    \leq A\E[z^2].
    \label{eq:first-moment-inequality}
\end{equation}

We next derive the strict reverse inequality. From
\eqref{eq:pressure-assumptions} and \eqref{eq:phi-z},
\begin{equation}
    \phi(0)=0,
    \qquad
    \phi'(x)=-p'(v_*+x)>0,
    \qquad
    \phi''(x)=-p''(v_*+x)<0.
    \label{eq:phi-properties}
\end{equation}
Thus $\phi$ is strictly increasing and strictly concave. Define
\begin{equation}
    r(x):=\frac{\phi(x)}{x},
    \qquad x>0.
    \label{eq:r}
\end{equation}
We verify directly that $r$ is strictly decreasing. Differentiation gives
\begin{equation}
    r'(x)=\frac{x\phi'(x)-\phi(x)}{x^2}.
\end{equation}
Since $\phi''<0$, the derivative $\phi'$ is strictly decreasing. Hence, for
$0<t<x$,
\begin{equation}
    \phi'(t)>\phi'(x),
\end{equation}
and therefore
\begin{equation}
    \phi(x)=\int_0^x\phi'(t)\,dt
    >\int_0^x\phi'(x)\,dt
    =x\phi'(x).
\end{equation}
It follows that $r'(x)<0$.

Let
\begin{equation}
    J_+:=\{j:x_j>0\}.
\end{equation}
Define a probability distribution on $J_+$ by
\begin{equation}
    \nu_j:=\frac{\mu_j^{(i)}x_j}{A},
    \qquad j\in J_+.
    \label{eq:nu}
\end{equation}
Indeed,
\begin{equation}
    \sum_{j\in J_+}\nu_j
    =\frac{1}{A}\sum_{j=1}^N\mu_j^{(i)}x_j=1.
\end{equation}
Atoms with $x_j=0$ have zero $\nu$-mass, so $r(0)$ need not be defined in our argument below.

Since $r$ is decreasing and $\phi$ is increasing, every summand in the
following double sum is nonpositive. Therefore
\begin{align}
    &2\Bigl(\E_{\nu}[r\phi]
    -\E_{\nu}[r]\E_{\nu}[\phi]\Bigr)
    \notag\\
    &\qquad=
    \sum_{j,k\in J_+}\nu_j\nu_k
    \bigl(r(x_j)-r(x_k)\bigr)
    \bigl(\phi(x_j)-\phi(x_k)\bigr)
    \leq0.
    \label{eq:covariance-sign}
\end{align}
Consequently,
\begin{equation}
    \E_{\nu}[r\phi]
    \leq\E_{\nu}[r]\E_{\nu}[\phi].
    \label{eq:opposite-Chebyshev}
\end{equation}
Each term has a simple expression in the original weights. Since
$z_j=\phi(x_j)$,
\begin{align}
    \E_{\nu}[r\phi]
    &=
    \sum_{j\in J_+}\frac{\mu_j^{(i)}x_j}{A}
    \frac{\phi(x_j)}{x_j}\phi(x_j)
    =\frac{\E[z^2]}{A},
    \label{eq:nu-rphi}\\
    \E_{\nu}[r]
    &=
    \sum_{j\in J_+}\frac{\mu_j^{(i)}x_j}{A}
    \frac{\phi(x_j)}{x_j}
    =\frac{\E[z]}{A},
    \label{eq:nu-r}\\
    \E_{\nu}[\phi]
    &=
    \sum_{j\in J_+}\frac{\mu_j^{(i)}x_j}{A}\phi(x_j)
    =\frac{\E[xz]}{A}.
    \label{eq:nu-phi}
\end{align}
Substitution of \eqref{eq:nu-rphi}--\eqref{eq:nu-phi} into
\eqref{eq:opposite-Chebyshev}, followed by multiplication by $A^2$, yields
\begin{equation}
    A\E[z^2]
    \leq\E[z]\E[xz].
    \label{eq:second-moment-inequality}
\end{equation}

It remains to obtain strictness. Fix $x>0$. For $0<t<x$, strict concavity of
$\phi$ and $\phi(0)=0$ give
\begin{equation}
    \phi(t)
    =\phi\left(\frac{t}{x}x+\left(1-\frac{t}{x}\right)0\right)
    >\frac{t}{x}\phi(x).
\end{equation}
Integrating from $0$ to $x$ gives
\begin{equation}
    H(x)=\int_0^x\phi(t)\,dt
    >\int_0^x\frac{t}{x}\phi(x)\,dt
    =\frac{x\phi(x)}{2}.
\end{equation}
Thus
\begin{equation}
    x\phi(x)<2H(x),
    \qquad x>0.
    \label{eq:strict-chord}
\end{equation}
For $x_j=0$, both sides of \eqref{eq:strict-chord} vanish. At least one
$x_j$ is positive and its weight is positive, so averaging
\eqref{eq:strict-chord} yields the strict inequality
\begin{equation}
    \E[xz]<2\E[H(x)].
    \label{eq:strict-average}
\end{equation}
Multiplying by $\E[z]>0$ and combining
\eqref{eq:first-moment-inequality},
\eqref{eq:second-moment-inequality}, and
\eqref{eq:strict-average}, we obtain the impossible chain
\begin{equation}
    2\E[z]\E[H(x)]
    \leq A\E[z^2]
    \leq\E[z]\E[xz]
    <2\E[z]\E[H(x)].
    \label{eq:impossible-chain}
\end{equation}
This contradiction rules out the nonconstant-specific-volume case.

\subsubsection*{The constant-specific-volume case}

It remains to consider
\begin{equation}
    v_1=\cdots=v_N=:v_0.
    \label{eq:all-v-equal}
\end{equation}
Choose a new index $i$ satisfying
\begin{equation}
    u_i=\min_{1\leq j\leq N}u_j,
\end{equation}
and use the full-support weights $\mu_j^{(i)}$ associated with this new index.
For a labeled family $a=(a_1,\ldots,a_N)$, define the corresponding average by
\begin{equation}
    \E_i[a]:=\sum_{j=1}^N\mu_j^{(i)}a_j.
    \label{eq:constant-v-expectation}
\end{equation}
Set
\begin{equation}
    y_j:=u_j-u_i\geq0.
\end{equation}
Since all $v_j$ are equal, all $p(v_j)$ are equal. Therefore the upper two
rows of
\begin{equation}
    P_i-X_i
    =\E_i\bigl[G(U_j)-G(U_i)\bigr]
\end{equation}
are
\begin{equation}
    \begin{pmatrix}
        0&\E_i[y]\\
        \E_i[y]&0
    \end{pmatrix}.
    \label{eq:constant-v-minor}
\end{equation}
The matrix $P_i-X_i=-\kappa_iC_i$ has rank one, so every $2\times2$ minor
vanishes. In particular,
\begin{equation}
    0=
    \det
    \begin{pmatrix}
        0&\E_i[y]\\
        \E_i[y]&0
    \end{pmatrix}
    =-\bigl(\E_i[y]\bigr)^2.
\end{equation}
Hence $\E_i[y]=0$. Since every $y_j\geq0$ and every weight
$\mu_j^{(i)}>0$, this forces
\begin{equation}
    y_j=0
    \qquad\text{for every }j.
\end{equation}
Thus all $u_j$ are equal. Together with \eqref{eq:all-v-equal}, this implies
that all $X_j$ coincide, say $X_j=X_*$. The barycentric formula then gives
\begin{equation}
    P_i=\sum_{j=1}^N\mu_j^{(i)}X_j
    =\left(\sum_{j=1}^N\mu_j^{(i)}\right)X_*
    =X_*=X_i.
\end{equation}
This contradicts
\begin{equation}
    X_i-P_i=\kappa_iC_i\neq0.
\end{equation}

Both cases lead to contradictions. The
exclusion of ordered $T_N$-configurations follows immediately.
\end{proof}

\begin{remark}[Where the assumptions enter]\label{rem:assumptions}
The proof uses the hypotheses in sharply separated ways:
\begin{enumerate}[label=\textup{(\roman*)},leftmargin=2.5em]
    \item $C_i\in\mathcal{R}_1$ makes $X_i-P_i$ nonzero and of rank one;
    \item $p'<0$ makes $\phi$ strictly increasing and gives $\E[z]>0$ in the
    nonconstant-specific-volume case;
    \item $p''>0$ makes $\phi$ strictly concave and yields the strict chord
    inequality \eqref{eq:strict-chord};
    \item the mechanical entropy and its entropy flux give exactly the
    transformed third row
    \begin{equation}
        \left(\frac{y^2}{2}+H(x),\,yz\right).
    \end{equation}
\end{enumerate}
No distinctness or pairwise-rank assumption on the matrices $X_i$ is used in
the proof. Thus Theorem~\ref{thm:no-TN} excludes the larger class of ordered $T_N$ configurations which may not verify $\rank(X_i-X_j)>1$.
\end{remark}

\section{Nonexistence of $T_\infty$ configurations}\label{sec:no-Tinfty}

Iqbal \cite{Iqbal2000} starts with a smooth closed curve $\alpha$ in matrix space whose tangent has rank one and perturbs it in its tangent direction. In the notation used below, the resulting evolute has the form
\begin{equation}\label{eq:evolute-intro}
 C(t)=\alpha(t)+\rho(t)\dot\alpha(t),
 \qquad \rho(t)>0.
\end{equation}
Under a smallness hypothesis on the positive tangent displacement, Iqbal constructs a nontrivial homogeneous laminate whose support is the entire image of $C$.

We have the following theorem:

\begin{theorem}[Embedded evolute obstruction]\label{thm:main_t_infty}
Let $\mathcal{V}\subset\mathbb{R}^2$ be open and convex, and let
\[
 F\in C^4(\mathcal{V};\mathbb{R}^2),
 \qquad
 \eta,q\in C^4(\mathcal{V})
\]
satisfy the entropy compatibility relation \eqref{eq:entropy-compatibility}. Assume that
\begin{enumerate}[label=\textup{(\roman*)}]
 \item the system \eqref{eq:system} is strictly hyperbolic on $\mathcal{V}$;
 \item both characteristic families are genuinely nonlinear in the sense of \eqref{eq:genuine-nonlinearity};
 \item $D^2\eta(U)$ is positive definite for every $U\in\mathcal{V}$.
\end{enumerate}
Let
\[
 \alpha\in C^4(\SoneL;\mathbb{R}^{3\times2}),
 \qquad
 \rho\in C^3(\SoneL;(0,\infty))
\]
satisfy
\begin{equation}\label{eq:rank-one-tangent-evolute}
 \rank\dot\alpha(t)=1
 \qquad\text{for every }t\in\SoneL.
\end{equation}
Define
\begin{equation}\label{eq:C-def}
 C(t):=\alpha(t)+\rho(t)\dot\alpha(t).
\end{equation}
Then $C$ cannot simultaneously be a regular embedded closed curve and satisfy
\begin{equation}\label{eq:C-in-K}
 C(t)\in K_{F,\eta,q}
 \qquad\text{for every }t\in\SoneL.
\end{equation}
\end{theorem}

\subsection{Proof of \Cref{thm:main_t_infty}}
\subsection{Setup}
Suppose, toward a contradiction, that
\begin{equation}
C(t)=G(U(t))
\label{eq:C-equals-G}
\end{equation}
for every $t\in\SoneL$, where
\[
\rank\dot\alpha(t)=1,
\qquad
\rho(t)>0,
\]
and $C$ is a regular embedded closed curve.

Write the columns of $\alpha$ as
\[
\alpha(t)=\bigl(a(t)\mid b(t)\bigr),
\qquad
 a(t),b(t)\in\R^3,
\]
and decompose
\[
a(t)
=
\begin{pmatrix}
 a_s(t)\\
 a_3(t)
\end{pmatrix},
\qquad
 a_s(t)\in\R^2.
\]
Set
\begin{equation}
\xi(t):=\dot a_s(t),
\qquad
n(t):=\nabla\eta(U(t)).
\label{eq:xi-n-def}
\end{equation}
Denote the two columns of $C$ by
\[
A(t):=a(t)+\rho(t)\dot a(t),
\qquad
B(t):=b(t)+\rho(t)\dot b(t).
\]
Because $C(t)=G(U(t))$, we have
\begin{equation}
A(t)
=
\begin{pmatrix}
 U(t)\\
 \eta(U(t))
\end{pmatrix},
\qquad
B(t)
=
\begin{pmatrix}
 -F(U(t))\\
 -q(U(t))
\end{pmatrix}.
\label{eq:A-B-graph}
\end{equation}

\subsection{Detailed derivation of an identity}

Work first on an interval on which $\dot a(t)\neq0$. Since
\[
\rank\dot\alpha(t)
=
\rank\bigl(\dot a(t)\mid\dot b(t)\bigr)
=1,
\]
the columns $\dot a(t)$ and $\dot b(t)$ are linearly dependent. Therefore, there exists a smooth scalar function $\mu(t)$ such that
\begin{equation}
\dot b(t)=\mu(t)\dot a(t).
\label{eq:b-dot-mu-a-dot}
\end{equation}
Differentiate $B=b+\rho\dot b$. Using \eqref{eq:b-dot-mu-a-dot}, we obtain
\begin{align*}
\dot B
&=\dot b+\dot\rho\,\dot b+\rho\ddot b\\
&=\mu\dot a+\dot\rho\,\mu\dot a
  +\rho\bigl(\dot\mu\,\dot a+\mu\ddot a\bigr)\\
&=\mu\bigl(\dot a+\dot\rho\,\dot a+\rho\ddot a\bigr)
  +\rho\dot\mu\,\dot a.
\end{align*}
But
\[
\dot A=\dot a+\dot\rho\,\dot a+\rho\ddot a.
\]
Consequently,
\begin{equation}
\dot B=\mu\dot A+\rho\dot\mu\,\dot a.
\label{eq:B-dot}
\end{equation}
This identity explains why the factor $\rho\dot\mu$ appears. All the terms involving $\dot\rho$ are already contained in $\mu\dot A$; the only extra term comes from differentiating the varying proportionality factor $\mu(t)$.

Now split \eqref{eq:B-dot} into its first two components and its third component. From \eqref{eq:A-B-graph},
\begin{equation}
\dot A_s=\dot U,
\qquad
\dot A_3=\frac{d}{dt}\eta(U)=n\cdot\dot U.
\label{eq:A-components}
\end{equation}
Likewise,
\begin{equation}
\dot B_s=-DF(U)\dot U,
\qquad
\dot B_3=-\nabla q(U)\cdot\dot U.
\label{eq:B-components}
\end{equation}
Taking the first two components of \eqref{eq:B-dot} gives
\begin{equation}
-DF(U)\dot U=\mu\dot U+\rho\dot\mu\,\xi.
\label{eq:first-components}
\end{equation}
Taking the third component gives
\begin{equation}
-\nabla q(U)\cdot\dot U
=\mu n\cdot\dot U+\rho\dot\mu\,\dot a_3.
\label{eq:third-component}
\end{equation}
On the other hand, the entropy compatibility relation \eqref{eq:entropy-compatibility} yields
\[
-\nabla q(U)\cdot\dot U
=-\bigl(DF(U)^{\mathsf T}\nabla\eta(U)\bigr)\cdot\dot U
=n\cdot\bigl(-DF(U)\dot U\bigr).
\]
Substituting \eqref{eq:first-components} into the right-hand side gives
\begin{equation}
\begin{aligned}
-\nabla q(U)\cdot\dot U
&=n\cdot\bigl(\mu\dot U+\rho\dot\mu\,\xi\bigr)\\
&=\mu n\cdot\dot U+\rho\dot\mu\,n\cdot\xi.
\end{aligned}
\label{eq:compatibility-substitution}
\end{equation}
Comparing \eqref{eq:third-component} and \eqref{eq:compatibility-substitution}, the common term
\[
\mu n\cdot\dot U
\]
cancels, leaving
\[
\rho\dot\mu\,\dot a_3
=
\rho\dot\mu\,n\cdot\xi.
\]
Therefore,
\begin{equation}
\rho(t)\dot\mu(t)
\bigl(\dot a_3(t)-n(t)\cdot \xi(t)\bigr)=0.
\label{eq:key-identity}
\end{equation}

\subsection{Why genuine nonlinearity forces $\dot a_3=n\cdot\xi$}

Since $\rho>0$, equation \eqref{eq:key-identity} gives
\begin{equation}
\dot a_3=n\cdot\xi
\label{eq:a3-local}
\end{equation}
at every point where $\dot\mu\neq0$. We now show that such points are dense.

Suppose that $\dot\mu=0$ on a nonempty open interval $I$. Then $\mu$ is constant on $I$, and \eqref{eq:first-components} becomes
\begin{equation}
DF(U(t))\dot U(t)=-\mu\dot U(t).
\label{eq:eigenvector-equation}
\end{equation}
Because $C(t)=G(U(t))$ is regular, $\dot U(t)\neq0$. Indeed, if $\dot U=0$, then
\[
\frac{d}{dt}G(U(t))=DG(U(t))[\dot U(t)]=0,
\]
contradicting regularity of $C$.

Thus, \eqref{eq:eigenvector-equation} says that $\dot U(t)$ is a right eigenvector of $DF(U(t))$ with eigenvalue $-\mu$. By strict hyperbolicity, on the connected interval $I$ there is a fixed characteristic family $k\in\{1,2\}$ such that
\begin{equation}
\lambda_k(U(t))=-\mu,
\qquad
\dot U(t)=\gamma(t)r_k(U(t)),
\qquad
\gamma(t)\neq0.
\label{eq:eigenfamily}
\end{equation}
Since $\mu$ is constant,
\[
0
=
\frac{d}{dt}\lambda_k(U(t))
=
\nabla\lambda_k(U(t))\cdot\dot U(t).
\]
Using \eqref{eq:eigenfamily},
\begin{equation}
0
=
\gamma(t)\nabla\lambda_k(U(t))\cdot r_k(U(t)).
\label{eq:gnl-contradiction}
\end{equation}
This contradicts genuine nonlinearity,
\[
\nabla\lambda_k(U(t))\cdot r_k(U(t))\neq0.
\]

Therefore, $\dot\mu$ cannot vanish identically on any open interval. Hence
\[
\{t:\dot\mu(t)\neq0\}
\]
is dense on every interval where $\dot a\neq0$. By \eqref{eq:key-identity} and continuity,
\[
\dot a_3=n\cdot\xi
\]
throughout every such interval.

Finally, the set
\[
\Omega:=\{t:\dot a(t)\neq0\}
\]
is itself dense. Indeed, if $\dot a=0$ on an interval, then the first column
\[
A=a+\rho\dot a=a
\]
would be constant there. Since $A_s=U$, this would make $U$ constant, and therefore $C=G(U)$ constant, again contradicting regularity.

It follows by continuity that
\begin{equation}
\dot a_3(t)=n(t)\cdot \xi(t)
\qquad
\text{for every }t.
\label{eq:a3-global}
\end{equation}

\subsection{A second orthogonality identity}

From the first column of \eqref{eq:A-B-graph},
\begin{equation}
U=a_s+\rho\xi,
\qquad
\eta(U)=a_3+\rho\dot a_3.
\label{eq:first-column-identities}
\end{equation}
Differentiating the first identity gives
\begin{equation}
\dot U=(1+\dot\rho)\xi+\rho\dot\xi.
\label{eq:U-dot}
\end{equation}
Differentiating the second identity gives
\begin{equation}
n\cdot\dot U
=(1+\dot\rho)\dot a_3+\rho\ddot a_3.
\label{eq:eta-derivative}
\end{equation}
Using \eqref{eq:a3-global},
\[
\dot a_3=n\cdot\xi,
\]
and therefore
\begin{equation}
\ddot a_3=\dot n\cdot\xi+n\cdot\dot\xi.
\label{eq:a3-second}
\end{equation}
Substituting \eqref{eq:U-dot}, \eqref{eq:a3-global}, and \eqref{eq:a3-second} into \eqref{eq:eta-derivative}, we get
\[
(1+\dot\rho)n\cdot\xi+\rho n\cdot\dot\xi
=
(1+\dot\rho)n\cdot\xi+\rho\dot n\cdot\xi+\rho n\cdot\dot\xi.
\]
After cancellation and using $\rho>0$,
\begin{equation}
\dot n(t)\cdot \xi(t)=0.
\label{eq:second-orthogonality}
\end{equation}
Thus $\xi(t)$ is normal to the plane curve $n(t)$.

\subsection{Strict convexity determines the sign of the curvature}

Use ``strictly convex entropy'' in the standard hyperbolic-systems sense
\[
D^2\eta(U)>0.
\]
Since $n=\nabla\eta(U)$,
\[
\dot n=D^2\eta(U)\dot U.
\]
Consequently,
\begin{equation}
\dot n\cdot\dot U
=
\dot U^{\mathsf T}D^2\eta(U)\dot U
>0,
\label{eq:strict-convexity-positive}
\end{equation}
because $\dot U\neq0$.

Using \eqref{eq:U-dot} and \eqref{eq:second-orthogonality},
\begin{align*}
\dot n\cdot\dot U
&=(1+\dot\rho)\dot n\cdot\xi+\rho\dot n\cdot\dot\xi\\
&=\rho\dot n\cdot\dot\xi.
\end{align*}
Thus
\begin{equation}
\dot n(t)\cdot\dot \xi(t)>0.
\label{eq:n-xi-positive}
\end{equation}

Notice also that $\xi$ can never vanish. If $\xi(t_0)=0$, differentiating \eqref{eq:second-orthogonality} gives
\[
0
=
\frac{d}{dt}(\dot n\cdot\xi)
=
\ddot n\cdot\xi+\dot n\cdot\dot\xi.
\]
At $t=t_0$, this would imply
\[
\dot n(t_0)\cdot\dot\xi(t_0)=0,
\]
contradicting \eqref{eq:n-xi-positive}.
Therefore,
\begin{equation}
\xi(t)\neq0
\qquad
\text{for every }t.
\label{eq:xi-nonzero}
\end{equation}

\subsection{Why $n(\SoneL)$ is a strictly convex oval}

Because the first two entries in the first column of $G(U)$ are exactly $U$, the map $G$ is injective. Thus, if $C=G(U)$ is embedded, then $U$ is an embedded closed plane curve in $\mathcal{V}$.

Moreover, $\nabla\eta$ is injective on the convex state space $\mathcal{V}$. Indeed, for $U_1\neq U_2$,
\begin{equation}
\begin{aligned}
&\bigl(\nabla\eta(U_1)-\nabla\eta(U_2)\bigr)\cdot(U_1-U_2)\\
&\qquad
=
\int_0^1
(U_1-U_2)^{\mathsf T}
D^2\eta\bigl(U_2+s(U_1-U_2)\bigr)
(U_1-U_2)\,ds
>0.
\end{aligned}
\label{eq:gradient-injective}
\end{equation}
Hence $n=\nabla\eta(U)$ is also a regular embedded closed plane curve.

Let
\[
T=\frac{\dot n}{\lvert\dot n\rvert}
\]
be its unit tangent. From \eqref{eq:second-orthogonality} and \eqref{eq:xi-nonzero}, $\xi$ is a nonzero normal vector. After choosing one of the two possible orientations of the unit normal $N$, we may write
\begin{equation}
\xi=\vartheta N,
\qquad
\vartheta>0.
\label{eq:xi-vartheta-N}
\end{equation}
Let $\kappa$ be the signed curvature, defined by
\[
\dot T=\kappa\lvert\dot n\rvert N,
\qquad
\dot N=-\kappa\lvert\dot n\rvert T.
\]
Then
\begin{equation}
\begin{aligned}
\dot n\cdot\dot\xi
&=
\lvert\dot n\rvert T\cdot
\bigl(\dot\vartheta N+\vartheta\dot N\bigr)\\
&=
-\vartheta\kappa\lvert\dot n\rvert^2.
\end{aligned}
\label{eq:curvature-computation}
\end{equation}
Since the left-hand side is strictly positive,
\begin{equation}
\kappa
=
-\frac{\dot n\cdot\dot\xi}{\vartheta\lvert\dot n\rvert^2}
<0.
\label{eq:kappa-negative}
\end{equation}
Thus the signed curvature is everywhere nonzero and has one fixed sign. A regular simple closed plane curve with curvature of one strict sign bounds a strictly convex domain. Equivalently, its tangent angle is strictly monotone and makes exactly one complete turn.

Therefore $n(\SoneL)$ is the boundary of a strictly convex compact set $\mathcal D$, and $\xi$ is everywhere either an inward-pointing normal or everywhere an outward-pointing normal.

Choose $z\in\interior\mathcal D$. Then
\[
(n(t)-z)\cdot \xi(t)
\]
has one strict sign for every $t$. Consequently,
\begin{equation}
\int_0^L(n(t)-z)\cdot \xi(t)\,dt\neq0.
\label{eq:nonzero-integral}
\end{equation}
Because $\alpha$ is closed,
\begin{equation}
\int_0^L \xi(t)\,dt
=
\int_0^L\dot a_s(t)\,dt
=
a_s(L)-a_s(0)
=0.
\label{eq:xi-integral-zero}
\end{equation}
It follows from \eqref{eq:nonzero-integral}--\eqref{eq:xi-integral-zero} that
\begin{equation}
\int_0^L n(t)\cdot \xi(t)\,dt
=
\int_0^L(n(t)-z)\cdot \xi(t)\,dt
\neq0.
\label{eq:n-xi-integral-nonzero}
\end{equation}
But \eqref{eq:a3-global} and the closedness of $a_3$ give
\begin{equation}
\int_0^L n(t)\cdot \xi(t)\,dt
=
\int_0^L\dot a_3(t)\,dt
=
a_3(L)-a_3(0)
=0,
\label{eq:n-xi-integral-zero}
\end{equation}
which contradicts \eqref{eq:n-xi-integral-nonzero}.

Therefore, under the regular embedded interpretation, the constitutive set of a strictly hyperbolic, genuinely nonlinear $2\times2$ system with a strictly convex entropy cannot contain an Iqbal-type closed evolute
\[
C=\alpha+\rho\dot\alpha
\]
with $\rho>0$ and $\rank\dot\alpha=1$.

The laminate property is not needed for the contradiction: the obstruction already rules out the underlying evolute geometry.

\paragraph{Qualification about immersed curves.}
Equations \eqref{eq:key-identity}, \eqref{eq:a3-global}, \eqref{eq:second-orthogonality}, and \eqref{eq:n-xi-positive} remain valid locally for a regular immersed evolute. The embeddedness assumption is used specifically in the implication
\[
\text{one-signed curvature}
\quad\Longrightarrow\quad
n(\SoneL)\text{ bounds a strictly convex domain}.
\]
For a self-intersecting closed immersion, one-signed curvature alone does not automatically give a convex oval. Thus our nonexistence statement should be understood with the embeddedness assumption; the genuinely immersed, self-intersecting case is not settled by this particular global argument.

\appendix

\section{Facts about convex functions}

\subsection{Algebraic inequalities which yield a strictly convex function}

We will take advantage of the following standard fact about convex functions.

\begin{lemma}[Algebraic inequalities which yield a strictly convex function]\label{convex_function_lemma}
Fix $n,N\in\mathbb{N}$. Assume there exists
\[
(x_i,h_i,D_i)\in\mathbb{R}^n\times\mathbb{R}\times\mathbb{R}^n,
\qquad i=1,\ldots,N,
\]
such that the strict inequalities
\begin{equation}
\label{eq:3.10}
    h_j > h_i + D_i\cdot(x_j-x_i)
    \qquad \text{for all } i\neq j
\end{equation}
are verified. Then, there exists a smooth and strictly convex function
$\xi\colon\mathbb{R}^n\to\mathbb{R}$ such that
\[
\xi(x_i)=h_i
\qquad\text{and}\qquad
\nabla\xi(x_i)=D_i,
\qquad\text{for all } i.
\]
\end{lemma}

\begin{remark}
It is well known that such an $\xi$ exists if we only ask for it to be
convex. See e.g.\ \cite[p.~143]{MR2048569}. However, in this  \Cref{convex_function_lemma} $\xi$ is
constructed in such a way to be \emph{strictly convex}.
\end{remark}

The proof of  \Cref{convex_function_lemma} is standard and can be found in e.g.\
\cite[p.~16]{Krupa2026NonUniqueness}.

\subsection{Increasing concave and convex functions with prescribed integral}
\label{sec:increasing-concave-convex-prescribed-integral}

We characterize the endpoint and integral data that can be realized by a
strictly increasing function having a prescribed strict curvature sign.  The
proof also gives an explicit smooth family realizing every admissible set of
data.

\begin{figure}[t]
    \centering
   \includegraphics[width=.9\textwidth]{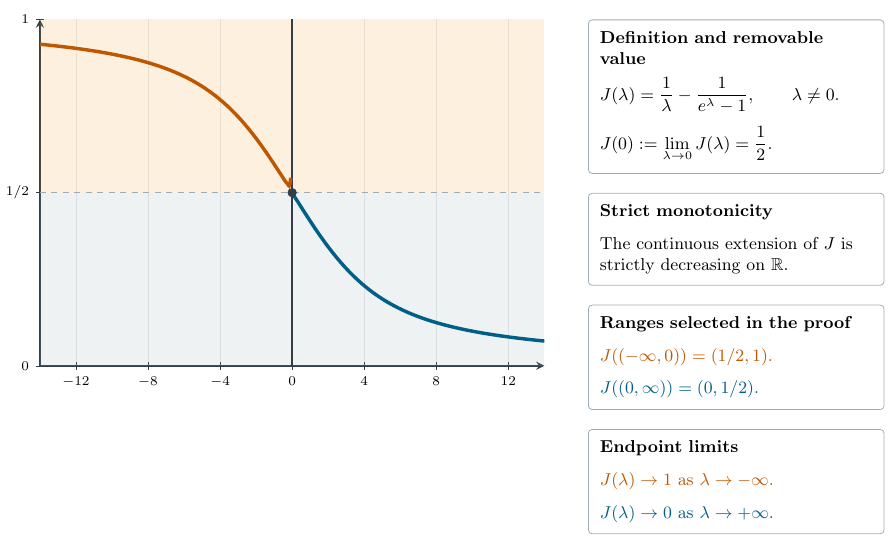}
    \caption{The continuous extension of \(J\) is strictly decreasing,
    maps \(( -\infty,0 )\) onto \((1/2,1)\), and maps
    \((0,\infty)\) onto \((0,1/2)\).}
    \label{fig:auxiliary-function-J}
\end{figure}

\begin{lemma}[Properties of the auxiliary function]\label{lem:properties-of-J}
For \(\lambda\in\mathbb{R}\setminus\{0\}\), define
\begin{equation}\label{eq:def-J}
    J(\lambda)
    :=
    \frac{1}{\lambda}-\frac{1}{e^{\lambda}-1}.
\end{equation}
Then the singularity at \(\lambda=0\) is removable, with
\begin{equation}\label{eq:J-at-zero}
    \lim_{\lambda\to 0}J(\lambda)=\frac12.
\end{equation}
After setting \(J(0):=1/2\), the function \(J\) is continuous and strictly
decreasing on \(\mathbb{R}\).  Moreover,
\begin{equation}\label{eq:J-endpoint-limits}
    \lim_{\lambda\to-\infty}J(\lambda)=1,
    \qquad
    \lim_{\lambda\to+\infty}J(\lambda)=0.
\end{equation}
Consequently,
\begin{equation}\label{eq:J-ranges}
    J\bigl(( -\infty,0)\bigr)=\left(\frac12,1\right),
    \qquad
    J\bigl((0,+\infty)\bigr)=\left(0,\frac12\right).
\end{equation}
\begin{remark}
See \Cref{fig:auxiliary-function-J}.
\end{remark}
\end{lemma}

\begin{proof}
First rewrite \eqref{eq:def-J} as
\begin{equation}\label{eq:J-common-denominator}
    J(\lambda)
    =
    \frac{e^{\lambda}-1-\lambda}
         {\lambda(e^{\lambda}-1)}.
\end{equation}
The Taylor expansion
\[
    e^{\lambda}
    =1+\lambda+\frac{\lambda^{2}}{2}+O(\lambda^{3})
    \qquad (\lambda\to0)
\]
gives
\[
    e^{\lambda}-1-\lambda
    =\frac{\lambda^{2}}{2}+O(\lambda^{3}),
    \qquad
    \lambda(e^{\lambda}-1)
    =\lambda^{2}+O(\lambda^{3}).
\]
Substitution into \eqref{eq:J-common-denominator} yields
\[
    \lim_{\lambda\to0}J(\lambda)=\frac12.
\]
Thus defining \(J(0)=1/2\) gives a continuous extension at the origin.

For \(\lambda\neq0\), differentiation gives
\begin{align}
    J'(\lambda)
    &=-\frac{1}{\lambda^{2}}
      +\frac{e^{\lambda}}{(e^{\lambda}-1)^{2}} \notag\\
    &=-\frac{1}{\lambda^{2}}
      +\frac{1}{\bigl(e^{\lambda/2}-e^{-\lambda/2}\bigr)^{2}} \notag\\
    &=-\frac{1}{\lambda^{2}}
      +\frac{1}{4\sinh^{2}(\lambda/2)}.
    \label{eq:J-prime}
\end{align}
We claim that
\begin{equation}\label{eq:sinh-strict-bound}
    \sinh t>t
    \qquad\text{for every }t>0.
\end{equation}
Indeed, for \(s>0\),
\begin{equation}\label{eq:cosh-greater-than-one}
    \cosh s-1
    =\frac{e^{s}+e^{-s}-2}{2}
    =\frac{(e^{s}-1)^{2}}{2e^{s}}
    >0.
\end{equation}
Therefore
\[
    \sinh t-t
    =\int_{0}^{t}(\cosh s-1)\,ds
    >0,
\]
which proves \eqref{eq:sinh-strict-bound}.  Since \(\sinh\) is odd, for every
\(\lambda\neq0\) we obtain
\[
    2\left|\sinh\left(\frac{\lambda}{2}\right)\right|
    =2\sinh\left(\frac{|\lambda|}{2}\right)
    >|\lambda|.
\]
Hence
\[
    4\sinh^{2}\left(\frac{\lambda}{2}\right)>\lambda^{2},
\]
and \eqref{eq:J-prime} implies
\[
    J'(\lambda)<0
    \qquad\text{for every }\lambda\neq0.
\]
Together with continuity at \(0\), this shows that the extended function is
strictly decreasing on all of \(\mathbb{R}\).

Finally, the limits in \eqref{eq:J-endpoint-limits} follow directly from
\eqref{eq:def-J}: as \(\lambda\to+\infty\), both terms tend to zero, whereas
as \(\lambda\to-\infty\), one has \(1/\lambda\to0\) and
\(-1/(e^{\lambda}-1)\to1\).  The range statements
\eqref{eq:J-ranges} now follow from continuity and strict monotonicity.
\end{proof}

\begin{theorem}[Prescribed endpoints and prescribed integral]
\label{thm:increasing-concave-convex-prescribed-integral}
Let \(a,b,c,d\in\mathbb{R}\), with \(c>0\).

\begin{enumerate}
    \item There exists a strictly increasing and strictly concave function
    \(f\colon[0,c]\to\mathbb{R}\) satisfying
    \begin{equation}\label{eq:prescribed-data}
        f(0)=a,
        \qquad
        f(c)=b,
        \qquad
        \int_{0}^{c}f(x)\,dx=d,
    \end{equation}
    if and only if
    \begin{equation}\label{eq:concave-conditions}
        b>a,
        \qquad
        \frac{c(a+b)}{2}<d<cb.
    \end{equation}

    \item There exists a strictly increasing and strictly convex function
    \(f\colon[0,c]\to\mathbb{R}\) satisfying \eqref{eq:prescribed-data} if
    and only if
    \begin{equation}\label{eq:convex-conditions}
        b>a,
        \qquad
        ca<d<\frac{c(a+b)}{2}.
    \end{equation}
\end{enumerate}

In both cases, the function may be chosen as the restriction of a function in
\(C^{\infty}(\mathbb{R})\).  More precisely, in the concave case it may be
chosen so that \(f'>0\) and \(f''<0\) on \(\mathbb{R}\), and in the convex
case it may be chosen so that \(f'>0\) and \(f''>0\) on \(\mathbb{R}\).
\end{theorem}

\begin{proof}
We first prove necessity.  If \(f\) is strictly increasing, then
\[
    a=f(0)<f(c)=b,
\]
so necessarily \(b>a\).

\begin{figure}[t]
    \centering
    \resizebox{\linewidth}{!}{

\definecolor{UTBurntOrange}{HTML}{BF5700}
\definecolor{UTCharcoal}{HTML}{333F48}
\definecolor{UTLightOrange}{HTML}{F8971F}
\definecolor{UTBeige}{HTML}{D6D2C4}
\definecolor{UTBlueGray}{HTML}{9CADB7}
\definecolor{UTDeepBlue}{HTML}{005F86}

\begin{tikzpicture}[
    every node/.style={font=\small},
    panel/.style={
        anchor=north west,
        draw=UTBlueGray,
        rounded corners=2pt,
        fill=white,
        text width=5.0cm,
        align=left,
        inner sep=6pt
    }
]
\begin{axis}[
    name=mainplot,
    at={(0,0)},
    anchor=south west,
    width=9.25cm,
    height=6.35cm,
    scale only axis,
    xmin=0, xmax=1,
    ymin=0.12, ymax=0.93,
    axis lines=left,
    axis line style={UTCharcoal, line width=0.9pt},
    xtick={0,1},
    xticklabels={$0$,$c$},
    ytick={0.20,0.86},
    yticklabels={$a$,$b$},
    tick label style={font=\small},
    tick style={UTCharcoal},
    clip=true,
    samples=260,
    domain=0:1,
]
    \def\aval{0.20}
    \def\bval{0.86}
    \def\lam{-2.4}

    \addplot[name path=chord, draw=none]
        {\aval+(\bval-\aval)*x};
    \addplot[name path=curve, draw=none]
        {\aval+(\bval-\aval)*(exp(\lam*x)-1)/(exp(\lam)-1)};
    \addplot[name path=topline, draw=none] {\bval};

    \addplot[UTLightOrange, opacity=0.36]
        fill between[of=curve and chord];
    \addplot[UTBeige, opacity=0.68]
        fill between[of=topline and curve];

    \addplot[UTBlueGray, dashed, line width=0.75pt]
        coordinates {(0,\bval) (1,\bval)};
    \addplot[UTBlueGray, dashed, line width=0.75pt]
        coordinates {(1,0.12) (1,\bval)};
    \addplot[UTCharcoal, dashed, line width=1.15pt]
        {\aval+(\bval-\aval)*x};
    \addplot[UTBurntOrange, line width=1.8pt]
        {\aval+(\bval-\aval)*(exp(\lam*x)-1)/(exp(\lam)-1)};
    \addplot[only marks, mark=*, mark size=2.1pt, UTBurntOrange]
        coordinates {(0,\aval) (1,\bval)};
\end{axis}

\node[panel] (p1) at ([xshift=8mm]mainplot.north east) {
    \textbf{Chord and curve}\par\medskip
    \(\displaystyle L(x)=a+\frac{b-a}{c}x.\)\par\medskip
    \textcolor{UTBurntOrange}{Solid dark-orange curve: \(f\).}\par
    \textcolor{UTCharcoal}{Dashed dark gray segment: \(L\).}
};

\node[panel] (p2) at ([yshift=-3.5mm]p1.south west) {
    \textbf{Pointwise inequalities}\par\medskip
    \textcolor{UTBurntOrange}{\(f(x)>L(x)\) for \(0<x<c\).}\par\medskip
    \textcolor{UTDeepBlue}{\(f(x)<b\) for \(0\leq x<c\).}
};

\node[panel] (p3) at ([yshift=-3.5mm]p2.south west) {
    \textbf{Positive integral gaps}\par\medskip
    \(\displaystyle
      \int_0^c\!\bigl(f-L\bigr)\,dx
      =d-\frac{c(a+b)}{2}>0.
    \)\par\medskip
    \(\displaystyle
      \int_0^c\!\bigl(b-f\bigr)\,dx
      =cb-d>0.
    \)
};

\node[
    anchor=north west,
    draw=UTBurntOrange,
    line width=1pt,
    rounded corners=2pt,
    fill=white,
    text width=5.0cm,
    align=center,
    inner sep=7pt,
    font=\normalsize
] at ([yshift=-3.5mm]p3.south west) {
    \(\displaystyle \frac{c(a+b)}{2}<d<cb.\)
};
\end{tikzpicture}}
    \caption{The two positive integral gaps in the strictly concave case:
    \(d-c(a+b)/2=\int_0^c(f-L)\,dx>0\) and
    \(cb-d=\int_0^c(b-f)\,dx>0\).}
    \label{fig:concave-integral-bounds}
\end{figure}
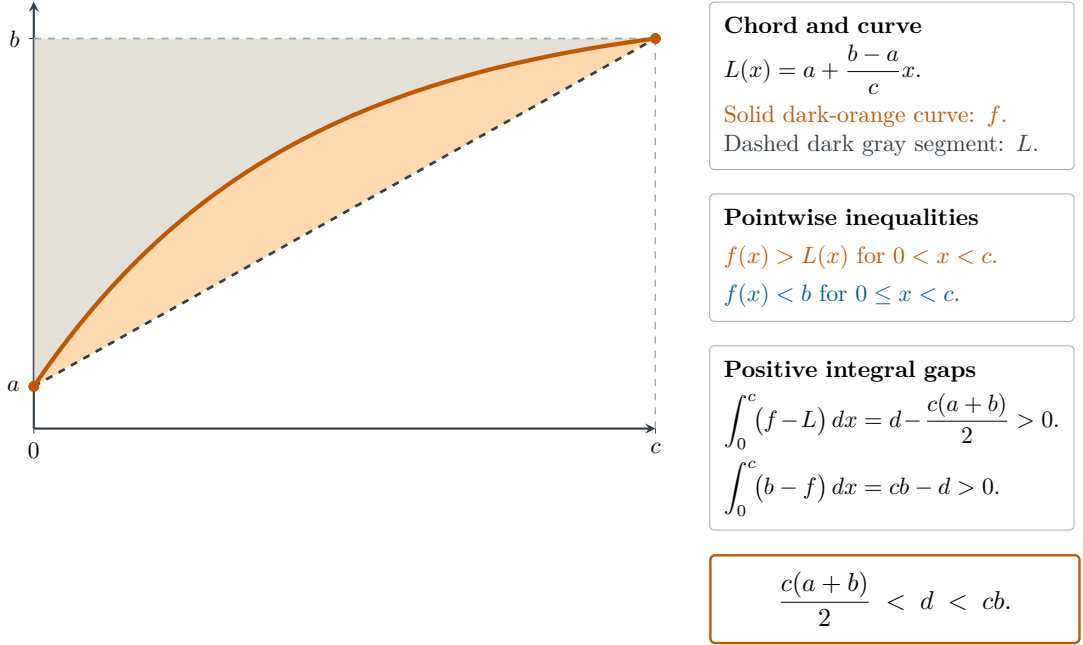

Suppose first that \(f\) is strictly concave.  For every \(x\in(0,c)\),
strict concavity applied to the endpoints \(0\) and \(c\) gives
\begin{equation}\label{eq:above-chord}
    f(x)
    >
    \left(1-\frac{x}{c}\right)f(0)+\frac{x}{c}f(c)
    =a+\frac{b-a}{c}x.
\end{equation}
A finite concave function is continuous in the interior of its domain; hence
the strict inequality in \eqref{eq:above-chord} is strict on a nontrivial
neighborhood of, for example, \(x=c/2\).  Integrating the chord inequality
therefore yields
\begin{equation}\label{eq:concave-lower-integral}
    d
    =\int_{0}^{c}f(x)\,dx
    >\int_{0}^{c}\left(a+\frac{b-a}{c}x\right)\,dx
    =\frac{c(a+b)}{2}.
\end{equation}

On the other hand, strict monotonicity gives \(f(c/2)<b\).  Therefore
\begin{align*}
    d
    &=\int_{0}^{c/2}f(x)\,dx+
      \int_{c/2}^{c}f(x)\,dx \\
    &\leq \frac{c}{2}f(c/2)+\frac{c}{2}b
    <cb.
\end{align*}
This proves \eqref{eq:concave-conditions}. We refer to \Cref{fig:concave-integral-bounds}.

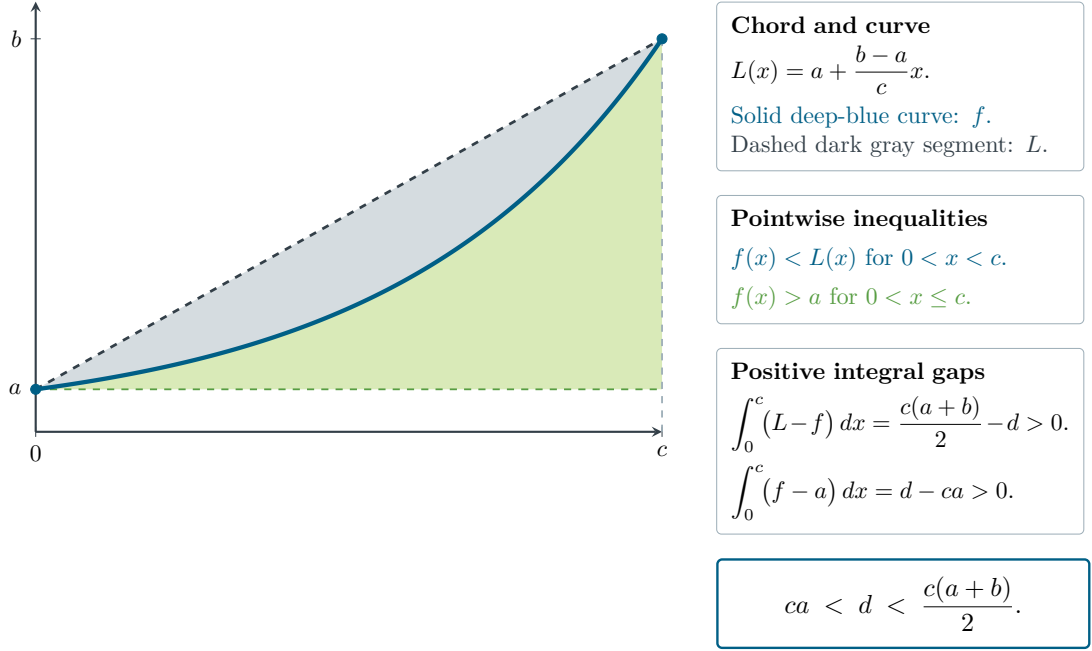
\begin{figure}[t]
    \centering
    \resizebox{\linewidth}{!}{

\definecolor{UTCharcoal}{HTML}{333F48}
\definecolor{UTBlueGray}{HTML}{9CADB7}
\definecolor{UTDeepBlue}{HTML}{005F86}
\definecolor{UTGreen}{HTML}{579D42}
\definecolor{UTLightGreen}{HTML}{A6CD57}

\begin{tikzpicture}[
    every node/.style={font=\small},
    panel/.style={
        anchor=north west,
        draw=UTBlueGray,
        rounded corners=2pt,
        fill=white,
        text width=5.0cm,
        align=left,
        inner sep=6pt
    }
]
\begin{axis}[
    name=mainplot,
    at={(0,0)},
    anchor=south west,
    width=9.25cm,
    height=6.35cm,
    scale only axis,
    xmin=0, xmax=1,
    ymin=0.12, ymax=0.93,
    axis lines=left,
    axis line style={UTCharcoal, line width=0.9pt},
    xtick={0,1},
    xticklabels={$0$,$c$},
    ytick={0.20,0.86},
    yticklabels={$a$,$b$},
    tick label style={font=\small},
    tick style={UTCharcoal},
    clip=true,
    samples=260,
    domain=0:1,
]
    \def\aval{0.20}
    \def\bval{0.86}
    \def\lam{2.4}

    \addplot[name path=chord, draw=none]
        {\aval+(\bval-\aval)*x};
    \addplot[name path=curve, draw=none]
        {\aval+(\bval-\aval)*(exp(\lam*x)-1)/(exp(\lam)-1)};
    \addplot[name path=baseline, draw=none] {\aval};

    \addplot[UTBlueGray, opacity=0.42]
        fill between[of=chord and curve];
    \addplot[UTLightGreen, opacity=0.42]
        fill between[of=curve and baseline];

    \addplot[UTGreen, dashed, line width=0.75pt]
        coordinates {(0,\aval) (1,\aval)};
    \addplot[UTBlueGray, dashed, line width=0.75pt]
        coordinates {(1,0.12) (1,\bval)};
    \addplot[UTCharcoal, dashed, line width=1.15pt]
        {\aval+(\bval-\aval)*x};
    \addplot[UTDeepBlue, line width=1.8pt]
        {\aval+(\bval-\aval)*(exp(\lam*x)-1)/(exp(\lam)-1)};
    \addplot[only marks, mark=*, mark size=2.1pt, UTDeepBlue]
        coordinates {(0,\aval) (1,\bval)};
\end{axis}

\node[panel] (p1) at ([xshift=8mm]mainplot.north east) {
    \textbf{Chord and curve}\par\medskip
    \(\displaystyle L(x)=a+\frac{b-a}{c}x.\)\par\medskip
    \textcolor{UTDeepBlue}{Solid deep-blue curve: \(f\).}\par
    \textcolor{UTCharcoal}{Dashed dark gray segment: \(L\).}
};

\node[panel] (p2) at ([yshift=-3.5mm]p1.south west) {
    \textbf{Pointwise inequalities}\par\medskip
    \textcolor{UTDeepBlue}{\(f(x)<L(x)\) for \(0<x<c\).}\par\medskip
    \textcolor{UTGreen}{\(f(x)>a\) for \(0<x\leq c\).}
};

\node[panel] (p3) at ([yshift=-3.5mm]p2.south west) {
    \textbf{Positive integral gaps}\par\medskip
    \(\displaystyle
      \int_0^c\!\bigl(L-f\bigr)\,dx
      =\frac{c(a+b)}{2}-d>0.
    \)\par\medskip
    \(\displaystyle
      \int_0^c\!\bigl(f-a\bigr)\,dx
      =d-ca>0.
    \)
};

\node[
    anchor=north west,
    draw=UTDeepBlue,
    line width=1pt,
    rounded corners=2pt,
    fill=white,
    text width=5.0cm,
    align=center,
    inner sep=7pt,
    font=\normalsize
] at ([yshift=-3.5mm]p3.south west) {
    \(\displaystyle ca<d<\frac{c(a+b)}{2}.\)
};
\end{tikzpicture}}
    \caption{The two positive integral gaps in the strictly convex case:
    \(d-ca=\int_0^c(f-a)\,dx>0\) and
    \(c(a+b)/2-d=\int_0^c(L-f)\,dx>0\).}
    \label{fig:convex-integral-bounds}
\end{figure}

Suppose instead that \(f\) is strictly convex.  The graph lies strictly below
its endpoint chord, so
\[
    f(x)<a+\frac{b-a}{c}x,
    \qquad 0<x<c.
\]
As above, integration gives
\begin{equation}\label{eq:convex-upper-integral}
    d<\frac{c(a+b)}{2}.
\end{equation}
Moreover, strict monotonicity gives \(f(c/2)>a\), and hence
\begin{align*}
    d
    &=\int_{0}^{c/2}f(x)\,dx+
      \int_{c/2}^{c}f(x)\,dx \\
    &\geq \frac{c}{2}a+\frac{c}{2}f(c/2)
    >ca.
\end{align*}
This proves \eqref{eq:convex-conditions}. We refer to \Cref{fig:convex-integral-bounds}.

We now prove sufficiency by an explicit construction.  Assume \(b>a\), set
\begin{equation}\label{eq:B-and-theta}
    B:=b-a>0,
    \qquad
    \theta:=\frac{d-ca}{cB},
\end{equation}
and, for \(\lambda\neq0\), define
\begin{equation}\label{eq:exponential-family}
    f_{\lambda}(x)
    :=a+B\frac{e^{\lambda x/c}-1}{e^{\lambda}-1},
    \qquad x\in\mathbb{R}.
\end{equation}
This function satisfies
\[
    f_{\lambda}(0)=a,
    \qquad
    f_{\lambda}(c)=a+B=b.
\]
Its first two derivatives are
\begin{align}
    f_{\lambda}'(x)
    &=\frac{B\lambda}{c(e^{\lambda}-1)}e^{\lambda x/c},
    \label{eq:f-lambda-first-derivative}\\
    f_{\lambda}''(x)
    &=\frac{B\lambda^{2}}{c^{2}(e^{\lambda}-1)}e^{\lambda x/c}.
    \label{eq:f-lambda-second-derivative}
\end{align}
Since \(\lambda\) and \(e^{\lambda}-1\) have the same sign,
\eqref{eq:f-lambda-first-derivative} shows that
\[
    f_{\lambda}'(x)>0
    \qquad\text{for all }x\in\mathbb{R}.
\]
Furthermore, \eqref{eq:f-lambda-second-derivative} shows that
\[
    \lambda<0 \implies f_{\lambda}''<0,
    \qquad
    \lambda>0 \implies f_{\lambda}''>0.
\]
Thus negative values of \(\lambda\) yield strictly concave functions, whereas
positive values yield strictly convex functions. We refer to \Cref{fig:exponential-realizing-family}.

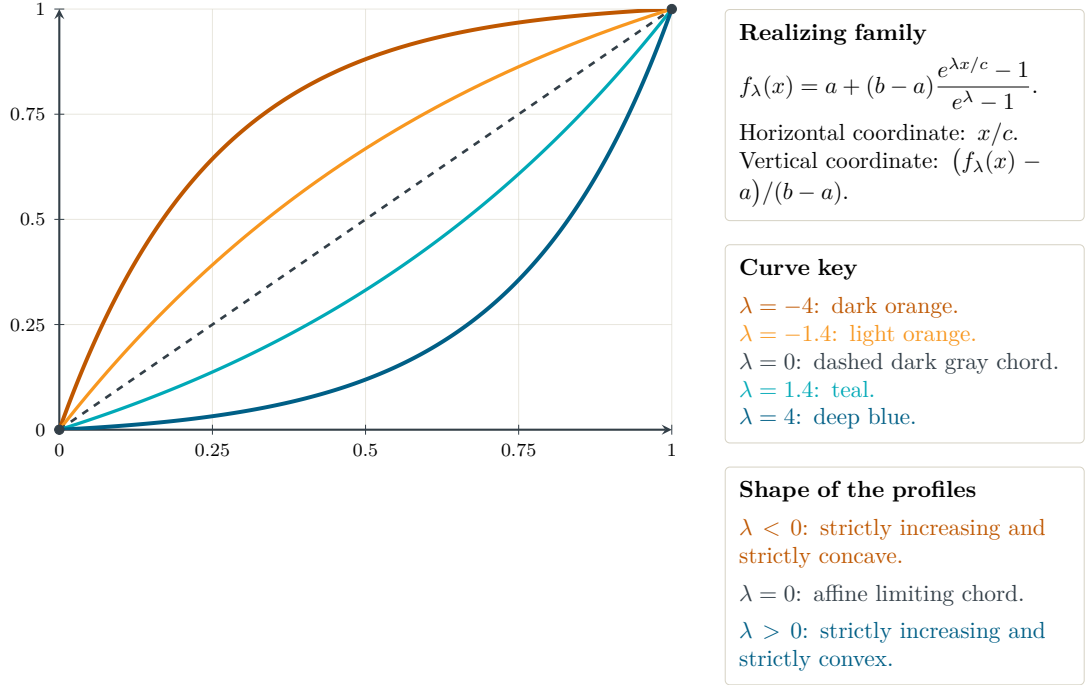
\begin{figure}[t]
    \centering
    \resizebox{\linewidth}{!}{

\definecolor{UTBurntOrange}{HTML}{BF5700}
\definecolor{UTCharcoal}{HTML}{333F48}
\definecolor{UTLightOrange}{HTML}{F8971F}
\definecolor{UTBeige}{HTML}{D6D2C4}
\definecolor{UTDeepBlue}{HTML}{005F86}
\definecolor{UTTeal}{HTML}{00A9B7}

\begin{tikzpicture}[
    every node/.style={font=\small},
    panel/.style={
        anchor=north west,
        draw=UTBeige,
        rounded corners=2pt,
        fill=white,
        text width=5.0cm,
        align=left,
        inner sep=6pt
    }
]
\begin{axis}[
    name=mainplot,
    at={(0,0)},
    anchor=south west,
    width=9.25cm,
    height=6.35cm,
    scale only axis,
    xmin=0, xmax=1,
    ymin=0, ymax=1,
    axis lines=left,
    axis line style={UTCharcoal, line width=0.9pt},
    xtick={0,0.25,0.5,0.75,1},
    ytick={0,0.25,0.5,0.75,1},
    tick label style={font=\scriptsize},
    tick style={UTCharcoal},
    grid=both,
    grid style={UTBeige, opacity=0.55},
    clip=true,
    samples=280,
    domain=0:1,
]
    \addplot[UTBurntOrange, line width=1.7pt]
        {(exp(-4*x)-1)/(exp(-4)-1)};
    \addplot[UTLightOrange, line width=1.4pt]
        {(exp(-1.4*x)-1)/(exp(-1.4)-1)};
    \addplot[UTCharcoal, dashed, line width=1.15pt] {x};
    \addplot[UTTeal, line width=1.4pt]
        {(exp(1.4*x)-1)/(exp(1.4)-1)};
    \addplot[UTDeepBlue, line width=1.7pt]
        {(exp(4*x)-1)/(exp(4)-1)};
    \addplot[only marks, mark=*, mark size=2.0pt, UTCharcoal]
        coordinates {(0,0) (1,1)};
\end{axis}

\node[panel] (p1) at ([xshift=8mm]mainplot.north east) {
    \textbf{Realizing family}\par\medskip
    \(\displaystyle
      f_\lambda(x)=a+(b-a)
      \frac{e^{\lambda x/c}-1}{e^\lambda-1}.
    \)\par\medskip
    Horizontal coordinate: \(x/c\).\par
    Vertical coordinate: \(\bigl(f_\lambda(x)-a\bigr)/(b-a)\).
};

\node[panel] (p2) at ([yshift=-3.5mm]p1.south west) {
    \textbf{Curve key}\par\medskip
    \textcolor{UTBurntOrange}{\(\lambda=-4\): dark orange.}\par
    \textcolor{UTLightOrange}{\(\lambda=-1.4\): light orange.}\par
    \textcolor{UTCharcoal}{\(\lambda=0\): dashed dark gray chord.}\par
    \textcolor{UTTeal}{\(\lambda=1.4\): teal.}\par
    \textcolor{UTDeepBlue}{\(\lambda=4\): deep blue.}
};

\node[panel] (p3) at ([yshift=-3.5mm]p2.south west) {
    \textbf{Shape of the profiles}\par\medskip
    \textcolor{UTBurntOrange}{\(\lambda<0\): strictly increasing and strictly concave.}\par\medskip
    \textcolor{UTCharcoal}{\(\lambda=0\): affine limiting chord.}\par\medskip
    \textcolor{UTDeepBlue}{\(\lambda>0\): strictly increasing and strictly convex.}
};
\end{tikzpicture}}
    \caption{The normalized exponential family used in the sufficiency
    argument. Negative values of \(\lambda\) give increasing concave
    profiles, positive values give increasing convex profiles, and
    \(\lambda=0\) is the affine limiting chord.}
    \label{fig:exponential-realizing-family}
\end{figure}

A direct integration of \eqref{eq:exponential-family} gives
\begin{align}
    \int_{0}^{c}f_{\lambda}(x)\,dx
    &=ca+B\int_{0}^{c}
      \frac{e^{\lambda x/c}-1}{e^{\lambda}-1}\,dx \notag\\
    &=ca+cB\left(\frac{1}{\lambda}
      -\frac{1}{e^{\lambda}-1}\right) \notag\\
    &=ca+cB J(\lambda).
    \label{eq:integral-f-lambda}
\end{align}

In the concave case, \eqref{eq:concave-conditions} is equivalent to
\[
    \frac12<\theta<1.
\]
By Lemma~\ref{lem:properties-of-J}, there is a unique \(\lambda<0\) such that
\(J(\lambda)=\theta\).  Equation \eqref{eq:integral-f-lambda} then gives
\[
    \int_{0}^{c}f_{\lambda}(x)\,dx
    =ca+cB\theta=d,
\]
and \(f_{\lambda}\) is strictly increasing and strictly concave.

In the convex case, \eqref{eq:convex-conditions} is equivalent to
\[
    0<\theta<\frac12.
\]
Again by Lemma~\ref{lem:properties-of-J}, there is a unique \(\lambda>0\) such
that \(J(\lambda)=\theta\).  The same integral identity gives the prescribed
value \(d\), and \(f_{\lambda}\) is strictly increasing and strictly convex.
\end{proof}

\begin{remark}[Why all inequalities are strict]
The value \(d=c(a+b)/2\) is the integral of the affine chord joining
\((0,a)\) to \((c,b)\), and therefore corresponds to the zero-curvature
limit \(\lambda\to0\).  The endpoint values \(d=ca\) and \(d=cb\) arise only
as the limiting values \(\lambda\to+\infty\) and
\(\lambda\to-\infty\), respectively.  None of these boundary values can be
attained under the required strict monotonicity and strict curvature
conditions.
\end{remark}

\section{Geometric results on the Hugoniot locus}\label{Hugoniot_facts}

\subsection{Hugoniot shock curves and relative entropy for the $p$-system}
\label{sec:p-system-hugoniot}

Let $I\subset (0,\infty)$ be an open interval, and let
$p\in C^3(I)$ satisfy
\begin{equation}
    p'(v)<0,
    \qquad
    p''(v)>0,
    \qquad v\in I.
    \label{eq:p-system-hugoniot-assumptions}
\end{equation}
We consider the one-dimensional $p$-system in Lagrangian coordinates \eqref{eq:p-system}.

\subsubsection{Rankine--Hugoniot relations and shock curves}

Fix a state
\begin{equation}
    U=(v_U,u_U)^{\mathsf T}\in I\times\mathbb R.
\end{equation}
We regard $U$ as the left state.  A right state
$V=(v,u)^{\mathsf T}$ is connected to $U$ by a discontinuity of speed $\sigma$ if
and only if the Rankine--Hugoniot relation
\begin{equation}
    \sigma(V-U)=F(V)-F(U)
\end{equation}
holds.  In components, this is
\begin{equation}
    \sigma(v-v_U)=-(u-u_U),
    \qquad
    \sigma(u-u_U)=p(v)-p(v_U).
    \label{eq:p-system-rankine-hugoniot-components}
\end{equation}
Eliminating $\sigma$ gives the Hugoniot-locus equation
\begin{equation}
    (u-u_U)^2
    =-\bigl(p(v)-p(v_U)\bigr)(v-v_U).
    \label{eq:p-system-hugoniot-locus}
\end{equation}
For $v\neq v_U$, introduce the secant sound speed
\begin{equation}
    \widehat c_U(v)
    :=\sqrt{-\frac{p(v)-p(v_U)}{v-v_U}}.
    \label{eq:p-system-secant-sound-speed}
\end{equation}
Because $p'<0$, the quantity under the square root is positive.  We
extend \eqref{eq:p-system-secant-sound-speed} continuously to $v=v_U$
by setting
\begin{equation}
    \widehat c_U(v_U):=c(v_U)=\sqrt{-p'(v_U)}.
\end{equation}

For the convention in which $U$ is the left state, the Lax-admissible
$k$-shock curve issuing from $U$ is denoted by $S_U^k$, and it is
parameterized by the specific-volume coordinate $v$:
\begin{align}
    S_U^1(v)
    &:=
    \begin{pmatrix}
        v\\[1mm]
        u_U+(v-v_U)\widehat c_U(v)
    \end{pmatrix},
    & v&<v_U,
    \label{eq:p-system-1-shock-curve}\\[2mm]
    S_U^2(v)
    &:=
    \begin{pmatrix}
        v\\[1mm]
        u_U-(v-v_U)\widehat c_U(v)
    \end{pmatrix},
    & v&>v_U.
    \label{eq:p-system-2-shock-curve}
\end{align}
Both curves are extended to their base point by
\begin{equation}
    S_U^1(v_U)=S_U^2(v_U)=U.
\end{equation}
The corresponding shock speeds are denoted by $\sigma_U^k(v)$ and are
\begin{align}
    \sigma_U^1(v)
    &:=-\widehat c_U(v)
    =-\sqrt{-\frac{p(v)-p(v_U)}{v-v_U}},
    && v<v_U,
    \label{eq:p-system-1-shock-speed}\\[2mm]
    \sigma_U^2(v)
    &:=\widehat c_U(v)
    =\sqrt{-\frac{p(v)-p(v_U)}{v-v_U}},
    && v>v_U.
    \label{eq:p-system-2-shock-speed}
\end{align}
Equivalently, for either family,
\begin{equation}
    \bigl(\sigma_U^k(v)\bigr)^2
    =-\frac{p(v)-p(v_U)}{v-v_U},
    \qquad
    \sigma_U^k(v)
    =-\frac{\bigl[S_U^k(v)\bigr]_2-u_U}{v-v_U},
    \label{eq:p-system-shock-speed-rh-formulas}
\end{equation}
where $\bigl[S_U^k(v)\bigr]_2$ denotes the second component of $S_U^k(v)$.
The continuous extensions at the base state satisfy
\begin{equation}
    \sigma_U^1(v_U)=\lambda_1(U),
    \qquad
    \sigma_U^2(v_U)=\lambda_2(U).
\end{equation}
Moreover, the convexity assumption $p''>0$ gives the Lax inequalities
\begin{align}
    \lambda_1\bigl(S_U^1(v)\bigr)
    &<\sigma_U^1(v)<\lambda_1(U),
    && v<v_U,
    \label{eq:p-system-lax-1}\\
    \lambda_2\bigl(S_U^2(v)\bigr)
    &<\sigma_U^2(v)<\lambda_2(U),
    && v>v_U.
    \label{eq:p-system-lax-2}
\end{align}

\subsubsection{Derivatives of the $2$-shock speed}
\label{subsec:p-system-shock-speed-derivatives}

Fix $U=(v_U,u_U)^{\mathsf T}$ and let $v>v_U$.  For brevity, write
\begin{equation}
    h:=v-v_U,
    \qquad
    \sigma:=\sigma_U^2(v).
\end{equation}
By \eqref{eq:p-system-2-shock-speed},
\begin{equation}
    h\sigma^2=p(v_U)-p(v).
    \label{eq:p-system-speed-identity}
\end{equation}
Differentiating \eqref{eq:p-system-speed-identity} with respect to $v$
gives
\begin{equation}
    \sigma^2+2h\sigma\frac{d\sigma}{dv}=-p'(v).
    \label{eq:p-system-speed-first-differentiated-identity}
\end{equation}
Consequently,
\begin{equation}
    \frac{d}{dv}\sigma_U^2(v)
    =-\frac{p'(v)+\bigl(\sigma_U^2(v)\bigr)^2}
    {2\sigma_U^2(v)(v-v_U)}
    \label{eq:p-system-speed-first-derivative}
\end{equation}
and, equivalently,
\begin{equation}
    \boxed{
    \frac{d}{dv}\sigma_U^2(v)
    =\frac{p(v)-p(v_U)-p'(v)(v-v_U)}
    {2\sigma_U^2(v)(v-v_U)^2}
    }.
    \label{eq:p-system-speed-first-derivative-pressure-form}
\end{equation}
Since $p''>0$,
\begin{equation}
    p'(v)+\bigl(\sigma_U^2(v)\bigr)^2
    =\frac{1}{v-v_U}
    \int_{v_U}^{v}\bigl(p'(v)-p'(s)\bigr)\,ds>0.
\end{equation}
It follows that
\begin{equation}
    \frac{d}{dv}\sigma_U^2(v)<0,
    \qquad v>v_U.
    \label{eq:p-system-speed-monotonicity}
\end{equation}
Thus the $2$-shock speed is strictly decreasing along the $2$-shock
curve when the curve is parameterized by the right-state specific
volume.

\subsubsection{2-shocks strengthen in the relative entropy ``norm''}
\label{subsec:p-system-relative-entropy}

For states $A=(v_A,u_A)^{\mathsf T}$ and $B=(v_B,u_B)^{\mathsf T}$, the relative entropy is
\begin{equation}
    \eta(A\mid B)
    :=\eta(A)-\eta(B)-\nabla\eta(B)\cdot(A-B).
    \label{eq:p-system-relative-entropy-definition}
\end{equation}
For the entropy \eqref{eq:p-system-natural-entropy}, this becomes
\begin{equation}
    \eta(A\mid B)
    =\frac{1}{2}(u_A-u_B)^2
    +\int_{v_A}^{v_B}\bigl(p(s)-p(v_B)\bigr)\,ds.
    \label{eq:p-system-relative-entropy-explicit}
\end{equation}

Now fix a left state $U_L=(v_L,u_L)^{\mathsf T}$ and consider the $2$-shock curve
$S_{U_L}^2(v)$ for $v>v_L$.  Define
\begin{equation}
    \mathcal E_L(v)
    :=\eta\bigl(U_L\mid S_{U_L}^2(v)\bigr).
    \label{eq:p-system-relative-entropy-along-shock-definition}
\end{equation}
Since
\begin{equation}
    S_{U_L}^2(v)
    =\begin{pmatrix}
        v\\[1mm]
        u_L-(v-v_L)\sigma_{U_L}^2(v)
    \end{pmatrix}
\end{equation}
and
\begin{equation}
    (v-v_L)\bigl(\sigma_{U_L}^2(v)\bigr)^2
    =p(v_L)-p(v),
\end{equation}
formula \eqref{eq:p-system-relative-entropy-explicit} gives
\begin{align}
    \mathcal E_L(v)
    &=
    \frac{1}{2}(v-v_L)^2\bigl(\sigma_{U_L}^2(v)\bigr)^2
    +\int_{v_L}^{v}\bigl(p(s)-p(v)\bigr)\,ds
    \notag\\
    &=\frac{1}{2}(v-v_L)\bigl(p(v_L)-p(v)\bigr)
    +\int_{v_L}^{v}\bigl(p(s)-p(v)\bigr)\,ds.
    \label{eq:p-system-relative-entropy-along-shock}
\end{align}
Differentiating with respect to $v$ yields
\begin{equation}
    \frac{d}{dv}
    \eta\bigl(U_L\mid S_{U_L}^2(v)\bigr)
    =\frac{1}{2}\left[
        p(v_L)-p(v)-3(v-v_L)p'(v)
    \right].
    \label{eq:p-system-relative-entropy-derivative}
\end{equation}
Using
\begin{equation}
    p(v_L)-p(v)
    =(v-v_L)\bigl(\sigma_{U_L}^2(v)\bigr)^2,
\end{equation}
we obtain the equivalent form
\begin{equation}
    \frac{d}{dv}
    \eta\bigl(U_L\mid S_{U_L}^2(v)\bigr)
    =\frac{v-v_L}{2}
    \left[
        \bigl(\sigma_{U_L}^2(v)\bigr)^2-3p'(v)
    \right].
    \label{eq:p-system-relative-entropy-derivative-speed-form}
\end{equation}
Because $v>v_L$, $\bigl(\sigma_{U_L}^2(v)\bigr)^2>0$, and $p'(v)<0$,
we conclude that
\begin{equation}
    \frac{d}{dv}
    \eta\bigl(U_L\mid S_{U_L}^2(v)\bigr)>0,
    \qquad v>v_L.
    \label{eq:p-system-relative-entropy-monotonicity}
\end{equation}
At the base point, the continuous extension satisfies
\begin{equation}
    \left.
    \frac{d}{dv}
    \eta\bigl(U_L\mid S_{U_L}^2(v)\bigr)
    \right|_{v=v_L^+}=0.
\end{equation}

\subsubsection{Monotonicity of the $2$-shock curve}\label{sec:mono_2shock}
\begin{lemma}[Monotonicity of the $2$-shock curve]
\label{lem:p-system-2-shock-curve-monotonicity}
Let $U=(v_U,u_U)^{\mathsf T}$, and suppose that $p'(v)<0$. The $2$-shock
curve issuing from $U$ is given by
\begin{equation}
    S_U^2(v)
    =
    \begin{pmatrix}
        v\\[1mm]
        u_U-\sqrt{(v-v_U)\bigl(p(v_U)-p(v)\bigr)}
    \end{pmatrix},
    \qquad v>v_U.
    \label{eq:p-system-2-shock-curve-monotonicity-form}
\end{equation}
Equivalently,
\begin{equation}
    S_U^2(v)
    =
    \begin{pmatrix}
        v\\[1mm]
        u_U-(v-v_U)\sigma_U^2(v)
    \end{pmatrix},
    \qquad
    \sigma_U^2(v)
    =
    \sqrt{\frac{p(v_U)-p(v)}{v-v_U}}.
    \label{eq:p-system-2-shock-curve-speed-form}
\end{equation}
The velocity component of $S_U^2(v)$ is strictly decreasing as a
function of $v$. More precisely,
\begin{equation}
    \frac{d}{dv}S_U^2(v)
    =
    \begin{pmatrix}
        1\\[3mm]
        \displaystyle
        \frac{p'(v)-\bigl(\sigma_U^2(v)\bigr)^2}
        {2\sigma_U^2(v)}
    \end{pmatrix},
    \qquad v>v_U,
    \label{eq:p-system-2-shock-curve-derivative}
\end{equation}
and hence
\begin{equation}
    \frac{d}{dv}\bigl[S_U^2(v)\bigr]_2<0,
    \qquad v>v_U.
    \label{eq:p-system-2-shock-curve-decreasing}
\end{equation}
Moreover, the right derivative at the base state is
\begin{equation}
    \left.
    \frac{d}{dv}\bigl[S_U^2(v)\bigr]_2
    \right|_{v=v_U^+}
    =
    -\sqrt{-p'(v_U)}
    =
    -\lambda_2(U).
    \label{eq:p-system-2-shock-curve-base-derivative}
\end{equation}
\end{lemma}

\begin{proof}
Denote the velocity component of the shock curve by
\begin{equation}
    u_U^2(v)
    :=
    \bigl[S_U^2(v)\bigr]_2
    =
    u_U-\sqrt{(v-v_U)\bigl(p(v_U)-p(v)\bigr)}.
\end{equation}
Differentiating with respect to $v$ gives
\begin{align}
    \frac{d}{dv}u_U^2(v)
    &=
    -\frac{
        p(v_U)-p(v)-(v-v_U)p'(v)
    }{
        2\sqrt{(v-v_U)\bigl(p(v_U)-p(v)\bigr)}
    }.
    \label{eq:p-system-2-shock-curve-derivative-pressure-form}
\end{align}
Using
\begin{equation}
    p(v_U)-p(v)
    =
    (v-v_U)\bigl(\sigma_U^2(v)\bigr)^2
\end{equation}
in \eqref{eq:p-system-2-shock-curve-derivative-pressure-form}, we
obtain
\begin{equation}
    \frac{d}{dv}u_U^2(v)
    =
    \frac{
        p'(v)-\bigl(\sigma_U^2(v)\bigr)^2
    }{
        2\sigma_U^2(v)
    }.
    \label{eq:p-system-2-shock-curve-derivative-speed-form}
\end{equation}
Since $p'(v)<0$ and $\sigma_U^2(v)>0$, it follows immediately that
\begin{equation}
    p'(v)-\bigl(\sigma_U^2(v)\bigr)^2<0.
\end{equation}
Consequently,
\begin{equation}
    \frac{d}{dv}u_U^2(v)<0,
    \qquad v>v_U.
\end{equation}
Thus the velocity component of the $2$-shock curve is strictly
decreasing with respect to the specific-volume parameter.

Finally,
\begin{equation}
    \lim_{v\to v_U^+}\sigma_U^2(v)
    =
    \sqrt{-p'(v_U)}.
\end{equation}
Passing to the limit in
\eqref{eq:p-system-2-shock-curve-derivative-speed-form} yields
\begin{align}
    \lim_{v\to v_U^+}\frac{d}{dv}u_U^2(v)
    &=
    \frac{
        p'(v_U)-\bigl(-p'(v_U)\bigr)
    }{
        2\sqrt{-p'(v_U)}
    }\\
    &=
    \frac{p'(v_U)}{\sqrt{-p'(v_U)}}
    =
    -\sqrt{-p'(v_U)}
    =
    -\lambda_2(U).
\end{align}
\end{proof}
Thus the graph $u=u_U^2(v)$ defining the $2$-shock curve is strictly decreasing.
Notice that this conclusion only requires $p'<0$; the additional assumption
$p''>0$ is not needed for this particular monotonicity statement.

\subsection{Shock control at remote endpoints}
\begin{lemma}[Shock control at remote endpoints]
\label{lem:remote-endpoint-control}

Consider the $p$-system \eqref{eq:system}.

Let $p\in C^3((0,\infty))$ satisfy $p'<0$. Suppose that, for some
$d>0$, $p''>0$ on $[d,\infty)$ and
$p(v)\to p_\infty\in\mathbb R$ as $v\to\infty$.
Fix $0<\underline v<d$ and $\delta>0$, and put
\[
 \mu=\min_{[\underline v,d]}(-p')>0,
 \qquad D=p(\underline v)-p_\infty.
\]
Choose
\begin{equation}\label{eq:remote-size}
 b>d+\delta+D/\mu.
\end{equation}
For $y\in[b-\delta,b+\delta]$ and $\underline v\le v<y$,
\begin{equation}\label{eq:remote-secant}
 p'(v)<\frac{p(y)-p(v)}{y-v}<p'(y)<0.
\end{equation}
Every nontrivial Rankine--Hugoniot discontinuity whose two volume
coordinates are at least $\underline v$, with at least one in
$[b-\delta,b+\delta]$, satisfies the natural entropy inequality if and
only if $\sigma(v_+-v_-)>0$. Every such entropy discontinuity is strict
Lax: negative-speed discontinuities satisfy
\[
 v_+<v_-,\qquad \lambda_1(U_+)<\sigma<\lambda_1(U_-),
\]
and positive-speed discontinuities satisfy
\[
 v_+>v_-,\qquad \lambda_2(U_+)<\sigma<\lambda_2(U_-).
\]
The natural entropy production is strictly negative.
\end{lemma}

\subsection{Proof of \Cref{lem:remote-endpoint-control}}\label{sec:proof:lem:remote-endpoint-control}
\subsubsection{Why moving the endpoints far to the right gives the crucial secant inequality}

The pressure has a bounded strictly convex tail:
\[
p''>0\quad\text{on }[d,\infty),
\qquad
p(v)\to p_\infty\in\R.
\]
There is no convexity assumption on the core $[\underline v,d]$.

Set
\[
\mu:=\min_{[\underline v,d]}(-p')>0,
\qquad
D:=p(\underline v)-p_\infty.
\]
Choose an endpoint window $[b-\delta,b+\delta]$ with
\begin{equation}
 b>d+\delta+\frac{D}{\mu}.
\eqref{eq:remote-size}
\end{equation}

The role of each quantity is simple. The pressure can change by at most $D$ between any allowed core volume and infinity. But the distance from the core to a remote endpoint is very large. Therefore the average slope across that long interval is small compared with every $-p'$ inside the core.

\subsubsection{Fix the remote volume and define the positive secant slope}

Fix
\[
y\in[b-\delta,b+\delta].
\]
For $v<y$, define
\[
a_y(v):=\frac{p(v)-p(y)}{y-v}>0.
\]
This is the average value of $-p'$ on $[v,y]$:
\[
a_y(v)=\frac{1}{y-v}\int_v^y[-p'(s)]\,ds.
\]
We first prove
\begin{equation}
a_y(v)<-p'(v).
\label{eq:remote-left-slope-bound}
\end{equation}

\subsubsection{When $v$ lies in the core}

For $\underline v\le v\le d$,
\[
p(v)-p(y)\le p(\underline v)-p_\infty=D,
\]
and
\[
y-v\ge b-\delta-d.
\]
Hence
\[
a_y(v)\le \frac{D}{b-\delta-d}.
\]
By (1),
\[
\frac{D}{b-\delta-d}<\mu.
\]
Finally,
\[
\mu\le -p'(v).
\]
Combining the three inequalities proves \eqref{eq:remote-left-slope-bound} in the core.

\textbf{No sign of $p''$ in the core was used.}

\subsubsection{When $v$ lies in the tail}

For $d<v<y$, strict convexity means $p'$ is strictly increasing, so $-p'$ is strictly decreasing. Therefore its average on $[v,y]$ is strictly below its value at the left endpoint:
\[
\frac{1}{y-v}\int_v^y[-p'(s)]\,ds<-p'(v).
\]
That is again \eqref{eq:remote-left-slope-bound}.

\subsection{Differentiate the secant slope}

Keeping $y$ fixed,
\[
 a_y'(v)
 =\frac{p'(v)(y-v)+p(v)-p(y)}{(y-v)^2}
 =\frac{p'(v)+a_y(v)}{y-v}.
\]
By \eqref{eq:remote-left-slope-bound}, the numerator is negative. Thus
\[
a_y'(v)<0.
\]
On the other hand,
\[
\lim_{v\uparrow y}a_y(v)=-p'(y).
\]
Because $a_y$ is strictly decreasing,
\[
a_y(v)>-p'(y)\qquad(v<y).
\]
We have proved
\begin{equation}
\boxed{
-p'(v)>\frac{p(v)-p(y)}{y-v}>-p'(y)
}.
\label{eq:remote-positive-secant}
\end{equation}
Equivalently,
\[
p'(v)<\frac{p(y)-p(v)}{y-v}<p'(y).
\]

\textbf{This is the replacement for global convexity.} It supplies the required derivative--secant ordering for every interval having a remote endpoint, even though it need not hold for intervals contained in the core.

\subsubsection{Why that inequality identifies every relevant entropy shock}

There are two separate deductions: an entropy-sign calculation, and a characteristic-speed calculation.

\subsection{First obtain the chord inequality}

Fix $v<t<y$. Since $a_y$ is decreasing,
\[
a_y(t)<a_y(v).
\]
Multiply by $y-t>0$:
\[
p(t)-p(y)<a_y(v)(y-t).
\]
Thus
\begin{equation}
p(t)<p(y)+a_y(v)(y-t).
\end{equation}
The right-hand side is the affine function joining $(v,p(v))$ and $(y,p(y))$. Indeed, its value at $y$ is $p(y)$, and its value at $v$ is
\[
p(y)+a_y(v)(y-v)=p(v).
\]
Therefore the pressure graph lies strictly below this chord.

If the other volume is larger than the remote endpoint, both volumes lie in the strictly convex tail, so the same below-chord statement holds directly.

\subsubsection{Write Rankine--Hugoniot with the correct orientation}

Let $U_-$ and $U_+$ be the left and right states of a discontinuity of speed $\sigma$. Write
\[
[g]:=g(U_+)-g(U_-).
\]
Rankine--Hugoniot gives
\[
\sigma[v]=-[u],\qquad \sigma[u]=[p].
\]
For a nontrivial discontinuity, $[v]\ne0$: if $[v]=0$, the first equation forces $[u]=0$.
Eliminating $[u]$,
\begin{equation}
\sigma^2=-\frac{[p]}{[v]},\qquad [u]=-\sigma[v].
\end{equation}

\subsubsection{Compute the entropy production}

Use
\[
[up]=\frac{u_++u_-}{2}[p]
+\frac{p(v_+)+p(v_-)}{2}[u],
\]
and
\[
[\eta]=\frac{u_++u_-}{2}[u]-\int_{v_-}^{v_+}p(s)\,ds.
\]
Therefore
\[
\begin{aligned}
[q]-\sigma[\eta]
&=\frac{u_++u_-}{2}\bigl([p]-\sigma[u]\bigr)\\
&\quad+\frac{p(v_+)+p(v_-)}{2}[u]
+\sigma\int_{v_-}^{v_+}p(s)\,ds.
\end{aligned}
\]
The first line vanishes by Rankine--Hugoniot. Substituting $[u]=-\sigma[v]$ gives
\begin{equation}
\boxed{
[q]-\sigma[\eta]
=\sigma\left[
\int_{v_-}^{v_+}p(s)\,ds
-\frac{v_+-v_-}{2}\bigl(p(v_-)+p(v_+)\bigr)
\right].
}
\end{equation}
When $v_+>v_-$, the expression in brackets is strictly negative by the below-chord property. When $v_+<v_-$, reversing the integral makes it strictly positive.

Thus, for a jump incident on a remote endpoint,
\begin{equation}
[q]-\sigma[\eta]\le0
\quad\Longleftrightarrow\quad
\boxed{\sigma(v_+-v_-)>0}.
\label{eq:remote-admissible-sign}
\end{equation}
In particular, every admissible incident jump is strictly entropy dissipative.

\subsubsection{Derive the strict Lax inequalities}

For two volumes $v<y$, equation \eqref{eq:remote-positive-secant} says
\[
-p'(v)>a_y(v)>-p'(y).
\]
Taking positive square roots,
\[
\sqrt{-p'(v)}>\sqrt{a_y(v)}>\sqrt{-p'(y)}.
\]
If the admissible shock has negative speed, \eqref{eq:remote-admissible-sign} says $v_+<v_-$. Therefore
\[
-\sqrt{-p'(v_+)}<-\sqrt{a}<-\sqrt{-p'(v_-)}.
\]
Using $\sigma=-\sqrt a$,
\[
\boxed{\lambda_1(U_+)<\sigma<\lambda_1(U_-)}.
\]
For positive speed, $v_+>v_-$, and similarly
\[
\boxed{\lambda_2(U_+)<\sigma<\lambda_2(U_-)}.
\]
This proves the complete endpoint-control lemma.

\section{$L^2$ stability, $a$-contraction results}

We will use the following Lemma, which gives an
estimate on the quantity of entropy which is dissipated along a shock.

\begin{lemma}[Lax's entropy dissipation formula]\label{LDF}
For $k=1,2$ and for any shock
\[
\bigl(U_L,S^k_{U_L}(s),\sigma^k_{U_L}(s)\bigr),
\]
\begin{equation}
q\bigl(S^k_{U_L}(s)\bigr)
-\sigma^k_{U_L}(s)\eta\bigl(S^k_{U_L}(s)\bigr)
=
q(U_L)-\sigma^k_{U_L}(s)\eta(U_L)
+
\int_{v_*}^s
\frac{d}{d\tau}\sigma^k_{U_L}(\tau)
\eta\bigl(U_L\mid S^k_{U_L}(\tau)\bigr)\,d\tau,
\end{equation}
where $v_*$ is the base volume such that $S^k_{U_L}(v_*)=U_L$.
\end{lemma}

\begin{remark}
This result can be dated back to Lax \cite{Lax1971ShockWavesEntropy}. A proof can be found in \cite{Leger2011}. 
\end{remark}

The following result tells us how the weight $a$ in $a$-contraction helps us localize around a particular state in state space.

\begin{lemma}\label{size_pi}
Fix $B>0$. Then there exists a constant $C>0$ depending on $B$ such that
the following holds:

If $U_L,U_R\in\mathcal{V}$ with $|U_L|,|U_R|\leq B$, then whenever
$\alpha,\theta\in(0,1)$ verify
\begin{equation}
\label{eq:3.6}
    \alpha < \frac{\theta^2}{C},
\end{equation}
then
\[
R_a := \left\{U \,\middle|\, \eta(U\mid U_L)
    \leq a\eta(U\mid U_R)\right\}
    \subset B_\theta(U_L)
    \qquad \text{for all } 0<a<\alpha.
\]
\end{lemma}

\begin{remark}
The set $R_a$ is compact.
\end{remark}

The proof of \Cref{size_pi} is found in the proof of Lemma 4.3 in \cite{MR3519973}.

\section{Convex integration facts}\label{sec:kim_theorem}

Let $m,n\ge 2$ be integers, and let $\Omega\subset\mathbb{R}^n$ be a bounded
domain with Lipschitz boundary. Let
\[
L\in \mathbb{M}^{m\times n}\setminus\{0\},
\]
and its corresponding linear function
\[
\mathcal{L}:\mathbb{M}^{m\times n}\to\mathbb{R}
\]
is given by
\[
\mathcal{L}(\xi)
=
L\cdot \xi
=
\sum_{\substack{1\le i\le m\\1\le j\le n}}
L_{ij}\xi_{ij}
\qquad
\forall\,\xi\in\mathbb{M}^{m\times n}.
\]

We also view $L$ as the linear map
\[
b\mapsto Lb
\]
from $\mathbb{R}^n$ into $\mathbb{R}^m$, which should be distinguished from
$\mathcal{L}$ (this is an abuse of notation). We fix any number $t\in\mathbb{R}$ and write
\[
\Sigma_t
=
\left\{
\xi\in\mathbb{M}^{m\times n}
\,\middle|\,
\mathcal{L}(\xi)=t
\right\},
\]
which is an $(mn-1)$-dimensional flat manifold in $\mathbb{R}^{m\times n}$.

We can then define \emph{in-approximations}:

\begin{definition}
Let $K \subset \Sigma_t$. A sequence $\{U_j\}_{j\in\mathbb{N}}$ of open sets in $\Sigma_t$ is called an \emph{in-approximation} of $K$ in $\Sigma_t$ if the following are satisfied:
\begin{enumerate}
    \item[(i)] $U_j$ $(j\in\mathbb{N})$ are uniformly bounded,
    \item[(ii)] $U_j \subset U_{j+1}^{rc}$ for every $j\in\mathbb{N}$, and
    \item[(iii)] $U_j \to K$ as $j\to\infty$ in the following sense: If $\xi_j\in U_j$ for all $j\in\mathbb{N}$ and
    \[
    \xi_j \to \xi
    \qquad \text{as } j\to\infty
    \]
    for some $\xi\in\Sigma_t$, then $\xi\in K$.
\end{enumerate}
\end{definition}

For completeness, we restate \cite[Theorem 1.3]{Kim2020ConvexIntegration}:

\begin{theorem}\label{kim_theorem}
Assume
\[
Lb \neq 0 \in \mathbb{R}^m
\qquad
\forall\, b \in \mathbb{R}^n \setminus \{0\}.
\] and $K \subset \Sigma_t$. Let $\{U_j\}_{j\in\mathbb{N}}$ be an in-approximation of $K$ in $\Sigma_t$, and let
\[
v \in W^{1,\infty}(\Omega;\mathbb{R}^m)
\]
be a piecewise $C^1$ map satisfying
\[
\nabla v \in U_1
\qquad \text{a.e. in } \Omega.
\]
Then, for each $\delta>0$, there exists a map
\[
u \in W^{1,\infty}(\Omega;\mathbb{R}^m)
\]
such that
\[
\begin{cases}
\nabla u \in K & \text{a.e. in } \Omega,\\
u=v & \text{on } \partial\Omega,\\
\|u-v\|_{L^\infty(\Omega)}<\delta.
\end{cases}
\]
\end{theorem}

\section{An example of an approximate $T_6$ from proof of \Cref{matlab_lemma}}\label{sec:T6-example}

Generated using \verb|export_solution_latex.m|

\begin{equation*}
\widehat P=
\begin{pmatrix}
-\dfrac{8207794446755627}{2199023255552}
& -\dfrac{4602094026561529}{35184372088832}\\[10pt]
-\dfrac{4602094026561529}{35184372088832}
& -\dfrac{5207541431105289}{281474976710656}\\[10pt]
-\dfrac{8995103167499449}{1099511627776}
& \phantom{-}\dfrac{8318704555107407}{4398046511104}
\end{pmatrix}.
\end{equation*}

\begin{equation*}
\begin{aligned}
\widehat\kappa_{1} &= \frac{6763823737350437}{562949953421312},\\
\widehat a_{1} &= \begin{pmatrix}
-\frac{6348406838565697}{36028797018963968}\\[4pt]
\frac{7983664226036559}{288230376151711744}\\[4pt]
\frac{7162198692843247}{35184372088832}
\end{pmatrix},\\
\widehat n_{1} &= \begin{pmatrix}-\frac{6178714499117911}{4503599627370496} & \frac{7770261606622485}{36028797018963968}\end{pmatrix}.
\end{aligned}
\end{equation*}

\begin{equation*}
\begin{aligned}
\widehat\kappa_{2} &= \frac{6520014986661195}{70368744177664},\\
\widehat a_{2} &= \begin{pmatrix}
-\frac{1526800909617329}{140737488355328}\\[4pt]
\frac{8282651816847125}{1125899906842624}\\[4pt]
\frac{6544241863525297}{8796093022208}
\end{pmatrix},\\
\widehat n_{2} &= \begin{pmatrix}\frac{5500683242223889}{18014398509481984} & -\frac{3730041337006263}{18014398509481984}\end{pmatrix}.
\end{aligned}
\end{equation*}

\begin{equation*}
\begin{aligned}
\widehat\kappa_{3} &= \frac{6083303638755}{137438953472},\\
\widehat a_{3} &= \begin{pmatrix}
-\frac{8906576701310669}{18014398509481984}\\[4pt]
-\frac{2770473510069383}{9007199254740992}\\[4pt]
\frac{6638703438417975}{562949953421312}
\end{pmatrix},\\
\widehat n_{3} &= \begin{pmatrix}-\frac{1146856117493235}{140737488355328} & -\frac{5707844170108083}{1125899906842624}\end{pmatrix}.
\end{aligned}
\end{equation*}

\begin{equation*}
\begin{aligned}
\widehat\kappa_{4} &= \frac{3075101860078263}{70368744177664},\\
\widehat a_{4} &= \begin{pmatrix}
-\frac{4969460337195129}{18014398509481984}\\[4pt]
-\frac{7756081206164689}{36028797018963968}\\[4pt]
\frac{1970727903418487}{35184372088832}
\end{pmatrix},\\
\widehat n_{4} &= \begin{pmatrix}\frac{6157812183636213}{562949953421312} & \frac{1201350048800457}{140737488355328}\end{pmatrix}.
\end{aligned}
\end{equation*}

\begin{equation*}
\begin{aligned}
\widehat\kappa_{5} &= \frac{8846450761394073}{2251799813685248},\\
\widehat a_{5} &= \begin{pmatrix}
-\frac{2481889280075591}{1152921504606846976}\\[4pt]
\frac{3949343957805623}{18446744073709551616}\\[4pt]
-\frac{3597532152031909}{2251799813685248}
\end{pmatrix},\\
\widehat n_{5} &= \begin{pmatrix}\frac{710037460610907}{1099511627776} & -\frac{564928938961197}{8796093022208}\end{pmatrix}.
\end{aligned}
\end{equation*}

\begin{equation*}
\begin{aligned}
\widehat\kappa_{6} &= \frac{7324349891446651}{140737488355328},\\
\widehat a_{6} &= \begin{pmatrix}
-\frac{7377531271569983}{2305843009213693952}\\[4pt]
\frac{5342950352776515}{2305843009213693952}\\[4pt]
-\frac{4738704101426039}{9007199254740992}
\end{pmatrix},\\
\widehat n_{6} &= \begin{pmatrix}-\frac{1185452820758119}{1099511627776} & \frac{3434111132151093}{4398046511104}\end{pmatrix}.
\end{aligned}
\end{equation*}

The supplied approximate configuration has negative volume coordinates. This is fixed by volume translation and the accompanying pressure reparametrization used to move the configuration into \(v>0\).

\bibliographystyle{plain}
\bibliography{references}

\begin{center} 
\includegraphics[width=.4\linewidth]{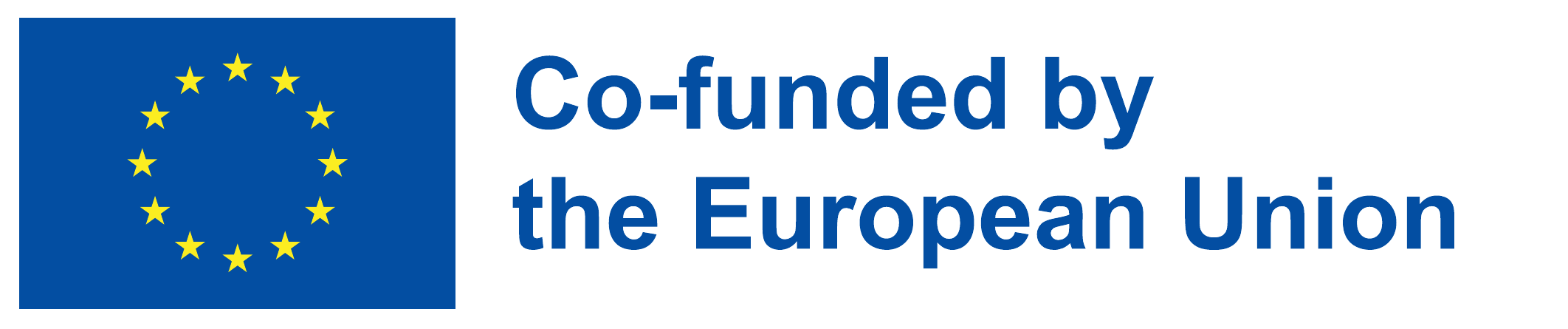}
\end{center}
\end{document}